\documentclass[10pt,a4paper]{article}
\usepackage{amsmath, amscd, amssymb, amsthm, amscd}
\newtheorem{thm}{Theorem}[subsection]
\newtheorem{lemma}[thm]{Lemma}
\newtheorem{prop}[thm]{Proposition}
\newtheorem{defn}[thm]{Definition}
\newtheorem{cor}[thm]{Corollary}
\newcommand{\uxc}{u_{x,y}}
\newcommand{\tc}{t_{y}}
\newcommand{\oxx}{\mathcal{O}_{X,x}}
\newcommand{\fc}{F_{y}}

\newcommand{\jx}{\mathcal{J}_{X}} \newcommand{\jc}{\mathcal{J}_{y}} \newcommand{\jxx}{\mathcal{J}_{x}}
\newcommand{\excone}{\mathcal{E}^{(1)}_{F_{x,y}}} \newcommand{\exctwo}{\mathcal{E}^{(2)}_{F_{x,y}}}
\newcommand{\aone}{\alpha^{(1)}}\newcommand{\atwo}{\alpha^{(2)}}
\newcommand{\phionek}{\phi^{(1)}_{k,x,y}(\uxc)}\newcommand{\phitwok}{\phi^{(2)}_{k,x,y}(\uxc)}
\newcommand{\phionej}{\phi^{(1)}_{j,x,y}(\uxc)} \newcommand{\phitwoj}{\phi^{(2)}_{j,x,y}(\uxc)}

\begin{document}
\title{Two-Dimensional Local-Global Class Field Theory in Positive Characteristic}\author{Kirsty Syder}\date{}
\maketitle
\begin{abstract}
Using the higher tame symbol and Kawada and Satake's Witt vector method, A. N. Parshin developed class field theory for higher local fields, defining reciprocity maps separately for the tamely ramified and wildly ramified cases. We extend this technique to the case of local-global fields associated to points and curves on an algebraic surface over a finite field. 
\end{abstract}

\section{Introduction}\setcounter{subsection}{1}

In the study of class field theory of algebraic curves, the local field associated to each point on the curve is used to define a ring of adeles for the curve. This ring provides the domain for a reciprocity map for the global field of functions of the curve. In this text, we extend this approach to the case of an algebraic surface over a finite field, using the reciprocity maps for higher local fields first defined by A. N. Parshin in \cite{P4}.

The study of higher local fields was initiated in the 1970s by Y. Ihara, with further work done by A. N. Parshin in the positive characteristic case and K. Kato in the general case. We recall the inductive definition: an $n$-dimensional local field $K$ is a complete discrete valuation field $F$ with ring of integers \[ \mathcal{O}_{F} : = \{ \alpha \in F : v_{F}(\alpha ) \geq 0 \} \] 
and maximal ideal 
\[ \mathfrak{m}_{F} : = \{ \alpha \in F : v_{F}(\alpha ) > 0 \} \] 
such that the residue field $\mathcal{O}_{F}/\mathfrak{m}_{F}$ is an $(n-1)$-dimensional local field.
One-dimensional local fields are the usual local fields, i.e. finite extensions of $\mathbb{Q}_{p}$ and $\mathbb{F}_{p}((t))$, for a prime $p$. 

The class field theory of higher local fields has been extensively studied, with different methods applied. Kato used cohomological methods to define the reciprocity map, see \cite[Section 5]{IHLF} for an overview or \cite{KK1}, \cite{KK2} and \cite{KK3} for full details. Fesenko provides an explicit version of the class field theory, as Neukirch did in the classical case (see \cite{N}) - see \cite[Section 10]{IHLF} for an overview or \cite{F7} and \cite{F8} for full details. We will concentrate on the work of Parshin on higher class field theory.

In his papers \cite{P4} and \cite{P5}, Parshin developed a reciprocity map for higher local fields by gluing together three separate maps for unramified, tamely ramified and wildly ramified extensions. The unramified part of the map is as usual a valuation map associated to the Frobenius element. See \cite{IHLF} section seven for a review of this theory. 

The map for tamely ramified extensions came from the higher tame symbol, which is a higher dimensional generalisation of the tame symbol, 
\[\{f,g\} = (-1)^{v(f)v(g)}\frac{f^{v(g)}}{g^{v(f)}}\]
for $f$, $g$ elements of a local field with valuation $v$. This symbol has been studied extensively, and the reciprocity laws described below proved using several different methods. In particular for a study of the tame symbol for an algebraic surface, see the work of Romo, \cite{FPR3}, \cite{FPR4}, \cite{FPR5} and \cite{FPR6}, Osipov, \cite{O}, and Osipov and Zhu, \cite{OZ}. See section 2.3 for a definition and discussion of the higher tame symbol.

The map for wildly ramified extensions is the Artin-Schreier-Witt pairing. The method of using the Witt pairing to define a reciprocity map for wildly ramified extensions of fields of positive characteristic was first developed by Kawada and Satake in their paper \cite{KS}. They proved the class field theory for local fields and function fields of positive characteristic. Parshin's method is a higher-dimensional generalisation of Kawada and Satake's method. See section 2.4 for a full definition of the Witt symbol and associated local reciprocity map.

\medskip
The structure of the paper is as follows. The first chapter, entitled `The Local Theory' will recall the situation for higher local fields. We first define higher local fields in more detail, then discuss their Milnor $K$-groups. The $K$-groups are a quotient of the $n$-fold tensor product of an $n$-dimensional local field, and are the domain of the local reciprocity map. This first section contains basic properties of all these objects. The following sections define the local Witt pairing and higher tame pairing, then finally we mention the class field theory proved by Parshin, which constructs a reciprocity map using these pairings.

Chapter three deals with the global situation. It begins by defining an adelic group associated to the algebraic surface, and its Milnor $K$-group. We then define the global Witt and higher tame pairings as a sum and product (respectively) of the local pairings, and prove these are well-defined. 

Chapters  four and five then study these pairings over two types of `semi-global' field, which will be discussed below, proving duality theorems which enable the proof of semi-global versions of Parshin's higher local class field theory.

\medskip
We now introduce the semi-global fields, and then state the main theorems of the paper. 

Our set-up is as follows: let $X$ be an algebraic surface over a finite field $k$ of size $q$, with function field $F$. Closed points of the surface will be denoted $x$, and curves on the surface by $y$.

We associate a product of higher local fields to a point $x$ lying on a curve $y$ on the surface $X$ by a series of completions and localisations of the local ring $\hat{\mathcal{O}}_{X,x}$ - see section 2.1. We can  then define a product of semi-global fields associated to a curve $y$ and a product of rings associated to a point $x$. Complete definitions can be found at \ref{fieldsdefs}, for now it is enough to think of $F_{y}$ as a complete discrete valuation field over a global field, isomorphic to $k(y)((t_{y}))$, and $F_{x}$ the ring generated by the complete local ring $\mathcal{O}_{X,x}$ and $F$, the field of functions of $X$. 

The class field theory of these objects has been studied before, primarily by Kato and Saito. See their papers \cite{KSCFT1}, \cite{KSCFT2}, and \cite{KSCFT3} for details. The approach used is similar to Kato's local class field theory. They also provide class field theory for the function field $F$ of a surface $X$. Other methods for the class field theory of arithmetic surfaces have been proposed by Wiesend and developed by Kerz and Schmidt - see \cite{MK}. Their method only considers the global case, without the use of local or local-global class field theory.

As in the classical case, most of these class field theories become very complicated when discussing the $p$-part of the reciprocity map in characteristic $p$. Kawada and Satake's Witt vector method greatly simplifies this in the one-dimensional case, and Parshin and Fesenko's methods both build on this work in the higher local case. This paper extends those methods to provide a more simple description of the $p$-part of the map in the semi-global case, while defining compatible reciprocity maps for the non-$p$-divisible parts.

We will write Gal$(F^{ab,p}/F)$ for the $p$-divisible part of the absolute abelian Galois group. We write Gal$(F^{unram}/F)$ for the part of the absolute abelian Galois group, isomorphic to $\hat{\mathbb{Z}}$, which is related to the algebraic closure of the finite field $\mathbb{F}_{q}$.

The main new results leading to the reciprocity map are duality theorems. We define certain subgroups of the $K$-groups of the adelic group of $X$ - see section 3.1 - and also define the global Witt and higher tame pairings as sums along a curve or around a point, then prove the following theorem.

\begin{thm}\label{introthm1}
Let $x \in X$ be a closed point and $y \subset X$ a curve. Then we have the isomorphisms
\[\frac{J_{y}}{\Delta(K_{2}^{top}(F_{y}))\cap J_{y}} \cong \mathrm{Hom}\left(\frac{W(F_{y})}{(\mathrm{Frob}-1)W(F_{y})}, \mathbb{Z}_{p}\right);\]
\[\frac{\mathfrak{J}_{y}}{\Delta(K_{2}^{top}(F_{y}))\cap \mathfrak{J}_{y}} \cong \mathrm{Hom}\left(\frac{F_{y}^{\times}}{(F_{y}^{\times})^{q-1}}, \mathbb{Z}/(q-1)\mathbb{Z}\right);\]
\[\frac{J_{x}}{\Delta(K_{2}^{top}(F_{x}))\cap J_{x}} \cong \mathrm{Hom}\left(\frac{W(F_{x})}{(\mathrm{Frob}-1)W(F_{x})}, \mathbb{Z}_{p}\right);\]
\[\frac{\mathfrak{J}_{x}}{\Delta(K_{2}^{top}(F_{x}))\cap \mathfrak{J}_{x}} \cong \mathrm{Hom}\left(\frac{F_{x}^{\times}}{(F_{x}^{\times})^{q-1}}, \mathbb{Z}/(q-1)\mathbb{Z}\right).\]
\end{thm}

The theorems can be found at \ref{mainthmcurves}, \ref {tamecurvethm}, \ref{mainthmpoints} and \ref{tamepointthm}.

These theorems are important for two reasons. The first is to define the reciprocity map: Witt duality and Kummer theory show that we can define a map from the groups on the left to the absolute abelian Galois group. The second is to prove the reciprocity map is injective.

Using these theorems to define the reciprocity maps $\phi_{y}$, $\phi_{x}$ for a fixed curve and point respectively, we can then prove the main theorems of the class field theory:

\begin{thm}\label{introthm}
Let $X/\mathbb{F}_{q}$ be a regular projective surface, $x \in X$ a closed point and $y \subset X$ an irreducible curve. Then the continuous maps
\[\phi_{y} : {\prod}'_{x \in y}K_{2}^{top}(\mathcal{O}_{x,y}) \to \text{Gal}(F_{y}^{ab}/F_{y}) 
\]
and
\[\phi_{x} : {\prod}'_{y \ni x}K_{2}^{top}(\mathcal{O}_{x,y}) \to \text{Gal}(F_{x}^{ab}/F_{x}) 
\]
are injective with dense image and satisfy:
\begin{enumerate}
\item{$\phi_{y}$, $\phi_{x}$ depend only on $F_{y}$, $F_{x}$ - not on the choice of model of $X$;}
\item{For any finite abelian extension $L/F_{y}$, the following sequence is exact:
\[\begin{CD}@.
\frac{\prod'_{x' \in y', \pi(x')=x}K_{2}^{top}(L_{x'})}{\Delta(K_{2}^{top}(L))\cap \prod'_{x' \in y', \pi(x')=x}K_{2}^{top}(L_{x'})} @>N>>   @. \\
@. \mathcal{J}_{y}/\Delta(K_{2}^{top}(F_{y}))\cap \mathcal{J}_{y} @>\phi_{y}>> 
@. \text{Gal}(L/F_{y})@>>> 0
\end{CD}
\] and the same sequence applies for a finite extension $L/F_{x}$.}
\item{For any finite separable extension $L/F_{y}$, the following diagrams commute:
\[\begin{CD}
\mathcal{J}_{L}/\Delta(K_{2}^{top}(L)) @>\phi_{L}>> \text{Gal}(L^{ab}/L)\\
@AAA @A V AA\\
\mathcal{J}_{y}/\Delta(K_{2}^{top}(F_{y})) @>\phi_{y}>> \text{Gal}(F_{y}^{ab}/F_{y})\\
\end{CD}\]
where  $V$ is the group transfer map, and

\[\begin{CD}
\mathcal{J}_{L}/\Delta(K_{2}^{top}(L)) @>\phi_{L}>> \text{Gal}(L^{ab}/L)\\
@V N VV @VVV\\
\mathcal{J}_{y}/\Delta(K_{2}^{top}(F_{y})) @>\phi_{y}>> \text{Gal}(F_{y}^{ab}/F_{y})\\
\end{CD}\] and the same diagrams apply for a finite extension $L/F_{x}$.
}

\end{enumerate}
\end{thm}

The theorems for $F_{y}$, $F_{x}$ can be found at \ref{cftcurves}
, \ref{cftpoints} respectively.

\medskip
The definition of the global and semi-global Witt pairings first appeared in the author's paper \cite{ME}, which proves the reciprocity laws for the Witt and higher tame pairings - an important result for the class field theory in this paper.

The proofs of the four duality theorems all follow the same basic pattern. Using basic theorems on the structure of the $K$-groups, and the definition of the adelic groups, we may restrict to a small set of generators for the adelic $K$-groups. Similarly, we find a ``nice" form for the right-hand side of the pairing - in the case of $F_{y}^{\times}/(F_{y}^{\times})^{q-1}$ and $F_{x}^{\times}/(F_{x}^{\times})^{q-1}$ we also find some generators. The case of the Witt vectors is more difficult - we can prove a useful general form for the entries in the Witt vector which allows us to show the non-degeneracy of the pairing.

The deduction of the main theorems of class field theory, as stated above at \ref{introthm} then follows from Witt duality and Kummer theory. The commutative diagrams follow from the local theory in \cite{P4} and the reciprocity laws from \cite{ME}.

The bulk of the work to prove these theorems is in proving that each section of the map is injective, and the exactness of the sequence in property 2. These follow from the duality theorems and properties of the $K$-groups.

It remains to mention one major obstacle to the proof of the duality theorems - the fact that we must consider singular points and curves which are not irreducible.

 For the case of a fixed curve $y$, we first look only at a smooth irreducible curve, and prove the above theorems for such an object. We then note that our adeles for a reducible curve are just the product of the adeles for the irreducible components - and the same applies for the $K$-groups and the semi-global fields we originally associated to the curve. So each group in theorem \ref{introthm1} separates into a product over the irreducible components $z \subset y$, and so every isomorphism from the theorem holds, as we have proven them for each $z$.  
 
 For a fixed point $x$, it is a little more complicated. We first prove the duality theorems \ref{introthm1} for a point satisfying condition $\dagger$:
 
``The surface $X$ has only normal crossings, so we can assume $k_{y}(x)=k(x)$ for all $y \ni x$ and $x$ has just two curves passing through it." 
 
The arguments for both the Witt and higher tame pairings follow fairly simply from combinatorial arguments and $K$-groups identities in this case. 
 
 We then generalise this case to the case where the point $x$ lies on more than two curves. This is much more difficult than the curves case above, as it means the ring $\mathcal{O}_{X,x}$ is a more complicated - i.e. non-regular - ring, rather than splitting as a product of rings we have already seen.
 
To prove the generalisation, we must look closely at the structure of the $K$-groups and their adelic groups and find generators for these groups. We can then calculate the pairings on these generators, and check using case $\dagger$ when the value is trivial. This enables us to prove non-degeneracy when quotienting by the diagonal elements, and complete the proof of the duality theorems \ref{introthm1} in the general case for a fixed point. The class field theory, theorem \ref{introthm}, follows from the duality theorems as explained above.
 
\medskip
\noindent\textbf{Acknowledgements}

\medskip
\noindent I am grateful to Alberto C\'amara, Matthew Morrow and Thomas Oliver for many useful conversations and guidance during my research. I am also very grateful to my supervisor Ivan Fesenko for suggesting the area of work and providing many comments and improvements. I am supported by an EPSRC grant at the University of Nottingham.

\section{The Local Theory}
\subsection{Two-Dimensional Local Fields and their Milnor $K$-groups}
Define an $n$-dimensional local field inductively as a complete discrete valuation field $F$ with ring of integers \[ \mathcal{O}_{F} : = \{ \alpha \in F : v_{F}(\alpha ) \geq 0 \} \] 
and maximal ideal 
\[ \mathfrak{m}_{F} : = \{ \alpha \in F : v_{F}(\alpha ) > 0 \} \] 
such that the residue field $\mathcal{O}_{F}/\mathfrak{m}_{F}$ is an $(n-1)$-dimensional local field.
One-dimensional local fields are the usual local fields, i.e. finite extensions of $\mathbb{Q}_{p}$ and $\mathbb{F}_{p}((t))$ for a prime $p$. 

We will discuss the class field theory of two-dimensional local fields, which have the following classification theorem.
\begin{thm}\label{classification}  
Let $F$ be a two-dimensional local field with valuation $v_{F}$.  Then $F$ is isomorphic to a field of one of the following types:
\begin{enumerate} 
\item{$\mathbb{F}_{q}((u))((t))$ for some prime power $q$ and $v_{F}\left( \sum a_{i}t^{i}\right) = \text{min}\{ i : a_{i} \neq 0\}$;}
\item{ $K((t))$, where $K$ is a finite extension of $\mathbb{Q}_{p}$ for some prime number $p$ and $v_{F}\left( \sum a_{i}t^{i}\right) = \text{min}\{ i : a_{i} \neq 0\}$;}
\item{ $K\{\{t\}\} : = \left\{ \sum a_{i}t^{i} : a_{i} \in K, \text{ inf}\{v_{K}(a_{i})\} > - \infty, v_{K}(a_{i}) \to 0 \text{ as } i \to - \infty\right\}$\\
\\
 where $K$ is a finite extension of $\mathbb{Q}_{p}$ for some prime $p$ and $v_{F}\left(\sum a_{i}t^{i}\right) = \text{inf}\{v_{K}(a_{i})\}$, or a finite extension of such a field.}
\end{enumerate} \end{thm}
\begin{proof} See \cite{IHLF} section 1.\end{proof}
We will only consider fields of type 1, the positive characteristic two-dimensional local fields.  In this case we say $u$ and $t$ are \emph{local parameters} for $F$.\\
Given the following data:
\begin{enumerate}   
\item{ A smooth projective algebraic surface $X$ over a finite field $k$;}
\item{A reduced irreducible curve $y \subset X$;}
\item{ A closed point $x \in y$;}
\end{enumerate}
we can associate a product of two-dimensional local fields $F_{x,y} = \prod_{z \in y(x)}F_{x,z}$ to the pair $(x,y)$, where $y(x)$ is the set of local irreducible branches of the curve $y$ at $x$.\\
For each $z \in y(x)$, let $t_{z} \in \mathcal{O}_{X,x}$ be a local equation for $z$ at $x$ and $u_{x,z} \in \mathcal{O}_{z,x}$ a local parameter at $x$.  Then
\[F_{x,z} : = k_{z}(x)((u_{x,z}))((t_{z}))\]
is a two-dimensional local field over the finite field $k_{z}(x)$, where $k_{z}(x)$ is the residue field of the local ring of the point $x$ on the curve $z$. To show this process is independent of the choices of $u_{x,z}$ and $t_{z}$, the field $F_{x,z}$ is constructed through a series of localisations and completions which are outlined below.  For full details, see \cite[section 3]{CHDLF}.\\
Let $\mathfrak{m} \subset \mathcal{O}_{X,x}$ be the maximal ideal associated to $x$ and $\mathfrak{p} \subset \mathfrak{m}$ a prime ideal associated to $z$.  Note that we may take $t_{z}$ to be any generator of $\mathfrak{p}$,  $u_{x,z}$ to be an other generator of $\mathfrak{m}$  and $\mathcal{O}_{X,x} $ a localisation of $ k(x)[u_{x,z}][t_{z}]$ such that its completion with respect to $\mathfrak{m}$ is $\hat{\mathcal{O}}_{X,x} \cong k(x)[[u_{x,z}, t_{z}]]$.\\
Take $\hat{\mathfrak{p}}$ to be any image of $\mathfrak{p}$ in $\hat{\mathcal{O}}_{X,x}$, i.e.
 \[\hat{\mathfrak{p}} \in \left\{ \mathfrak{q} \subset \hat{\mathcal{O}}_{X,x}: \mathfrak{q} \text{ is an ideal of } \hat{\mathcal{O}}_{X,x}, \ \ \mathfrak{q} \cap \oxx = \mathfrak{p}\right\}.\]
Localise with respect to $\hat{\mathfrak{p}}$ to get the ring $(\hat{\mathcal{O}}_{X,x})_{\hat{\mathfrak{p}}}$.  Completing with respect to the ideal $\hat{\mathfrak{p}}(\hat{\mathcal{O}}_{X,x})_{\hat{\mathfrak{p}}}$ produces the ring
\[ \widehat{(\hat{\mathcal{O}}_{X,x})_{\hat{\mathfrak{p}}}} \cong k_{z}(x)((u_{x,z}))[[t_{z}]].\]

\noindent Finally localising this ring with respect to a minimal prime ideal will produce the field $F_{x,z} \cong k_{z}(x)((u_{x,z}))((t_{z}))$. Then $F_{x,y}$ is the product of these two-dimensional local fields.

We define the topology on the multiplicative group of a two dimensional local field of positive characteristic as follows:

Take the product topology of the discrete topology on $k_{z}(x)^{\times} = (\mathcal{O}_{x,z}/\mathfrak{m}_{x,z})^{\times}$  and the discrete topologies on the groups generated by the local parameters $u_{x,z}$, $t_{z}$. For the remaining generating elements, the group of principal units, we use the topology induced from the topology on $F_{x,z}$, which we now describe.

Fix the local parameters $t_{z}$ and $u_{x,z}$, and a lifting from $\bar{F}_{x,z} \cong k_{z}(x)((u_{x,z}))$. The topology is usually defined inductively, starting from the discrete topology on $k_{z}(x)$ - but as this gives the usual topology on the local field $k_{z}(x)((u_{x,z}))$, we just discuss the induction step to $F_{x,z}$. 

An element $\alpha$ of $F_{x,z}$ is the limit of a sequence of elements $\alpha_{n}$ in $F_{x,z}$ if and only if given any series $\alpha_{n}=\sum_{i}\theta_{n,i}t_{y}^{i}$, we have $\alpha=\sum_{i}\theta_{i}t_{y}^{i}$, satisfying the following conditions. For every set $\{U_{i}: -\infty <i<\infty\}$ of neighbourhoods of zero in $\bar{F}_{x,z}$ and every $i_{0}$, for almost all $n$ the residue of $\theta_{n,i}-\theta_{i}$ is in $U_{i}$ for all $i<i_{0}$.
Now we may call a subset $U$ of $F_{x,z}$ open if and only if for every $\alpha \in U$ and every sequence $\alpha_{n}$ having $\alpha$ as a limit, all but finitely many $\alpha_{n}$ are in $U$. For further details of this definition, see \cite{F6}. 

\medskip
\noindent\emph{Milnor $K$-groups}

\medskip
\noindent In higher local class field theory, the Milnor $K$-groups play the role of the multiplicative group of the field in the one-dimensional case. We define these groups and prove some useful properties.

 \begin{defn} For a ring $R$, let \[I_{n} : = \left\{ \alpha_{1} \otimes \dots \otimes \alpha_{n} \in (R^{\times})^{\otimes n} : \alpha_{i} + \alpha_{j} = 1, \text{ some } 1 \leq i, j \leq n \right\}.\]
Define the $n^{th}$ Milnor $K$-group of $R$ as:
\[ K_{n}(R) : = (R^{\times})^{\otimes n}/ I_{n}.\]
\end{defn}
\noindent For a higher local field $L$, denote elements of $K_{n}(L)$ by $\{\alpha_{1}, \dots , \alpha_{n}\}$ and define the symbol map 
$ \phi : (L^{\times})^{n} \to K_{n}(L)$ by $(\alpha_{1}, \dots , \alpha_{n}) \mapsto \{\alpha_{1}, \dots , \alpha_{n}\}$.  The group law on $K_{n}(L)$ will be written multiplicatively. \\
We will also use the Milnor $K$-groups $K_{2}(\mathcal{O}_{L})$ and \[K_{2}(\mathcal{O}_{L}, \mathfrak{p}_{L}) : = \text{ker}\left( K_{2}(\mathcal{O}_{L}) \to K_{2}(\mathcal{O}_{L}/\mathfrak{p}_{L})\right).\]

\noindent For a product of fields $F_{x,y}$ at a singular point $x$, we define the group $K_{n}(F_{x,y})$ to be the product of the $K$-groups $K_{n}(F_{x,z})$ at the branches $z \in y(x)$.\\

We now mention some basic properties of these groups. For $n \geq 1$ and a discrete valuation field $L$ with residue field $\bar{L}$, there is the boundary homomorphism\[ \delta : K_{n}(L) \to K_{n-1}(\bar{L}).\]
For $n=2$ this can be explicitly calculated as \[\delta(\{\alpha, \beta\}) = (-1)^{v(\alpha)v(\beta)}\overline{\alpha^{v(\beta)}\beta^{-v(\alpha)}}.\]
See \cite{FV} chapter seven section two for details, and the next section on the higher tame symbol on an algebraic surface for calculations when $n = 3$. The boundary homomorphism enables us to investigate the relationship between the Milnor $K$-groups of a discrete valuation field and those of its residue field - in particular we will use the Bass-Tate theorem:
\begin{thm}\label{blochkato} 
Fix $n \geq 1$.  Let $F = E(X)$ and $v$ run through the discrete valuations of $F$ trivial on $E$, with $\delta_{v} : K_{n}(F_{v}) \to K_{n-1}(\bar{F_{v}})$ the boundary homomorphism for each $v$.  The sequence
\[ \begin{CD} 0 @>>> K_{n}(E) @>>> K_{n}(F) @>\oplus \delta_{V} >> \oplus_{v}K_{n-1}(\bar{F_{v}}) @>>> 0\end{CD}\]
is exact and splits. \end{thm}
\begin{proof} See \cite{FV}, 7.4.2. \end{proof}

Next, for $L/M$ a field extension of prime degree, we wish to define a map $N : K_{2}(L) \to K_{2}(M)$ to be the analogue of the norm map.  Following \cite[9.3]{FV}, $K_{2}(L)$ is generated by symbols $\{\alpha, \beta\}$ with $\alpha \in L$, $\beta \in M$.  So for $\gamma$ a symbol purely of this form, we can take $N(\gamma) = N(\{ \alpha, \beta\}) = \{ N_{L/M}(\alpha), \beta\}$ - where $N_{L/M}$ is the usual norm map $L \to M$ - and extend linearly.  This is independent of the choice of representative for $\gamma$. 

\begin{defn} $N: K_{2}(L) \to K_{2}(M)$ is called the norm map, or the transfer map. \end{defn}

We finally define a quotient group of the Milnor $K$-groups, which allows us to describe an \emph{injective} reciprocity map for higher local fields. For full details on the following definition, see \cite{F6}. Endow $K_{n}(F_{x,z})$ with the strongest topology such that negation and the symbol map $(F_{x,z}^{\times})^{n} \to K_{n}(F_{x,z})$ are sequentially continuous. 
\begin{defn}
Define the $n^{th}$ topological Milnor K-group, $K_{n}^{top}(F_{x,z})$, as the quotient of $K_{n}(F_{x,z})$ by the intersection of all its neighbourhoods of zero. \end{defn}
\noindent Fesenko proves that \[K_{n}^{top}(F_{x,z}) = K_{n}(F_{x,z})/\cap_{l\geq 1}lK_{n}(F_{x,z})\] in \cite{F6}.

As discussed in \cite[6]{IHLF}, the convergent sequences in the topological $K$-groups are the same as in the Milnor $K$-groups, and so a series converges in $K_{2}^{top}(F)$ if and only if its terms converge to zero.

The structure of the topological $K$-groups of a two dimensional local field can be described as follows.
\begin{thm}\label{topkgpstructure}   Let $F$ be a two-dimensional local field of positive characteristic, $u$, $t$ a system of parameters and $\alpha \in K_{2}^{top}(F)$.  Then $\alpha$ is a convergent product of symbols of the form:
\begin{enumerate}
\item{$\{u, t\}$;}
\item{$\{a, u\}$, $a \in \mathbb{F}_{q}^{\times}$;}
\item{ $\{a, t\}$, $a \in \mathbb{F}_{q}^{\times}$;}
\item{\[\prod_{j \geq N_{2}} \prod_{i \geq N_{1}(j)} \{1 + a_{i,j} u^{i} t^{j}, u\},\] $N_{2} \geq 0$, $N_{1} \geq 0$ if $N_{2} = 0$, $p \nmid j$, $a_{i,j}$ in a fixed basis of $\mathbb{F}_{q}/\mathbb{F}_{p}$.}
\item{ \[\prod_{j \geq N_{2}} \prod_{i \geq N_{1}(j)} \{ 1 + a_{i,j} u^{i} t^{j}, t \},\] $N_{2} \geq 0$, $N_{1} \geq 0 $ if $ N_{2} = 0$, $p \nmid i,j$, $a_{i,j}$ in a fixed basis of $\mathbb{F}_{q}/\mathbb{F}_{p}$.}
\end{enumerate}
In fact, these elements form a topological basis for $K_{2}^{top}(F)$.\end{thm}
\begin{proof} See \cite{P4}, section 2 proposition 1. \end{proof}

The boundary map $\delta,$ and the norm map $ N: K_{2}^{top}(L) \to K_{2}^{top}(F)$ when restricted to the topological $K$-groups are well-defined, which comes from the fact that $K_{2}^{top}(L) = K_{2}(L)/\cap_{l \geq 1}K_{2}(L)$ - see \cite[4.8]{F6} for details.

\subsection{Witt Vectors and Duality}
For a field $\mathbb{F}$ of positive characteristic, let $W_{m}(\mathbb{F})$ denote the Witt vectors of length $m$ with entries in $\mathbb{F}$ and \[W(\mathbb{F}) = \varprojlim W_{m}(\mathbb{F})\] the Witt ring of $\mathbb{F}$ - see \cite{JPS}. The projective limit is taken with respect to the maps $V: W_{m-1}(\mathbb{F}) \to W_{m}(\mathbb{F})$ where $V(w_{0}, \dots, w_{m-2}) = (0, w_{0}, \dots,w_{m-2})$.\\

We recall the definition of the continuous differential forms. For a pair $x \in y$, let $\mathfrak{m}_{x,y}$ be the maximal ideal of $\mathcal{O}_{x,y}$, generated by $t_{y}$ and $u_{x,y}$. Let $\phi: \hat{\mathcal{O}}_{X,x} \to \bar{F}_{x,y}$ be the quotient map for the ideal $t_{y}\hat{\mathcal{O}}_{X,x}$. Define the subgroups $P_{i}$ and $T_{j}$ in $\omega_{F_{x,y}/\mathbb{F}_{q}}$ to be generated by elements $\phi^{-1}(\mathfrak{m}_{x,y})^{i}d\hat{\mathcal{O}}_{X,x}$ and $\mathfrak{m}_{x,y}^{j}d\mathcal{O}_{x,y}$ respectively. Then define

\[\Omega^{1,cts}_{F_{x,y}/\mathbb{F}_{q}} : = \Omega^{1}_{F_{x,y}/\mathbb{F}_{q}}/\left(F_{x,y} . \cap_{i,j \geq 0} (P_{i}+T_{j})\right),\]
and \[\Omega^{2, cts}_{F_{x,y}/\mathbb{F}_{q}} : = \Omega^{1, cts}_{F_{x,y}/\mathbb{F}_{q}} \wedge \Omega^{1, cts}_{F_{x,y}/\mathbb{F}_{q}}.\]
\\

\noindent Next we recall the definition of the residue homomorphism. 
\begin{defn}
Let $F_{x,z}$ a two-dimensional local field of positive characteristic, and fix an isomorphism $F_{x,z} \cong k_{z}(x)((t_{1}))((t_{2}))$, where $k_{z}(x)$ has size $q$. Define the residue homomorphism
\[\mathrm{res}_{F_{x,z}}: \Omega^{2, cts}_{F_{x,z}/k_{z}(x)} \to \mathbb{F}_{q}\]
by $\mathrm{res}_{F_{x,z}}(\omega) = \mathrm{Tr}_{k_{z}(x)/\mathbb{F}_{q}}a_{-1,-1}$ where
\[\omega = \sum a_{a_{1},a_{2}}t_{1}^{a_{1}}t_{2}^{a_{2}}dt_{1}\wedge dt_{2}.\]
\end{defn}
The residue map is independent of the choice of local parameters $t_{1}$ and $t_{2}$, see \cite{P1} section one.\\
Now let $A$ be the fraction field of the ring of Witt vectors of $\mathbb{F}_{q}$ and $L = A((t_{1}))((t_{2}))$. This lift to characteristic zero is necessary to define the following auxiliary co-ordinates and polynomials, but notice that in the end the formulae will be `denominator free', so the reduction back down to positive characteristic is well-defined.\\
Let $x=(x_{0}, x_{1}, \dots ) \in L$, and for each $m \in \mathbb{Z}$ introduce the auxiliary co-ordinates
\[x(m) = x_{0}^{p^{m}} + px_{1}^{p^{m-1}} + \dots + p^{m}x_{m}\]
and the polynomials $P_{m}(X_{0}, X_{1}, \dots , X_{m}) \in \mathbb{Z}[p^{-1}][X_{0}][X_{1}]\dots [X_{m}]$ such that $P_{m}(x(0), x(1), \dots, x(m)) = x_{m}$.
\begin{defn}
Let $f_{1}$, $f_{2} \in F_{x,z}^{\times}$, $g \in W(F_{x,z})$ and $\bar{g} \in W(L)$ an element such that $\bar{g}$ mod $p$ $= g$. Define the Witt pairing by
\[ (f_{1}, f_{2}|g]_{x,z} = (\mathrm{Tr}_{\mathbb{F}_{q}/\mathbb{F}_{p}}w_{i})_{i \geq 0} \in W(\mathbb{F}_{p})\]
where for each $i \in \mathbb{Z}$, 
\[w_{i} = P_{i}\left(\mathrm{res}_{L}\left(\bar{g}(0)\frac{df_{1}}{f_{1}}\wedge \frac{df_{2}}{f_{2}}\right), \dots, \mathrm{res}_{L}\left(\bar{g}(i)\frac{df_{1}}{f_{1}}\wedge \frac{df_{2}}{f_{2}}\right)\right) \text{ mod } p \] where the $\bar{g}(j)$ are the auxiliary co-ordinates for the Witt vector $\bar{g}$. Then for a curve $y$ with branches $z$, define
\[( \ , \ | \ ]_{x,y} = \sum_{z \in y(x)}( \ , \ | \ ]_{x,z}.\]
\end{defn}

\begin{prop}\label{wittproperties}
The Witt pairing satisfies the following properties:
\begin{enumerate}
\item{$(f_{1}.f_{1}', f_{2}|g]_{x,y} = (f_{1},f_{2}|g]_{x,y} + (f_{1}',f_{2}|g]_{x,y}$ and $(f_{1},f_{2}.f_{2}'|g]_{x,y} = (f_{1},f_{2}|g]_{x,y} + (f_{1},f_{2}'|g]_{x,y}$;}
\item{$(f_{1},f_{2}|g+h]_{x,y} = (f_{1},f_{2}|g]_{x,y} + (f_{1},f_{2}|h]_{x,y}$;}
\item{$(f_{1}, 1-f_{1}|g]_{x,y} = 0$;}
\item{$(f_{1},f_{2}|g]_{x,y} = (w_{0}, w_{1}, \dots) \implies (f_{1},f_{2}|g^{p}]_{x,y} = (w_{0}^{p}, w_{1}^{p}, \dots);$}
\item{$(f_{1},f_{2}|g]_{x,y}$ is continuous in each argument;}
\item{$(f_{1},f_{2}| g_{0},  \dots, g_{m-1}]_{x,y} = (w_{0}, \dots, w_{m-1}) \implies (f_{1}, f_{2}|g_{0}, \dots , g_{m-2}]_{x,y} = (w_{0}, \dots, w_{m-2});$}
\item{$(f_{1}, f_{2}| 0, g_{1}, \dots , g_{m-1}]_{x,y} = (0, (f_{1}, f_{2}| g_{1}, \dots, g_{m-1}]_{x,y})$.}
\end{enumerate}\end{prop}
\begin{proof} In \cite{P4}, 3.3.6, Parshin proves this for a single higher local field. We will prove it here for the case where $x$ is a singular point of $y$ and so we must sum the pairings over each branch of $y$ at $x$. \\
Property 3 follows straight away, and properties 1 and 2 follow from the fact that trace distributes over addition.\\
Property 4 is true as
\[(f_{1},f_{2}|g^{p}]_{x,y} = \sum_{z \in y(x)} (f_{1},f_{2}|g^{p}]_{x,z} = \sum_{z \in y(x)} (w_{0,x,z}^{p},w_{1,x,z}^{p}, \dots) \]\[= \left( \sum_{z \in y(x)} w_{0,x,z}^{p}, \dots\right) = \left( \left(\sum_{z \in y(x)}w_{0,x,z}\right)^{p}, \dots \right) = (w_{0}^{p}, w_{1}^{p}, \dots)\]
where equality holds as the sum of Witt vectors is given by polynomials in their coefficients, and when taking powers of $p$ we just raise each coefficient to the power $p$.\\
Property 5 follows from the continuity of trace and addition. 7 is true because when summing Witt vectors, the $n^{th}$ term depends linearly only on the $0^{th}, \dots (n-1)^{th}$ terms of the vectors being summed: so if the $0^{th}$ term is $0$ for all $z \in y(x)$ then it will be in the sum also.\\
Finally, property 6 follows straight from \cite{P4}, and the fact that Witt vector summation depends only on lower terms as mentioned above.

\end{proof}

Properties one and three show that the Witt symbol is a symbol on $K_{2}(F_{x,y}) \times F_{x,y}^{\times}$.\\

In his extension of Kawada and Satake's local theory, Parshin proves the following proposition.

\begin{prop}\label{} For an n-dimensional local field $L$ of characteristic $p$, the symbol $( \ | \ ]_{L}$ defines a non-degenerate pairing
\[( \ | \ ]_{L} : K_{n}^{top}(L)/p^{m}K_{n}^{top}(L) \times W_{m}(L)/(\mathrm{Frob} - 1)W_{m}(L) \to W_{m}(\mathbb{F}_{q})\]
where $\mathrm{Frob}$ is the Frobenius map.\end{prop}
\begin{proof} See \cite[3.3.7]{P4}. \end{proof}
Let $\mathfrak{W}(L) = \varprojlim W_{m}(L)/($Frob$ - 1)W_{m}(L)$ be the projective limit with respect to the mappings $V: (y_{0}, \dots, y_{m-1}) \mapsto (0, y_{0}, \dots, y_{m-1})$.  Then  following Kawada and Satake's argument from \cite[Chapter 2]{KS} gives the pairing
\[K_{n}^{top}(L) \times \mathfrak{W}(L) \to \mathbb{Q}/\mathbb{Z}\]
which is non-degenerate in the second argument.  The kernel with respect to the first argument is $K_{n}^{top}(L)_{tors}$, see \cite[3.3]{P4}.
\\
This section is concluded with a lemma describing some properties of the residue map.
\begin{lemma}\label{residueproperties}  Let $F_{x,z}$ be a two-dimensional local field of positive characteristic over $\mathbb{F}_{q}$, and $t_{z}$ a generator of the maximal ideal of $\mathcal{O}_{F_{x,z}}$.  The residue map $\mathrm{res}_{x,z}$ satisfies:
\begin{enumerate}
\item{ $\mathrm{res}_{x,z}(\omega) = 0 \text{ for all }\omega \in \Omega^{2, cts}_{\mathcal{O}_{F_{x,z}}/\mathbb{F}_{q}}$.}
\item{$\mathrm{res}_{x,z}\left(\frac{dx}{x} \wedge \frac{dt_{z}}{t_{z}}\right) = \mathrm{res}_{\overline{F}_{x,z}}\left( \frac{d\bar{x}}{\bar{x}}\right)$ for all $ x \in \mathcal{O}^{\times}_{F_{x,z}}$.}\end{enumerate}\end{lemma}
\begin{proof}

\begin{enumerate}
\item{ Fix an isomorphism $F_{x,z} \cong \mathbb{F}_{q}((t_{1}))((t_{2}))$ and let $f \in \mathcal{O}_{F_{x,z}}$. Similarly to lemma 2.8 in \cite{Mor3}, we may write $f = \sum_{i,j=0}^{n}a_{i,j}t_{1}^{i}t_{2}^{j} + gt_{1}^{n+1}t_{2}^{n+1}$ for any integer $n$ and some $g \in \mathcal{O}_{F_{x,z}}$. Applying the universal derivation $d: \mathcal{O}_{F_{x,z}}\to \Omega_{\mathcal{O}_{F_{x,z}}/\mathbb{F}_{q}}^{1}
$, we have
\[ df = \sum_{i,j=0}^{n} a_{i,j}(it_{1}^{i-1}t_{2}^{j}dt_{1}+jt_{1}^{i}t_{2}^{j-1}dt_{2})\] 
\[ \ \ \ \ \ \ \ \ \ \ \ \ \ \ \ \ \ \ \ \ \ \ \ \ \ \ \ \ \ \ + g(n+1)(t_{1}^{n}t_{2}^{n+1}dt_{1}+ t_{1}^{n+1}t_{2}^{n}dt_{2}) + t_{1}^{n+1}t_{2}^{n+1}dg.\]
Hence $df - \left(\frac{df}{dt_{1}}dt_{1} + \frac{df}{dt_{2}}dt_{2}\right) \in \cap_{n=1}^{\infty} t_{1}^{n}t_{2}^{n}\Omega_{\mathcal{O}_{F_{x,z}}/\mathbb{F}_{q}}^{1}$. So taking the separated quotient, $\Omega_{\mathcal{O}_{F_{x,z}}/\mathbb{F}_{q}}^{1}$ is generated by $dt_{1}$ and $dt_{2}$. Then $\Omega_{\mathcal{O}_{F_{x,z}}/\mathbb{F}_{q}}^{2} = \Lambda^{2}\Omega_{\mathcal{O}_{F_{x,z}}/\mathbb{F}_{q}}^{1, cts}$ is generated over $\mathcal{O}_{F_{x,z}}$ by $dt_{1}\wedge dt_{2}$, as all other types of terms in the exterior product are zero.\\
Hence we can restrict to the case $\omega = adt_{1}\wedge dt_{2}$ where $a \in \mathcal{O}_{F_{x,z}}$ and $t_{1}$ and $t_{2}$ are the local parameters of $F_{x,z}$. Decomposing $a$ as a series
\[a = \sum_{i \geq I}\sum_{j \geq 0}a_{i,j}t_{1}^{i}t_{2}^{j}\]
gives the result.}
\item{  First let $x = 1+at$, some $a \in \mathcal{O}_{K}$. Then
\[\frac{dx}{x} \wedge \frac{dt}{t} = x^{-1}da\wedge dt \in \Omega^{2, cts}_{\mathcal{O}_{K}/\mathbb{F}_{q}}\]
and so its residue is zero - but res$_{\bar{F}_{x,z}}(d\bar{x}/\bar{x}) = 0$ also, so we are done in this case. \\
The symbol $dt/t$ is additive with respect to multiplication by $t$, so we can now restrict to the case $\bar{x} \in \bar{F}^{\times}$, $x = \bar{x} + bt$ with $b \in \mathcal{O}_{K}$. Then
\[\text{res}_{K}\left(\frac{dx}{x}\wedge\frac{dt}{t}\right) = \text{res}_{K}\left( \frac{d(\bar{x}+bt)}{\bar{x}+bt}\wedge\frac{dt}{t}\right) = \text{res}_{\bar{K}}\left(\frac{d\bar{x}}{\bar{x}}\right)\]
by expanding $(\bar{x} + bt)^{-1}$.}\end{enumerate}\end{proof}

\subsection{The Higher Tame Symbol on an Algebraic Surface}
As before, let $X$ be an algebraic surface over $k$ and $x \in y \subset X$ a point on a curve contained in $X$. The higher tame symbol takes values in $k_{z}(x)$. First let $x$ be a smooth point of $y$. If $f$, $g$ and $h$ are elements of $F_{x,y}$, then the higher tame symbol is expressed as
\[ (f,g,h)_{x,y} = (-1)^{\alpha_{x,y}}\left( \frac{f^{v_{y}(g)\bar{v}_{x}(h) - v_{y}(h)\bar{v}_{x}(g)}}{g^{v_{y}(f)\bar{v}_{x}(h) - v_{y}(h)\bar{v}_{x}(f)}}h^{v_{y}(f)\bar{v}_{x}(g) - v_{y}(g)\bar{v}_{x}(f)} \right) \text{ mod } \mathfrak{m}_{x,y}\]
where:\\
\[\alpha_{x,y} = v_{y}(f)v_{y}(g)\bar{v}_{x}(h)+v_{y}(f)v_{y}(h)\bar{v}_{x}(g) + v_{y}(g)v_{y}(h)\bar{v}_{x}(f) +\]\
 \[v_{y}(f)\bar{v}_{x}(g)\bar{v}_{x}(h) + v_{y}(g)\bar{v}_{x}(f)\bar{v}_{x}(h) + v_{y}(h)\bar{v}_{x}(f)\bar{v}_{x}(g);\]
$v_{y}$ is the surjective discrete valuation induced by $y$ and $\bar{v}_{x}$ is the function
\[\bar{v}_{x}: F_{x,y}^{\times} \to \mathbb{Z}\]
defined by $\bar{v}_{x}(\beta) = v_{x,y}(p(t_{y}^{-v_{y}(\beta)}\beta))$, where $p$ is the projection map from $\mathcal{O}_{x,y}$ to $\bar{F}_{x,y}$ and $v_{x,y}$ is the discrete valuation on the local field $\bar{F}_{x,y}$. Finally, $\mathfrak{m}_{x,y}$ is the maximal ideal of $\mathcal{O}_{x,y}$.\\

Parshin introduced this symbol without the sign $(-1)^{\alpha_{x,y}}$ - this was first defined by Fesenko and Vostokov in their paper \cite{FV2}. They gave a simpler definition of the symbol using a two-dimensional discrete valuation. Let $\textbf{v}:= (\bar{v}_{x}, v_{y})= (v_{1}, v_{2})$. Then the symbol $(f_{1},f_{2},f_{3})_{x,y}$ is equal to the $(q-1)^{th}$ root of unity in $\mathbb{F}_{q}^{\times}$ which is equal to the residue of 
\[f_{1}^{b_{1}}f_{2}^{b_{2}}f_{3}^{b_{3}}(-1)^{b}\]
in $\mathbb{F}_{q}$, where \[b= \sum_{s,i<j}v_{s}(b_{i})v_{s}(b_{j})b_{i,j}^{s},\] 
$b_{j}$ is $(-1)^{j-1}$ multiplied by the determinant of the matrix $(v_{i}(f_{j}))$ with the $j^{th}$ column removed and $b_{i,j}^{s}$ is the determinant of the matrix with the $i^{th}$ and $j^{th}$ columns and $s^{th}$ row removed.\\

Notice the relation to the boundary homomorphism of $K$-theory - for $L$ an $n$-dimensional local field with first residue field $\bar{L}$, there is a map
\[ \delta : K_{i}(L) \to K_{i-1}(\bar{L}).\]
See \cite{FV}, chapter seven for details of this homomorphism. 

If $x$ is not a smooth point of the curve $y$, we can define the higher tame symbol for each local branch $z \in y(x)$ and then let $( \ ,\ ,\ )_{x,y} = \prod_{z \in y(x)} N_{k_{z}(x)/\mathbb{F}_{q}}( \ ,\ , \ )_{x,z}$.\\

In \cite{P4}, Parshin proved the following analogue of Kummer theory, related to ramified extensions of higher local fields of degrees prime to the characteristic.

\begin{prop}
Let $L$ be a local field of dimension 2 and $l$ an integer dividing $q-1$. The higher tame symbol defines a continuous and non-degenerate pairing 
\[( \ , \ , )_{F} : K_{2}^{top}(L)/lK_{2}^{top}(L) \times L^{\times}/(L^{\times})^{l} \to \mathbb{Z}/l\mathbb{Z}.\]
\end{prop}

\subsection{Higher Local Class Field Theory}
This section will state the class field theory for a two-dimensional local field of characteristic $p$, using Parshin's methods in \cite{P4}.  Let $L \cong \mathbb{F}_{q}((u))((t))$ be a two-dimensional local field and $L^{ab}$ the maximal abelian extension of $L$.
\begin{thm} \label{localcft}
There exists a canonical reciprocity map \[\phi_{L}:K_{2}^{top}(L) \to \text{Gal}(L^{ab}/L)\]
such that:
\begin{enumerate}
\item{ ker$(\phi_{L})$ is trivial and im$(\phi_{L})$ is dense in Gal$(L^{ab}/L)$.}
\item{ For $M/L$ an abelian extension, the sequence
\[\begin{CD}
K_{2}^{top}(M) @>N>> K_{2}^{top}(L) @>\phi_{L}>> \text{Gal}(M/L) @>>> 1 \end{CD}\]
is exact.}
\item{ For $M/L$ a finite separable extension, there are the following commutative diagrams:
\[\begin{CD}
K_{2}^{top}(M) @>\phi_{M}>> \text{Gal}(M^{ab}/M) \\
@AAA       @AVAA\\
K_{2}^{top}(L) @>\phi_{L}>> \text{Gal}(L^{ab}/L)\\
\end{CD}\]
\[\begin{CD}
K_{2}^{top}(M) @>\phi_{M}>> \text{Gal}(M^{ab}/M)\\
@VNVV     @VVV \\
K_{2}^{top}(L) @>\phi_{L}>> \text{Gal}(L^{ab}/L)\\ \end{CD}\]
where $V$ is the group transfer map.}
\item{ The diagram
\[\begin{CD}
K_{2}^{top}(L) @>\phi_{L}>> \text{Gal}(L^{ab}/L) \\
@V\delta VV    @VVV \\
K_{1}^{top}(\bar{L}) @>\phi_{\bar{L}}>> \text{Gal}(\bar{L}^{ab}/\bar{L})\\ \end{CD}\]
is commutative.}
\end{enumerate}\end{thm}
\begin{proof}
Parshin defines the map as the pasting together of three separate maps, for unramified, tamely ramified and wildly ramified extensions. We will describe this map, for full proofs of compatibility and the commutative diagrams, see \cite{P4} and \cite{P5}.

Let $\text{Frob}$ be the canonical generator of the maximal unramified extension of $L$. Define the unramified map by
\[ \phi_{L, un}(\alpha, \beta) = \text{Frob}^{v_{\bar{L}}(\delta(\alpha, \beta)_{L})}.\]
The isomorphism of our analogue of Kummer theory, and then the usual Kummer isomorphism show that
\[ K_{2}^{top}(L)/lK_{2}^{top}(L) \cong \text{Hom}(L^{\times}/(L^{\times})^{l}, \mathbb{Z}/l\mathbb{Z}) \]
\[\ \ \ \ \ \ \ \ \ \ \ \ \ \ \ \ \ \ \ \ \ \ \ \ \ \ \ \ \ \ \ \ \ \ \ \ \ \ \ \ \ \ \ \ \ \cong \text{Gal}(L^{ab}/L)/(\text{Gal}(L^{unram}/L)\text{Gal}(L^{ab,p}/L))\]
yielding the tame part of the map.

The isomorphism of Artin-Schreier theory, \cite[4.3]{N}, shows \[\text{Gal}(L^{ab,p}/L) \cong \text{Hom}(L/(\text{Frob} - 1)L, \mathbb{Q}/\mathbb{Z}).\]
Together with the non-degenerate pairing \ref{localwittduality}, this isomorphism shows that Gal$(L^{ab,p}/L)$ is dual to $\mathfrak{W}(L)$.  Then Witt duality yields the map
\[ \phi_{L}: K_{2}^{top}(L) \to \text{Gal}(L^{ab,p}/L).\]
The exactness of the sequence in $ii$ is proved in \cite{P5}, and $iii$ is proved in the same way.  For property $iv$, see \cite{P4}, section four, theorem one.
\end{proof}
\emph{Remark 1} 
The proof of Parshin's local class field theory is unchanged for an $n$-dimensional local field, where $n > 2$.

\medskip
\emph{Remark 2}
These theorems can also be proved using Fesenko's explicit class field theory, which defines the reciprocity map using similar methods to Neukirch's method for the one-dimensional case - see \cite{F7}, \cite{F8} and \cite{N}.

\section{The Global Theory}
This section will define the Witt pairing, higher tame pairing and various groups associated to an arithmetic surface $X$.  Section 3.1 will define these groups, their associated adelic objects, and discuss their structure. 3.2 and 3.3 will define the global versions of the pairings. 
\subsection{The Adeles and their $K$-groups}\label{global1} 

Let $X$ be a smooth projective algebraic surface over a finite field $k$. We define several fields and rings related to $X$.
\begin{defn}\label{fieldsdefs} \begin{enumerate}
\item{Let $F = k(X)$ be the function field of $X$.  $F$ is a function field in two variables over $k$.}
\item{ For an irreducible component $y_{i}$ of a  curve $y \subset X$, let  $\hat{\mathcal{O}}_{X,y_{i}}$ be the completion of the local ring at $y_{i}$ and $F_{y_{i}}$ its field of fractions.  $F_{y_{i}}$ has the structure of a complete discrete valuation field with residue field a function field in one variable over a finite extension of $k$ - i.e. a global field of positive characteristic. Let $F_{y} = \prod_{y_{i} \subset y} F_{y_{i}}$.}
\item{For a closed point $x \in X$, define the ring $F_{x}$ to be the ring generated by $\hat{\mathcal{O}}_{X,x}$ and $F$. This is a subring of Frac$(\hat{\mathcal{O}}_{X,x})$ where each function in the ring will have only globally defined poles.}
\item{For an irreducible component $y_{i}$ of a curve $y \subset X$, $k_{y_{i}}(x)$ is the finite field obtained by quotienting $\hat{\mathcal{O}}_{y_{i},x}$ by the ideal defined by $y_{i}$ and $x$. It is a finite extension of the residue field at $x$, $k(x)$. Let $k_{y}(x) = \prod_{y_{i} \subset y}k_{y_{i}}(x)$.}
\item{For a singular curve $y$, the local parameter $t_{y}$ is the element of the product of fields $F_{y_{i}}$ with a local parameter $t_{y_{i}}$ in each entry. The ring $\mathcal{O}_{x,y}$ is the product $\prod_{z \in y(x)}\mathcal{O}_{x,z}$.}
\end{enumerate}\end{defn}
\medskip
\noindent We have the inclusions
\[\begin{CD}
F_{x,z} @<<< F_{z} \\
@AAA    @AAA \\
F_{x} @<<< F. \\ \end{CD}\]

\noindent Define the topological $K$-groups of these global objects in the same way as for the local fields discussed in section 2.1.

Next we define the geometric adeles associated to $X$ and subspaces associated to a curve $y$ and a point $x$.  The following definitions appeared originally in \cite{F3}, where the characteristic zero and mixed characteristic cases are also considered.
\begin{defn}  For a curve $y \subset X$ and $r \in \mathbb{Z}$, define the adelic object $\mathbb{A}^{r}_{y}$ by:
 \[\mathbb{A}^{r}_{y} := \bigg\{ \left( \sum_{i \geq r} a_{i,x}t_{y}^{i}\right)_{x \in y} = \sum_{i \geq r} a_{i}t_{y}^{i} : a_{i} = (a_{i, x})_{x \in y} = (a_{i,x,z})_{x \in z, z \in y(x)} \]
 \[ \ \ \ \ \ \ \ \ \ \ \in \prod_{x \in y} \mathcal{O}_{x,y} \text{ is the lift of an adele } \bar{a}_{i} \in \mathbb{A}_{k( y)}\bigg\}.\] 
\\
Define also \[ \mathbb{A}_{y} = \cup_{r \in \mathbb{Z}} \mathbb{A}^{r}_{y} = \mathbb{A}_{y}^{1}[t_{y}^{-1}].\] \end{defn}

So we have defined a ``higher adelic object" associated to each curve on the surface $X$. Notice we use the adeles of the underlying (one-dimensional) global field associated to the curve and the two dimensional structure of the surface to limit which coefficients can occur, similarly to the classical definition of adeles. We now define the geometric adeles associated to the surface, using the above definition.

\begin{defn}\label{restricteddef} The geometric adeles associated to a surface $X$ are \[\mathbb{A}_{X} : = {\prod_{y \subset X}}'\mathbb{A}_{y}\]
where the restricted product is taken with respect to the rings $\mathcal{O}_{x,y}$ and $\mathbb{A}^{r}_{y}$, i.e. $\mathbb{A}_{X}$ is the set of $(a_{x,y})_{x \in y} = (a_{x,z})_{x \in z, z \in y(x)}$ such that :
\begin{enumerate} 
\item{ $y$ runs through curves on the surface $X$;}
\item{ $(a_{x,y})_{x \in y} \in \mathbb{A}_{y}$ for all $y \subset X$;}
\item{ for all but finitely many $y$, $a_{x, y} \in \mathcal{O}_{x,y}$ for all $x \in y$;}
\item{ $\exists r \in \mathbb{Z}$ such that $(a_{x,y})_{x \in y} \in \mathbb{A}^{r}_{y}$ for all $y \subset X$.} \end{enumerate}\end{defn}

\noindent For more properties of the geometric adeles, see Fesenko's paper \cite{F9}.

\medskip

We also define \[ \mathbb{B}_{X} : = \prod_{y \subset X} \Delta(F_{y}) \cap \mathbb{A}_{X};\]
\\ \[\mathbb{C}_{X} := \prod_{x \in X}\Delta(F_{x}) \cap \mathbb{A}_{X};\]
where $\Delta$ is the diagonal embedding of the rings $F_{y}$ and $F_{x}$. These two adelic rings provide us with adelic analogues of the semi-global rings $F_{y}$ and $F_{x}$. Similarly to the diagram of fields above, we get
\[ \begin{CD} \mathbb{A}_{X} @<<< \mathbb{B}_{X} \\
@AAA   @AAA \\
\mathbb{C}_{X}  @<<< F \end{CD} \]
where $F$ injects into the adeles via the diagonal map as in the one-dimensional case.\\
As in the local case, Milnor $K$-groups will replace the multiplicative group in higher global class field theory.  We define:
\[K_{2}(\mathbb{A}_{X}) : = (\mathbb{A}_{X}^{\times})^{\otimes 2} / < \alpha \otimes (1 - \alpha) \in (\mathbb{A}_{X}^{\times})^{\otimes 2}>.\]
Defining the topological $K$-groups as the quotient of the $K$-groups by the neighbourhood of the identity as before, we have $(f_{x,y})_{x \in y \subset X} \in K^{top}_{2}(\mathbb{A}_{X})$ if and only if:
\begin{enumerate}
\item{ $f_{x,y} \in K_{2}^{top}(F_{x,y})$ for all $x$ and $y$;}
\item{For all but finitely many $y$, $f_{x,y} \in K_{2}^{top}(\mathcal{O}_{x,y})$ for all $x \in y$;}
\item{ $\exists r \in \mathbb{Z}$ such that $(f_{x,y})_{x \in y} \in K_{2}^{top}(\mathbb{A}^{r}_{y})$ for all $y \subset X$.}
\end{enumerate}
Note that we write $K_{2}^{top}(R_{x,y}) = \prod_{y_{i} \in y(x)}K_{2}^{top}(R_{x,y_{i}})$ for any ring $R_{x,y_{i}}$ associated to a singular point $x$ on a curve $y$ with irreducible components $y_{i}$.

The following will define our analogue of the idele group and some important subgroups.
\begin{defn} Denote $K_{2}^{top}(\mathbb{A}_{X})$ by $\mathcal{J}_{X}$.  Some useful subgroups of $\mathcal{J}_{X}$ will be denoted:
\begin{enumerate}
\item{ $\mathcal{J}_{y}: = \left\{ (f_{x,y} \in \mathcal{J}_{X} : f_{x,y'} = 1 \ \forall \ y' \neq y \text{ curves on } X \right\};$}
\item{ $\mathcal{J}_{x} : = \left\{ (f_{x,y}) \in \mathcal{J}_{X} : f_{x',y} = 1 \ \forall \ x' \neq x \text{ points on } X \right\};$}
\item{ $\mathcal{J}_{1} $ is the intersection of $\mathcal{J}_{X}$ with the diagonal image of $\prod_{y \subset X}K_{2}^{top}(F_{y})$ in $\prod_{x,y}K_{2}^{top}(F_{x,y})$;}
\item{ $\mathcal{J}_{2}$ is the intersection of $\mathcal{J}_{X}$ with the diagonal image of $\prod_{x \in X}K_{2}^{top}(F_{x})$ in $\prod_{x,y}K_{2}^{top}(F_{x,y})$.}
\end{enumerate} \end{defn}
$\mathcal{J}_{1}$ and $\mathcal{J}_{2}$ are the $K$-group analogues of $\mathbb{B}_{X}$ and $\mathbb{C}_{X}$ respectively. The next proposition proves that these groups depend only on the underlying field and not on the model of $X$ (i.e. the choice of embedding into the algebraic closure $X \times_{F} F^{alg}$) - an important fact for class field theory.

\begin{prop}\label{independenceofmodel} $\mathcal{J}_{y}$ and $\mathcal{J}_{x}$ are independent of the choice of model of $X$.  \end{prop}

\begin{proof}   $\mathcal{J}_{y}$: For each component $y_{i} \subset y$, $F_{y_{i}}$ has the structure of a complete discrete valuation field over $k(y_{i})$.  By the usual theory of complete discrete valuation fields - see \cite{Mor1} -  we may fix an isomorphism $F_{y_{i}} \cong k(y_{i})((t_{y_{i}}))$. \\
The points on (the normalisation of) $y_{i}$ correspond to the valuations of $k(y_{i})$.  Hence the local fields $F_{x,y_{i}}$ are given by $k(y_{i})_{x}((t_{y_{i}}))$, which is well-defined - see \cite[Section 3]{CHDLF}.  So the product \[\prod_{x \in y} K_{2}^{top}(F_{x,y})\] is also well-defined.\\
The following exact sequences follow from the local theory. The second sequences follows from the standard facts about the local boundary maps, which each have kernel $K_{2}^{top}(\mathcal{O}_{x,y})$ and surject onto the (product of) residue fields $k(y)_{x}$. The first sequence follows from the surjection from the groups $K_{2}^{top}(\mathcal{O}_{x,y})$ to the final residue fields $\oplus_{z \in y(x)}k_{z}(x)^{\times}$, which has kernel the principal units.

 By the theory of complete discrete valuations fields and the boundary map of $K$-theory, the sequences are independent of the choices of the $t_{y_{i}}$.
\[\begin{CD} 
0 @>>> \prod_{x \in y}' \excone \times \exctwo @>>> \prod_{x \in y}' K_{2}^{top}(\mathcal{O}_{x,y})  \\
  @. @>pr>> \oplus_{x \in y}\oplus_{z \in y(x)}k_{z}(x)^{\times} @>>> 0 \end{CD}\] and
\[ \begin{CD} 0 @>>> \prod_{x \in y}' K_{2}^{top}(\mathcal{O}_{x,y}) @>>> \jc @>\delta_{X}>> \prod_{x \in y_{i}, y_{i} \in y(x)}' k(y_{i})_{x}^{\times}  @>>> 0  \end{CD}\]
where $pr$ and $\delta_{X}$ are defined as follows. The first term of the first sequence will be defined below, as the kernel of the map $pr$. Note that the restricted product and direct sums appearing in the final terms of the second sequence corresponds to the usual one-dimensional adelic products, i.e. all but finitely many terms are in $\mathcal{O}_{k(y_{i})_{x}}$.  

The boundary homomorphism $\delta : K_{2}^{top}(F_{x,y_{i}}) \to K_{1}^{top}(\bar{F}_{x,y_{i}})$ induces \[\delta_{X}: \mathcal{J}_{X} \to \bigoplus_{y_{i} \subset y \subset X} \mathbb{A}_{k(y_{i})}^{\times}\]
where the range is because of the definition of $\mathfrak{J}_{X}$.
The projection map $K_{2}^{top}(\mathcal{O}_{x,y}) \to K_{2}^{top}(\bar{F}_{x,y})$ induces the surjective map \[pr: {\prod_{x \in y}}' K_{2}^{top}(\mathcal{O}_{x,y}) \to \bigoplus_{x \in y}\oplus_{y_{i} \in y(x)} k_{y_{i}}(x)^{\times}.\] 
The kernel of $pr$ is given by \[ {\prod_{x \in y}}' \excone \times \exctwo\]
where $\mathcal{E}^{(j)}_{x,y} = \prod_{y_{i} \in y(x)}\mathcal{E}_{x,y_{i}}^{(j)}$,  and the $\mathcal{E}_{x,y_{i}}^{(j)}$ for $j=1,2$ are respectively generated by the elements of types \emph{iv} and \emph{v} in \ref{topkgpstructure}. \\
These exact sequences and maps characterise $\jc$ in $\prod_{x \in y}K_{2}^{top}(F_{x,y})$.  \\
As we know the independence of the $k(y_{i})_{x}$ and $k(x)$ from the choice of model (this is a consequence of basic valuation theory, see \cite{CHDLF}), we just need to show the independence of $\prod_{x \in y}' \excone \times \exctwo$.  This will enable the independent characterisation of $\prod_{x \in y}' K_{2}^{top}(\mathcal{O}_{x,y})$ in the first sequence, and hence that of $\jc$ in the second sequence.  \\
$ \prod_{x \in y}' \excone \times \exctwo$ is generated as a group by elements $\{1 + \phi_{k}(u_{x,y_{i}})t_{y_{i}}^{k}, \beta\}$, where $k \geq 1$, $\phi_{k}(u_{x,y_{i}}) \in k(x)((u_{x,y_{i}}))$ and $\beta$ is one of the local parameters $u_{x,y_{i}}$ and $t_{y_{i}}$.  Hence if $\alpha = (\alpha_{x,y_{i}})_{x,y_{i}}$ such that $\{\alpha_{x,y_{i}},\beta\} \in  \prod_{x \in y}' \excone \times \exctwo$, there is a decomposition
\[\alpha_{x,y_{i}} = \prod_{k \geq 1} (1 + \phi_{k,x,y_{i}}(u_{x,y_{i}})t_{y_{i}}^{k}),\]
which enables us to construct $\mathcal{E}_{x,y_{i}}^{(1)} \times \mathcal{E}_{x,y_{i}}^{(2)}$ from $F_{y_{i}}$.  This proves $\mathcal{J}_{y}$ is well-defined as a topological group by $F_{y}$.\\
$\jxx$: Let $R$ be a two-dimensional reduced excellent local ring of characteristic $p$, $\mathfrak{m}$ its maximal ideal and $\mathfrak{p} \subset \mathfrak{m}$ a prime ideal of height one. A ring is excellent if it satisfies some technical conditions, see \cite[remark 4.11]{CHDLF} for a simple discussion of these rings, or \cite[Section 7]{EGA} for a full definition.  \\
Define $\mathcal{J}_{R} \subset \prod_{\mathfrak{p}}K_{2}^{top}(R_{\mathfrak{m}, \mathfrak{p}})$, where $R_{\mathfrak{m}, \mathfrak{p}}$  is constructed by a series of localisations and completions as in section 2.1, by the commutative diagram with exact rows:
\[\minCDarrowwidth20pt \begin{CD}
0 @>>> \prod_{\mathfrak{p}}K_{2}^{top}(R_{\mathfrak{m}, \mathfrak{p}}) @>>> \mathcal{J}_{R} @>\delta>> \oplus_{\mathfrak{p}}K_{1}^{top}(\bar{K}_{\mathfrak{m}, \mathfrak{p}}) @>>> 0 \\
@. @VVV @VVV @VVV @. \\
0 @>>> \prod_{p}K_{2}^{top}(R_{\mathfrak{m}, \mathfrak{p}}) @>>> \prod_{\mathfrak{p}}K_{2}^{top}(K_{\mathfrak{m}, \mathfrak{p}}) @>\delta>> \prod_{\mathfrak{p}}K_{1}^{top}(\bar{K}_{\mathfrak{m}, \mathfrak{p}}) @>>> 0\end{CD}\]
where $K_{\mathfrak{m}, \mathfrak{p}} =$ Frac$(R_{\mathfrak{m}, \mathfrak{p}})$ and the vertical arrows are injective.\\
 Now for a pair $(X,x)$, a two-dimensional scheme over a finite field and a point $x \in X$, let $R = \hat{\mathcal{O}}_{X,x}$ and define $\jxx = \mathcal{J}_{R}$.  Then we have defined $\jxx$ depending only on $\hat{\mathcal{O}}_{X,x}$.
\end{proof}

Hence $\jxx$ depends only on the completion of $\mathcal{O}_{X,x}$, and $\jc$ only on the product of fields $F_{y}$ - so for a complete discrete valuation field $L$ with global residue field, it makes sense to write $\mathcal{J}_{L}$ for the topological $K$-group of its adeles. \\

\emph{Structure of $\jx$}\\

We now look at the structure of this group, providing us with some useful decompositions and a topology.

As above, we have the boundary homomorphism \[\delta_{X}: \jx \to \bigoplus_{y \subset X} \mathbb{A}_{k(y)}^{\times}.\]
Since in each local factor, $\delta(\{ \alpha, \beta\}) = (-1)^{v(\alpha)v(\beta)}\overline{\alpha^{v(\beta)}\beta^{-v(\alpha)}}$ (see \cite{FV} 9.2.3), we have
\[ker(\delta_{X}) = \prod_{y \subset X}{\prod_{x \in y}}' K_{2}^{top}(\mathcal{O}_{x,y}) = \jx \cap \prod_{y \subset X}\prod_{x \in y}K_{2}^{top}(\mathcal{O}_{x,y}).\]
So the map $pr$ is a map with domain the kernel of $\delta_{X}$.\\
Locally, the structure of $ker(pr)$ is clear by the structure theorem for the topological $K$-groups of higher local fields, but globally we must define the restricted product \[{\prod}'_{x \in y, y \subset X} \excone \times \exctwo \subset \prod_{x \in y, y \subset X}\{\mathcal{O}_{x,y}, \tc\} \times \{\mathcal{O}_{x,y}, \uxc\}.\]
As above,  if $\alpha = (\alpha_{x,y})_{x,y}$ such that $\{\alpha_{x,y},\beta\} \in  \prod_{x \in y}' \excone \times \exctwo$, there is a decomposition
\[\alpha_{x,y} = \prod_{z \in y(x)}\prod_{k \geq 1} (1 + \phi_{k,x,z}(u_{x,z})t_{z}^{k}).\]
We have $\{\alpha,\beta\} \in \prod_{x \in y}' \excone \times \exctwo$ if and only if for all $k \geq 1$ and all $ y \subset X$, $(\phi_{k,x,z})_{x \in z, z \in y(x)} \in \prod_{z \in y(x)}\mathbb{A}_{k(z)}$. The decomposition is unique, see \cite[Section 2, Proposition 3]{P4}.
\\

\emph{Topology} \\

To define the topology of $\mathcal{J}_{X}$, we follow Fesenko's definition from \cite[Section 2]{F3}. Fesenko shows that our group $\mathcal{J}_{X}$ is isomorphic to the group which is defined as followed. Let $V_{X}$ be the image of the $K$-group symbol map on the subgroup of the adeles $\mathbb{A}_{X}$ where for each pair $(x,y)$, the entry $a_{x,y}$ is in $\mathcal{O}_{x,y}$ and its image in the residue field is in the ring of integers $\mathcal{O}_{k(y)_{x}}$. Then $\mathcal{J}_{X}$ is isomorphic to:
\[V_{X} + \oplus_{x \in y \subset X}K_{0}(k_{y}(x)).\]
See \cite[Section 2]{F3} for more details. 

We give $V_{X}$ the product topology from the subgroup of the adeles, and then $J_{X}$ the sequential saturation of the topology induced by the product of this and the discrete topology on the $K_{0}$ terms.

\subsection{The Global Witt Pairing}

In this section we will define the global Witt pairing as a sum of the traces of local Witt pairings, prove that it is a well-defined sum and check some basic properties.

\begin{defn}
For each positive $m \in \mathbb{Z}$, define the global Witt pairing
\[( \ \ | \ ]_{X} : \jx \times W_{m}(\mathbb{A}_{X})\to W_{m}(\mathbb{F}_{q})\]
by \\
$(\{(f_{x,y})_{x \in y}, (g_{x,y})_{x \in y}\}| (h_{x,y})_{x \in y}]_{X}$\[ \ \ \ \ \ \ \ \ \ \ \ \ \ \ \ \ \ \ \ \ \ \  \ \ \ \ \ \ \ \ \ \ \ \  \mapsto \sum_{y \subset X}\sum_{x \in y, z_{i} \in y(x)}\mathrm{Tr}_{W_{m}(k_{z_{i}}(x))/W_{m}(\mathbb{F}_{p})}(f_{x,z_{i}}, g_{x,z_{i}} | h_{x,z_{i}}]_{F_{x,z_{i}}}.\] \end{defn}

\noindent We now check that this sum converges.
\begin{lemma}\label{wittwelldefined}
Let $( \ \ | \ ]_{F_{x,y}}$ be the Witt pairing associated to the product of two-dimensional local fields $F_{x,y}$, and let $m \in \mathbb{Z}$.  Then the map
\[\mathbb{A}_{X}^{\times} \times \mathbb{A}_{X}^{\times} \times W_{m}(\mathbb{A}_{X}) \to W_{m}(\mathbb{F}_{p})\]
$((f_{x,y})_{x \in y}, (g_{x,y})_{x \in y}, (h_{x,y})_{x \in y})$\[ \ \ \ \ \ \ \ \ \ \ \ \ \ \ \ \ \ \ \ \ \ \  \ \ \ \ \ \ \ \ \ \ \ \ \ \ \ \ \mapsto \sum_{y \subset X}\sum_{x\in y, z_{i} \in y(x)}\mathrm{Tr}_{W_{m}(k_{z_{i}}(x))/W_{m}(\mathbb{F}_{p})}(f_{x,y}, g_{x,y} | h_{x,y}]_{F_{x,y}}\]
is well-defined, i.e. there are only finitely many non-zero terms appearing in the sum.\end{lemma}

\begin{proof}
We induct on the length of the Witt vectors.\\
\emph{m = 1:}  Firstly note that the pairing is symbolic as in the local case (see proposition \ref{residueproperties} property $iii$) so in fact we consider a pairing
\[\jx \times \mathbb{A}_{X} \to \mathbb{F}_{p}.\]
From the discussion of the structure of $\mathcal{J}_{X}$ above, if we let $\Gamma$ be the image of a section of $\delta_{X}$ \[\sigma: \mathcal{J}_{X} \to \bigoplus_{y \subset X}\prod_{x \in y, z_{i} \in y(x)} k(z_{i})_{x}^{\times},\]
then $\mathcal{J}_{X}$ can be decomposed as
\[\mathcal{J}_{X} \cong \Gamma \times \prod_{y \subset X}\prod_{x \in y, z_{i} \in y(x)}K_{2}^{top}(\mathcal{O}_{x,z_{i}}).\]
Note here that when we write $K_{2}^{top}(\mathcal{O}_{x,y_{i}})$, we mean the topological quotient of the tensor product $\mathcal{O}_{x,y_{i}}^{\times} \otimes \mathcal{O}_{x,y_{i}}^{\times}$ by $I_{2}\cap( \mathcal{O}_{x,y_{i}}^{\times} \otimes \mathcal{O}_{x,y_{i}}^{\times})$.\\
From the additive property of the local Witt pairing, we can evaluate on $\Gamma$ and $ \prod_{y \subset X}\prod_{x \in y, z_{i} \in y(x)}K_{2}^{top}(\mathcal{O}_{x,z_{i}})$ separately.\\
\\
For $ \prod_{y \subset X}\prod_{x \in y, z_{i} \in y(x)}K_{2}^{top}(\mathcal{O}_{x,z_{i}})$,  as the last term in the pairing $h = (h_{x,z_{i}})$ satisfies $h_{x,z_{i}} \in \mathcal{O}_{x,z_{i}}$ for all but finitely many $(x,z_{i})$ we may apply property $i$ in lemma \ref{residueproperties}.  Hence the pairing takes only finitely many non-zero values here.\\
Let $\Gamma$ be generated by the section
\[\bar{\gamma} \mapsto \{\gamma, \tc\}\]
where $\gamma$ is the lift of $\bar{\gamma}$ induced by $F_{x,y} \cong k(y)_{x}((\tc))$ - this does depend on the choice of $\tc$, but we will see this does not affect the proof. 
By lemma \ref{residueproperties} $ii$, we have \[(\{\gamma, t_{z_{i}}\}| h_{x,z_{i}}]_{F_{x,z_{i}}} = \mathrm{res}_{\overline{F_{x,z_{i}}}}\left( \bar{h}_{x,z_{i}}\frac{d\bar{\gamma}}{\bar{\gamma}}\right)\]
so we have reduced to the one-dimensional case.  It is well-known from the study of differential forms on curves that there are only finitely many non-zero values here, so the base case is complete.\\
\emph{Induction:}  Suppose $(\{f_{x,y}, g_{x,y}\} | h_{x,y}]_{F_{x,y}} = 0$ for all but finitely many $(x,y)$, where $h = (h_{x,y}) \in W_{m-1}(\mathbb{A}_{X})$.
By $viii$ in proposition \ref{wittproperties},  \[(\{f_{x,y}, g_{x,y}\} | (h_{0}, \dots, h_{m-1})_{x,y}]_{F_{x,y}} = (w_{o}, \dots, w_{m-1})\]
implies \[ (\{f_{x,y}, g_{x,y}\} | (h_{0}, \dots, h_{m-2})_{x,y}]_{F_{x,y}} = (w_{0}, \dots, w_{m-2}).\]
Suppose there exists $h \in W_{m}(\mathbb{A}_{X})$ such that the pairing $(\{f_{x,y}, g_{x,y}\} | h_{x,y}]_{F_{x,y}}$ takes infinitely many non-zero values for some $f, g \in \mathbb{A}_{X}$.  Any pair $(x,y)$ such that 
\[(\{f_{x,y}, g_{x,y}\} | (h_{0},\dots,h_{m-1})_{x,y}]_{F_{x,y}} = (w_{0}, \dots, w_{m-1}) \neq 0\]
has
\[(\{f_{x,y}, g_{x,y}\} | (h_{0},\dots,h_{m-2})_{x,y}]_{F_{x,y}} = (w_{0},\dots,w_{m-2})\]
so $w_{m-1}$ must be the only non-zero term for all but finitely many such values of the pairing.\\
By definition of the Witt pairing, it can be seen that this implies $h_{x,y} = (0,0, \dots, 0, h_{m-1})_{x,y}$ for almost all of the pairs $(x,y)$ giving non-zero values.  But by induction on relation $vii$ in proposition \ref{wittproperties}, we can reduce to the case $m=1$, which gives a contradiction.
\end{proof}

The lemma above shows this pairing is well-defined.  By \cite{P4}, 3.3.1, the components of each local pairing are polynomials in the components of $f_{x,z}, g_{x,z}$ and $h_{x,z}$, proving continuity of the local pairing.  Since there are only finitely many non-zero terms this extends to continuity of the global pairing.

\begin{prop}\label{reciprocitylaw} Reciprocity Law\\
For a fixed curve $y \subset X$, $m \in \mathbb{Z}$, $f,g \in F_{y}^{\times}$ and $h \in W_{m}(F_{y})$,
\[\sum_{x \in y, y_{i} \in y(x)} \mathrm{Tr}_{W_{m}(k_{y_{i}}(x))/W_{m}(\mathbb{F}_{p})}(\{f_{x,y}, g_{x,y}\}|h_{x,y}]_{F_{x,y}} = 0.\] 
For a fixed point $x \in X$, $m \in \mathbb{Z}$, $f,g \in F_{x}$ and $h \in W_{m}(F_{x})$
\[\sum_{y \ni x} \mathrm{Tr}_{W_{m}(k_{y}(x))/W_{m}(\mathbb{F}_{p})}(\{f_{x,y}, g_{x,y}\} | h_{x,y}]_{F_{x,y}} = 0.\]

\end{prop}
\begin{proof}
See \cite{ME}.
\end{proof}

\begin{cor}\label{pairing} For each $ m \in \mathbb{Z}$ there is a continuous pairing
\[ ( \ | \ ]_{X}: \frac{\mathcal{J}_{X}}{\mathcal{J}_{1}+\mathcal{J}_{2}} \times W_{m}(F)/(Frob - 1)W_{m}(F) \to \mathbb{Z}/p^{m}\mathbb{Z}.\]

\end{cor}
Notice the relation to one-dimensional class field theory - the quotient here is an analogue of the quotient of the idele group by the global elements to obtain the idele class group. The higher tame symbol described below will also take values on this group, as a similar reciprocity law is proved in the paper above.\\ 

In the following sections, we aim to prove that the Witt pairing is non-degenerate on certain subgroups and quotients of the groups on which it is defined, along with similar results for the higher tame pairing. \\

\subsection{The Global Higher Tame Pairing}

This section begins with a definition of the global higher tame pairing then proceeds in a manner similar to the previous section on the Witt pairing - we check the pairing is well-defined and prove basic properties. 

\begin{defn}
For a surface $X$ over a finite field  $k$, $\{f,g\} = (\{f_{x,y},g_{x,y}\})_{x,y}$ $ \in \mathcal{J}_{X}$ and $h= (h_{x,y})_{x,y} \in \mathbb{A}_{X}$, define the global higher tame pairing by 

\[(\{f,g\}, h)_{X} = \prod_{y \subset X}\prod_{x \in y, z_{i} \in y(x)} N_{k_{z_{i}}(x)/k}(\{f_{x,z_{i}},g_{x,z_{i}}\}, h_{x,z_{i}})_{x,z_{i}}\]
where for each pair $(x,z_{i})$, the symbol $( \ , \ )_{x,z_{i}}: K_{2}^{top}(F_{x,z_{i}}) \times F_{x,z_{i}} \to k$ is the local higher tame symbol. 
\end{defn}
\begin{lemma}\label{tamewlldefined}
The global higher tame pairing is well-defined, i.e. for fixed $(\{f_{x,y},g_{x,y}\})_{x,y} \in \mathcal{J}_{X}$, $h \in \mathbb{A}_{X}$, as $(x,z)$ range over all points on all branches of the curves $y$ on $X$, the value of $(\{f_{x,y}, g_{x,y}\}, h_{x,y})_{x,z}$ is not equal to one for only finitely many pairs $(x,y)$.
\end{lemma}
\begin{proof}
First fix $x \in X$. For each $z \ni x$, we may decompose our elements $f_{x,z}$, $g_{x,z}$, $h_{x,z}$ as products $\alpha_{x,z}u_{x,z}^{i}t_{z}^{j}\varepsilon_{x,z}$, with $j=0$ for all but finitely many $z$, $\alpha_{x,z} \in k_{z}(x)$ and $\varepsilon_{x,z}$ a principal unit of $\mathcal{O}_{F_{x,z}}$.

 If we fix a $j \in \mathbb{Z}$, then there are only finitely many expansions of a fixed element with the exponent of $t_{z}$ being $j$ and the exponent of $u_{x,z}$ being non-zero. So the number of $z$ with $i$ and $j$ not equal to zero is certainly finite. So for all but finitely many $z$, $f_{x,z},g_{x,z}$ and $h_{x,z}$ will all be in the group $k_{z}(x)^{\times} \times \mathcal{U}_{x,z}$.\\
Basic properties of the higher tame symbol show it is trivial if any entry is in the group of principal units $\mathcal{U}_{x,z}$, and a simple calculation shows it is also trivial if more than one of the entries is in $k_{z}(x)^{\times}$, which shows there are only finitely many values not equal to one for a fixed point.

Now we fix a curve $y$, and proceed in exactly the same way to the case for a fixed point. Putting these two cases together, the proof is complete.
\end{proof}

\begin{prop}\label{tamereciprocity}
For a fixed curve $y \subset X$, $f$, $g$, $h \in F_{y}^{\times}$, the product
\[\prod_{x \in y, z_{i} \in y(x)}N_{k_{z_{i}}(x)/k}(\{f,g\},h)_{x,z_{i}} = 1.\]
For a fixed point $x \in X$, $f$, $g$ and $h \in F_{x}$, the product
\[\prod_{y \ni x}N_{k_{y}(x)/k}(\{f,g\}, h)_{x,y} = 1.\]
\end{prop}
\begin{proof}
See \cite{ME}.\end{proof}

\begin{cor}
The higher tame symbol defines a pairing
\[ ( \ \ , \ )_{X} : \frac{\mathcal{J}_{X}}{\mathcal{J}_{1} + \mathcal{J}_{2}} \times \mathbb{A}_{X}^{\times} \to k^{\times}.\]
\end{cor}

In the following sections, we will use Kummer theory to get duality theorems which will enable us to define the tamely ramified part of the reciprocity map for $X$.

We will now proceed by splitting into the two semi-global cases of a fixed curve and a fixed point, proving the duality of the Witt pairing for $F_{y}$ and $F_{x}$.

\section{Complete Discrete Valuation Fields over Global Fields}
\setcounter{subsection}{1}\subsection*{}\setcounter{thm}{0}

In this section we fix a curve $y \subset X$, and hence a product of fields $F_{y} \cong \prod_{y_{i} \subset y}k(y_{i})((t_{y_{i}}))$ - each one a complete discrete valuation field over a global field $k(y_{i})$. We will denote the finite constant field of $k(y_{i})$ by $k_{y_{i}}$, and let $k_{y}=\prod_{y_{i} \subset y}k_{y_{i}}$.

We begin with a definition of a subgroup of the adeles for the curve $y$. This will be the subgroup on which the Witt pairing is non-trivial.
\begin{defn}Define $J_{y}: = \prod'_{x \in y}K_{2}^{top}(\mathcal{O}_{x,y}, \mathfrak{m}_{x,y})$. Recall the definition of the restricted product is from \ref{restricteddef}.\end{defn}
The first theorem we aim to prove is the following Witt duality theorem for a non-singular curve. The Frobenius element Frob is the canonical generator of the Galois group of the maximal unramified extension of $F_{y}$. It acts on each term of the Witt vectors.
\begin{thm}\label{mainthmcurves}
For a fixed non-singular curve $y \subset X$ and $m \in \mathbb{N}$, the pairing
\[J_{y}/(K_{2}^{top}(F_{y})\cap J_{y})J_{y}^{p^{m}} \times W_{m}(F_{y})/(\mathrm{Frob} -1)W_{m}(F_{y})+W_{m}(k(y)) \to \mathbb{Z}/p^{m}\mathbb{Z}\]
is continuous and non-degenerate, and the induced homomorphism 
from $J_{y}/(K_{2}^{top}(F_{y})\cap J_{y})J_{y}^{p^{m}}$
to \[\mathrm{Hom}(W_{m}(F_{y})/(\mathrm{Frob} - 1)W_{m}(F_{y})+W_{m}(k(y)), \mathbb{Z}/p^{m}\mathbb{Z})\] is a topological isomorphism.\end{thm}

We will proceed by induction on $m$. To prove the theorem for $m = 1$, we need a series of technical lemmas. \\

We first discuss why the pairing is taken on this group. The quotient by the diagonal elements $K_{2}^{top}(F_{y})$ is because of the reciprocity law \ref{reciprocitylaw}, and the quotient by $J^{p^{m}}_{y}$ is because of \ref{wittproperties}, properties four and seven. From Parshin's calculations in \cite[3.2.5]{P4} we see that for each field $F_{x,y}$, elements of the $K$-group containing principal units and elements of the finite field $k_{y}(x)$ are the only elements where the Witt pairing can take non-zero values. We quotient by the constant elements as these are the ones related to unramified extensions, i.e. $p^{th}$-powers of the Frobenius element. The following lemma on the structure of the $K$-groups will complete this discussion.

\begin{lemma}\label{kgpstructurecurves}
Fix non-singular $y \subset X$.  Then $K_{2}^{top}(F_{y})$ is generated by symbols of the form:
\begin{enumerate}
\item{ $\{a,\tc\}$ with $a \in k(y)^{\times}$;}
\item{ $\{a,b\}$ with $ a,b \in k(y)^{\times}$;}
\item{ $\{1+a\tc^{k}, \tc\}$ with $ a \in k(y)^{\times}$, $k \geq 1$;}
\item{$\{1+a\tc^{k}, b\}$ with $a, b \in k(y)^{\times}$, $k \geq 1$.}\end{enumerate}
\end{lemma}
The proof of this lemma is exactly the same as for a two-dimensional local field, see \ref{topkgpstructure}. Notice again we are choosing a smooth irreducible curve $y$  - the discussion of the singular case follows at the end of the section.

For fixed $y \subset X$, $x \in y$, let $\mathcal{E}^{(1)}_{x,y}$ be the group generated by the symbols with entries as in proposition \ref{topkgpstructure} part four in the first position, and $\mathcal{E}^{(2)}_{x,y}$ the group generated with symbols in part five in the first position, and local parameters in the second.  Using proposition \ref{topkgpstructure} we can now write
\[{\prod}'_{y \subset X}{\prod}'_{x \in y} K_{2}^{top}(\mathcal{O}_{x,y}, \mathfrak{m}_{x,y}) = {\prod}'_{y \subset X}{\prod}'_{x \in y} \mathcal{E}^{(1)}_{x,y} \times \mathcal{E}^{(2)}_{x,y}\]
  and using the lemma above, we know the two groups we quotient by are generated by symbols
  \[\{1+a\tc^{k}, \tc\} \text{ with } a \in k(y)^{\times},  \ k \geq 1; \ \ \  \{1+a\tc^{k}, b\} \text{ with }a, b \in k(y)^{\times}, \ k \geq 1.\]
  
\noindent We examine the structure of this group further.

\begin{lemma}\label{relativekgpcurves}
Let $\alpha \in \prod'_{x \in y} \mathcal{E}^{(1)}_{x,y} \times \mathcal{E}^{(2)}_{x,y}$.  Then $\alpha$ can be decomposed as $\alpha^{(1)}\alpha^{(2)}$, where:
\[ \aone_{x,y} = \{\varepsilon_{x,y}^{(1)}, \tc\}, \text{     with   } \varepsilon^{(1)}_{x,y} \in \mathcal{E}^{(1)}_{x,y}\]
and
\[ \atwo_{x,y} = \{\varepsilon_{x,y}^{(2)}, \uxc\}, \text{      with    } \varepsilon^{(2)}_{x,y} \in \mathcal{E}^{(2)}_{x,y}\]
are unique expansions for each $x \in y$.
The unique decomposition of the $\varepsilon_{x,y}^{(i)}$ can be rewritten as:
\[\varepsilon_{x,y}^{(1)} = \prod_{j \geq 1}(1 + \phi_{j,x,y}^{(1)}(\uxc)\tc^{j})\]
\[\varepsilon_{x,y}^{(2)} = \prod_{j \geq 1}(1 + \phi_{j,x,y}^{(2)}(\uxc)\tc^{j})\]
where $\phi_{j,x,y}^{(i)}(\uxc) \in k(y)_{x}$ satisfy:
\begin{enumerate}
\item{ $(\phi_{j,x,y}^{(i)}(\uxc))_{x \in y} \in \mathbb{A}_{k(y)}$ for $i = 1,2 $ and for all $j$;}
\item{ If $k$ is such that $\phi^{(1)}_{j,x,y}(\uxc) = 0$ for all $ j < k$, then $\phi^{(1)}_{k,x,y}(\uxc)$ contains no powers $\uxc^{i}$ with $p | i$. }
\item{ If $k$ is such that $p | k$ and $\phi^{(2)}_{j,x,y}(\uxc) = 0$ for all $j < k$, then $\phi^{(2)}_{k,x,y}(\uxc) = 0$.}
\item{  For all $k$ and for all $x \in y$, $\phi^{(2)}_{k,x,y}(\uxc) = \psi_{k,x,y}(\uxc^{p})$ for some series $\psi_{k,x,y} \in k(y)[[X]]$.}
\end{enumerate}
\end{lemma}

\begin{proof}
By the structure theorem for $K_{2}^{top}(\fc)$, the decomposition $\alpha = \aone\atwo$ is clear.  The uniqueness follows from \cite{P4}, corollary to proposition 4, section one.\\
Property 1 follows from the induced (by our definition of the adeles) restricted product of the groups $\mathcal{E}^{(1)}_{x,y}$, $ \mathcal{E}^{(2)}_{x,y}$.\\
Suppose $k$ is such that $\phionej = 0$ for all $j < k$.  Since the product in $\mathcal{E}^{(1)}_{x,y}$ is taken over the indices not divisible by $p$, the only powers $\uxc^{i}$ with $p | i$ must come from sums terms in  $\phionej$ for $j < k$ - but these are all zero.  So property 2 is proved.\\
Suppose $k$ is such that $p | k$ and $\phitwoj = 0$ for all $j < k$.  The product in $\mathcal{E}^{(2)}_{x,y}$ is taken over the indices with $p$ not dividing the index of $\tc$, so 3 is proved in the same way as 2 above.\\
Finally, for any $k$ and $x \in y$, the product in $\mathcal{E}^{(2)}_{x,y}$ is taken so that the second index is divisible by $p$, so property 4 is clear.
\end{proof}

We now look at the expansion given above for elements of $K_{2}^{top}(F_{y})$. This will give us a general form for elements of the diagonal group, enabling us to prove that elements of the kernel of the Witt pairing are exactly the diagonal elements.

\begin{lemma}\label{diagonalelements}
Let $\{1 + a\tc^{k}, \tc\}$, $\{1 + h\tc^{l}, b\} \in K_{2}^{top}(F_{y})$ for some $k, l > 0, \ a,b,h \in k_{y}(x)$ and $\alpha_{1}, \alpha_{2}$ their respective images in $J_{y}$.  Then for $\alpha_{1}$:
\[
\phionej_{1} = \begin{cases} 0 & \text{if } j < k \\
a \text{ mod }k(y)_{x}^{p} & \text{if } j = k \end{cases}\]
\[ \phitwoj_{1} =  0 \ \ \ \  j\leq k .\]
For $\alpha_{2}$, let $\eta = (\eta_{x})_{x \in y} \in \mathbb{A}_{k(y)}$ be defined by:
\[h\uxc b^{-1}\frac{db}{d\uxc} + \uxc\frac{d\eta_{x}}{d\uxc} \in k(y)^{p}_{x}.\]
Then:
\[\phionej_{2} = \begin{cases} 0 & \text{ if } j < l\\
l\eta_{x} & \text{ if } j = l \end{cases}\]
\[\phitwoj_{2} = \begin{cases} 0 & \text{ if } j < l \\
h\uxc b^{-1}\frac{db}{d\uxc} + \uxc\frac{d\eta_{x}}{d\uxc} & \text{ if } j= l. \end{cases}\]
\end{lemma}
\begin{proof}
First consider $\alpha_{1}$.  For $j < k$, the claim is clear.  Let $\delta = (\delta_{x})_{x \in y} \in \mathbb{A}_{k(y)}$ be defined by
\[a = \phionek + \delta_{x}(\uxc)^{p}\]
with $\delta_{x} \in k(y)_{x}$.  Such a delta exists  by the expansion of $\aone \in \excone \times \exctwo$.  \\
For any $j \in \mathbb{Z}$, define $J_{y, \geq j}$ as
\[ \left\{ \alpha \in {\prod}'_{x \in y}K_{2}^{top}(\mathcal{O}_{x,y}, \mathfrak{m}_{x,y}): (\phi^{(1)}_{i,x,y})_{x \in y} = 0 \text{ and } (\phi^{(2)}_{i,x,y})_{x \in y} = 0  \ \forall \ i < j\right\}.\]
It is enough to show that
\[ \{1 + \delta^{p}_{x}\tc^{k}, \tc\} = \{1 + a\tc^{k}, \tc\}\{ 1 - \phionek \tc^{k}, \tc\} \in J_{y}^{p}J_{y, \geq k+1}\]
as then we have the correct value modulo $p^{th}$-powers, and the remaining terms affect only $\phionej$ with $j > k$.\\
If $p | k$, then $\{1 + \delta_{x}^{p}\tc^{k}, \tc\} \in J_{y}^{p}$, so assume $p \nmid k$.  We have the identity:
\[\{1 + \delta_{x}^{p}\tc^{k}, -\delta_{x}^{p}\tc^{k}\} = 1\]
by definition of the $K$-groups.  Hence
\[\{1 + \delta_{x}^{p}\tc^{k}, \tc\} \equiv  \{1 + \delta_{x}{^p}\tc^{k}, \delta_{x}\}^{p} \text{ mod } J_{y, \geq k + 1}.\]
See \ref{Kgpcalc1} in the appendix for the details of this calculation.\\
So now consider $\alpha_{2}$.
Let $f_{i}, g_{j} \in k(x)$.  We have:
\[\{1 + f_{i}\uxc^{i}\tc^{l}, 1+g_{j}\uxc^{j}\} \equiv \left\{1 + f_{i}\uxc^{i}\frac{jg_{j}\uxc^{j}}{1 + g_{j}\uxc^{j}}\tc^{l}, \uxc\right\} \text{ mod } K_{2}^{top}(\mathcal{O}_{x,y}, \mathfrak{m}_{x,y}^{l+1}),\]
see appendix, \ref{Kgpcalc3}.\\
Let $h = \sum_{i}f_{i}\uxc^{i}$ and $b = \prod_{j}(1+g_{j}\uxc^{j})$, so that
\[\frac{db}{d\uxc} = \left( \sum_{j}\frac{jg_{j}\uxc^{j-1}}{1 + g_{j}\uxc^{j}}\right)b.\]
Hence
\[\{1 + h\tc^{l}, b\} \equiv \left\{ 1 + \uxc b^{-1}\frac{db}{d\uxc}\tc^{l}, \uxc\right\} \text{ mod } K_{2}^{top}(\mathcal{O}_{x,y}, \mathfrak{m}_{x,y}^{l+1})\]
and so
\[\{1 + h\tc^{l}, b\} \equiv \left\{1 + \phi^{(2)}_{l,x,y}(\uxc)\tc^{l}, \uxc\right\}\left\{1 - \uxc\frac{d\eta_{x}}{d\uxc}\tc^{l},  \uxc\right\}\] modulo $K_{2}^{top}(\mathcal{O}_{x,y}, \mathfrak{m}_{x,y}^{l+1})$, if we let $\phi^{(2)}_{l,x,y}$ be as required.\\

\noindent We have the relation
\[\{1 - iv_{i}\uxc^{i}\tc^{l}, \uxc\} \equiv \{ 1+ lv_{i}\uxc^{i}\tc^{l}, \tc\} \text{ mod } K_{2}^{top}(\mathcal{O}_{x,y}, \mathfrak{m}_{x,y}^{l+1})\]
for $p \nmid i$, $v_{i} \in k(x)$.  See \ref{Kgpcalc2} for details. So if we let $\eta_{x} = \sum_{p \nmid i} m_{i}\uxc^{i}$, we get
\[\left\{ 1 - \uxc \frac{d\eta_{x}}{d\uxc}\tc^{l}, \uxc\right\} \equiv \left\{ 1 + l\eta_{x}\tc^{l}, \tc\right\} \text{ mod } K_{2}^{top}(\mathcal{O}_{x,y}, \mathfrak{m}_{x,y}^{l+1}).\]
Combining the two calculations, we see
\[\{ 1 + h\tc^{l}, b\} \equiv \{1 + \phi^{(2)}_{l,x,y}(\uxc)\tc^{l}, \uxc\} + \{1 + l\eta_{x}\tc^{l}, \tc\} \text{ mod } K_{2}^{top}(\mathcal{O}_{x,y}, \mathfrak{m}_{x,y}^{l+1})\]
so we may let $\phi^{(1)}_{l,x,y}(\uxc) = l\eta_{x}$ as required.  Note that we need only to prove the lemma $\text{ mod } K_{2}^{top}(\mathcal{O}_{x,y}, \mathfrak{m}_{x,y}^{l+1})$ as higher terms will affect $\phi^{(i)}_{j,x,y}$ with $j > l$.
\end{proof}

\emph{Remark} The uniqueness of the decomposition means that if we show an element of $J_{y}$ can be written in the above form, then it is in the diagonal image of $K_{2}^{top}(F_{y})$.\\

\medskip
The next lemma will provide a simple form for the elements of $F_{y}/(Frob - 1)F_{y}.k( y)$, enabling us to prove non-degeneracy on the right-hand side of the Witt pairing.
\begin{lemma}\label{structuremodfrob}
Let $f \in \fc/((\mathrm{Frob} - 1)\fc +k(y))$. Then $f$ has a unique representation as a finite sum
\[f= \sum_{k < 0}f_{k}\tc^{k}\]
with $f_{k} \in k(y)$ and if $p | k$, $f_{k} \in \mathcal{R}_{p}$, a fixed set of representatives for $k(y)/k(y)^{p}$.\end{lemma}
\begin{proof}
Decompose $f$ as $f = \sum_{k \geq v_{y}(f)}f_{k}\tc^{k}$.  If $v_{y}(f) > 0$, consider the convergent (for $v_{y}(f) > 0$) sum:
\[f' = (-f) + (-f)^{p} + (-f)^{p^{2}} + \dots \]
with $f = f'^{p}-f' \in (Frob - 1)\fc$.  So modulo $(Frob - 1)\fc$ we need only consider the terms with $k < 0$.\\
Suppose $k < 0$ is the least such with $p | k$ and $f_{k} \not\in \mathcal{R}_{p}$.  Then \[f_{k} = f''_{k} + g^{p}\] some $f''_{k} \in \mathcal{R}_{p}$, $g \in k(y)$.  Replace $f_{k}$  by $f''_{k}$ and $f_{k/p}$ by $f_{k/p} + g$ to get another representation of $f$, and continue this process until the representation is as required.\\
\emph{Uniqueness}:  Suppose $\sum_{k \leq 0}f_{k}\tc^{k}$ and $\sum_{k \leq 0}f'_{k}\tc^{k}$ represent $y$ in the required form.  Then:
\[\sum_{k \leq 0}(f_{k} - f'_{k})\tc^{k} = \left( \sum_{k}h_{k}\tc^{k}\right)^{p} -  \sum_{k}h_{k}\tc^{k} = h^{p} - h \ \ \ \ \ \ (\star)\]
some $h \in \fc$. \\
Then $f_{0} - f'_{0} = h_{0}^{p} - h_{0} \in (Frob - 1)\fc+ k(y)$, which implies $f_{0} = f'_{0}$.\\
Let $k < 0$ be the least $k$ with $f_{k} \neq f'_{k}$.  Then for $i > 0$, 
\[h_{p^{i}k} = h_{p^{i}k} + f_{p^{i}k} - f'_{p^{i}k}.\]
But equating coefficients in $(\star)$ gives:
\[h_{p^{i}k} + f_{p^{i}p} - f'_{p^{i}k} = h_{p^{i-1}k}^{p}\]
and hence $h_{p^{i}} = h_{k}^{p^{i}}$ by induction.  \\
But for large enough $i$, we have $h_{p^{i}k} = 0$, so $h_{k} = 0$ and we must have
\[ f_{k} - f'_{k} = \begin{cases} 0 & p \nmid k \\
(-h_{k}/p)^{p} & p | k\end{cases} \equiv 0\]
contradicting our choice of $k$.\end{proof}

We can now calculate the value of the Witt pairing on elements of $J_{y}$ and $F_{y}/($Frob$ - 1)F_{y}.k(y)$ in these useful forms.

\begin{lemma}\label{firstcalculation}
Fix some $k \geq 1$, $l \leq -1$.  Let $\aone_{k} \in J_{y}$ be an element of the form described in \ref{relativekgpcurves} such that $\phionej = 0$ for all $ j \neq k$ and $\phitwoj = 0$ for all $j$.  Let $\atwo_{k} \in J_{y}$ be an element of the form described in \ref{relativekgpcurves} such that  $\phionej = 0 $ for all $j$ and $\phitwoj = 0$ for all $j \neq k$. Let $f_{l} \in k_{y}$. Then:
\[ (\aone_{k}|f_{l}\tc^{l}]_{y} = \begin{cases} 0 & k + l > 0 \\
\sum_{x \in y}\mathrm{Tr}_{k(x)/k)}(\mathrm{res}_{x}(f_{l}d\phionek)) & k + l = 0 \end{cases}\]
and 
\[(\atwo_{k}|f_{l}\tc^{l}]_{y} = \begin{cases} 0 & k + l > 0 \\
-\sum_{x \in y}\mathrm{Tr}_{k(x)/k}\left(\mathrm{res}_{x}\left(f_{l}k\phitwok\frac{d\uxc}{\uxc}\right)\right) & k + l = 0, \ p \nmid k \\
0 & k+ l = 0,  \ p | k.\end{cases}\]
\end{lemma}
\begin{proof}
For each $x \in y$, we have:
\[ (\aone_{x,y}|f_{l}\tc^{l}]_{y} = \mathrm{res}_{x,y}\left( f_{l}\tc^{l}\frac{d(\phionek \tc^{k})}{1 + \phionek\tc^{k}}\wedge \frac{d\tc}{\tc}\right)\]
which is equal to
\[\mathrm{res}_{x,y}\left( f_{l}\tc^{k+l}\frac{d\phionek}{1 +\phionek\tc^{k}}\wedge \frac{d\tc}{\tc}\right) = 
\begin{cases} 0 & k + l > 0 \\
\mathrm{res}_{x}(f_{l}d\phionek) & k + l = 0\end{cases}\]
by expanding $(1 + \phionek \tc^{k})^{-1}$ and using property 3.4.1, \emph{ii}.  Summation over $x \in y$ gives the first part of the lemma.\\
Similarly, we have
\[(\atwo_{x,y}|f_{l}\tc^{l}]_{y} = \mathrm{res}_{x,y}\left( f_{l}\tc^{l}\frac{d\phitwok\tc^{k}}{1+\phitwok\tc^{k}}\wedge \frac{d\uxc}{\uxc}\right)\]
which is equal to
\[\mathrm{res}_{x,y}\left(f_{l}k\tc^{l+k}\frac{\phitwok}{1+\phitwok\tc^{k}}\frac{d\tc}{\tc}\wedge\frac{d\uxc}{\uxc}\right) \]\[ \ \ \ \ \ \ \ \ \ \ \ \ \ \ \ \ \ \ \ \ \ \ \ \ \ \ \ \ \ \ \ = 
\begin{cases} 0 & k + l > 0\\
-\mathrm{res}_{x}\left( f_{l}k\phitwok\frac{d\uxc}{\uxc}\right) & k + l = 0.
\end{cases}\] exactly as for $\alpha^{(1)}$.
As before summing over $x \in y$ gives the lemma.\end{proof}

Denote the set of elements with both $\phi_{j,x,y}^{(1)}(u_{x,y}) = 0$ and $\phi^{(2)}_{j,x,y}(u_{x,y}) = 0$ for all $j < k$ by $J_{y,\geq k}$. combining these two lemmas gives:

\begin{cor}\label{firstcalccor}
Fix $k \geq 1$ and let $\alpha_{\geq k} \in J_{y,\geq k}$.  We may decompose this element as $\alpha_{\geq k} = \alpha_{\geq k + 1}\alpha_{k}$, where $\alpha_{k} = \aone_{k}\atwo_{k}$ as in the lemma above.  Let $l \leq -1$.  Then
$ (\alpha_{\geq k}|f_{l}\tc^{l}]_{y}$ is given by \[  
\begin{cases} 0 & k + l > 0\\
\sum_{x \in y}\mathrm{Tr}_{k(x)/k}\left(\mathrm{res}_{x}\left(f_{l}\left(d\phionek - k\phitwok\frac{d\uxc}{\uxc}\right)\right)\right) & k + l = 0.
\end{cases}\]\end{cor}

We now move on to studying the case $k + l =0$, treating the two cases $p \nmid k$ and $ p | k$ separately for now. Note that we have not mentioned the case $k+l<0$ yet, as this will not be needed in the proof of the main theorem. The following lemmas prove non-degeneracy of the Witt pairing on the subspaces $J_{y,\geq k}$ modulo the higher powers and the diagonal elements.

\begin{lemma}\label{dualityone}
Fix $k \geq 1$ with $ p \nmid k$.  Then the map
\[ ( \ | \ ]_{y}: J_{y,\geq k}/(\Delta(K_{2}^{top}(\fc)) \cap J_{y,\geq k})J_{y,\geq k + 1} \times t_{y}^{-k}k(y) \to k_{y}\]
is a non-degenerate pairing of $k_{y}$-vector spaces.  The induced homomorphism 
\[ J_{y,\geq k}/(\Delta(K_{2}^{top}(\fc)) \cap J_{y,\geq k})J_{y,\geq k+1} \to \mathrm{Hom}(k(y), k_{y})\]
is an isomorphism.\end{lemma}
\begin{proof}
Let $\alpha_{\geq k} \in J_{y,\geq k}$. By lemma \ref{relativekgpcurves}, $\alpha_{\geq k}$ is uniquely determined modulo $J_{y,\geq k+1}$ by $\phionek$ and $\phitwok \in \mathbb{A}_{k(y)}$.\\
Further, $\phionek$ contains no $p$-powers and $\phitwok$ contains only $p$-powers, so the pairing becomes a pairing on the groups:
\[ \left(\mathbb{A}_{k(y)}/\mathbb{A}_{k(y)}^{p} \oplus \mathbb{A}_{k(y)}\right) \times k( y) \to k_{y}\]
which maps $(\phionek, \phitwok, f_{-k})$ to \[\sum_{x \in y} \text{Tr}_{k(x)/k_{y}}\left( \mathrm{res}_{x}\left( f_{-k}\left( d\phionek - k\phitwok \frac{d\uxc}{\uxc}\right)\right)\right)\]
by corollary \ref{firstcalccor}.

This reduces us to the classical one-dimensional case. By \cite{We}, chapter IV 2.3, the pairing
\[\mathbb{A}_{k(y)} \times \mathbb{A}_{k(y)}(\Omega^{1}_{y}) \to k_{y}\]
mapping $(f_{x}, \omega_{x})$ to $\sum_{x \in y}\text{Tr}_{k_{y}(x)/k_{y}}(\mathrm{res}_{x}(f_{x}\omega_{x}))$ is a continuous non-degenerate pairing of vector spaces such that $k(y)^{\perp} = \Omega_{k(y)}^{1}$, where $\mathbb{A}_{k(y)}(\Omega^{1}_{y})$ is defined to be
\[\left\{ (\omega_{x})_{x \in y} \in \prod_{x \in y} \Omega^{1}_{k(y)_{x}/k_{y}} : v_{x}(\omega_{x}) \geq 0 \text{ for all but finitely many } x \in y \right\}.\]
This reduces us to showing the map
\[ \mathbb{A}_{k(y)}/\mathbb{A}_{k(y)}^{p} \oplus \mathbb{A}_{k(y)}^{p} \to \mathbb{A}_{k(y)}(\Omega_{y}^{1})/\Omega^{1}_{k(y)}\]
sending $(\phionek, \phitwok)$ to  $d\phionek - k\phitwok \frac{d\uxc}{\uxc}$ is a surjection, with kernel the diagonal elements as characterised in lemma \ref{diagonalelements}.\\
Let $\omega \in \Omega_{k(y)_{x}}^{1} \subset \mathbb{A}_{k(y)}(\Omega_{y}^{1})$.  Then $\omega$ decomposes as $P(\uxc)d\uxc$ for some  $P(\uxc) \in k_{y}(x)((\uxc))$ as $\Omega_{k(y)_{x}}^{1}$ is a $k_{y}(x)$-module generated by $du_{x,y}$.
It is clear this decomposition can be rewritten in the required form for suitable $\phionek$, $\phitwok$, so the map is surjective.\\
For the kernel, let $\phionek$, $\phitwok \in \mathbb{A}_{k(y )}$ and suppose $\omega = (\omega_{x})_{x \in y} \in \Omega_{k(y)}^{1}$ is such that
\[\omega_{x} = d\phionek - k\phitwok\frac{d\uxc}{\uxc}\]
for each $x \in y$.\\
As $\Omega_{k(y)}^{1}$ is a rank one $k(y)$-module, we may choose $a, b \in k(y)$ such that
\[ \omega = -ka\frac{db}{b}.\]
As in lemma \ref{diagonalelements}, define $\eta(a,b) \in \mathbb{A}_{k(y)}$ uniquely by 
\[ a\uxc b^{-1}\frac{db}{d\uxc} + \uxc \frac{\eta_{x}(a,b)}{d\uxc} \in k(y)^{p}_{x}\] for each $x \in y$.
Then for each point $x$:
\[d(k\eta(a,b)) - k\left( a\uxc b^{-1}\frac{db}{d\uxc} + \uxc \frac{\eta_{x}(a,b)}{d\uxc}\right) \frac{d\uxc}{\uxc}\]
\[ \ \ \ \ \ \ \ \ \ \ \ \ \ \ \ \ \ \ \ \ \ \ \ \ \  = d(k\eta(a,b)) - kab^{-1}db - kd\eta(a,b) = \omega.\]
But then by the uniqueness of the decomposition of $\omega$, we have
\[ \phionek = k\eta_{x}\]
and \[ \phitwok = a\uxc b^{-1}\frac{db}{d\uxc} + \uxc \frac{d\eta_{x}}{d\uxc}\]
 for each $x \in y$, as required.\\
 The surjection Hom$_{cont}(\mathbb{A}_{k(y)}, k_{y}) \to$ Hom$(k(y), k_{y})$ combined with the induced map $\mathbb{A}_{k(y)}(\Omega^{1}_{y}) \to $ Hom$(\mathbb{A}_{k(y)}, k_{y})$ from the pairing above proves the required isomorphism:
 \[ J_{\geq k}/(\Delta(K_{2}^{top}(\fc))\cap J_{\geq k})J_{\geq k +1} \to \mathbb{A}_{k(y)}/\mathbb{A}_{k(y)}^{p} \oplus \mathbb{A}_{k(y)}^{p}\]
 \[\ \ \ \ \ \ \ \ \ \ \ \ \ \ \ \ \ \ \ \ \ \ \ \ \ \ \ \ \ \ \ \ \ \ \ \ \ \ \ \ \ \ \ \ \ \ \ \ \ \ \ \ \ \ \ \ \ \ \ \ \ \  \to \mathbb{A}_{k(y)}(\Omega_{y}^{1}) \to \text{Hom}(k(y), k_{y}).\]
\end{proof}

The following lemma is similar to the above, but considers the case $p|k$.

\begin{lemma}\label{dualitytwo}
Fix $k \geq 1$ with $ p | k$.  Then the pairing
\[J_{y,\geq k}/(\Delta(K_{2}^{top}(\fc)) \cap J_{y,\geq k})J_{\geq k+1} \times \mathcal{R}_{p}(k(y)) \to k_{y}\]
mapping $(\alpha_{k}, f_{-k})$ to $(\alpha_{k}|f_{-k}\tc^{-k}]_{y}$ is non-degenerate.  The induced homomorphism
\[\frac{J_{y, \geq k}}{(\Delta(K_{2}^{top}(\fc) \cap J_{y, \geq k})J_{y, \geq k+1}} \to \mathrm{Hom}\left(\frac{k(y)}{k(y)^{p}}, k_{y}\right)\]
is an isomorphism. \end{lemma}
\begin{proof}
Let $\alpha_{\geq k} \in J_{y, \geq k}$ be uniquely determined up to $J_{y, \geq k +1}$ by $\phionek$ and $\phitwok$.  From lemma \ref{relativekgpcurves}, we know $\phitwok = 0$ and $\phionek$ contains no $p$-powers.  Hence there is an isomorphism
\[J_{y, \geq k}/J_{y, \geq k + 1} \to \mathbb{A}_{k(y)}/\mathbb{A}_{k(y)}^{p}\]
mapping $\alpha_{\geq k}$ to $(\phionek \text{ mod } k(y)_{x}^{p})_{x \in y}$.\\
Lemmas \ref{structuremodfrob} and \ref{firstcalculation}  show it is enough to prove that the pairing
\[\frac{\mathbb{A}_{k(y)}}{(\mathbb{A}_{k(y)}^{p} + k(y))} \times \frac{k(y)}{k(y)^{p}} \to k_{y}\]
sending $((\phionek)_{x \in y}, f_{-k})$ to $ \sum_{x \in y}$Tr$_{k(x)/k_{y}}(\mathrm{res}_{x}(f_{-k}d\phionek))$ is non-degenerate and induces an isomorphism
\[ \frac{\mathbb{A}_{k(y)}}{(\mathbb{A}_{k(y)}^{p} + k(y))} \to \text{Hom}\left(\frac{k(y)}{k(y)^{p}}, k_{y}\right).\]
Fix $s \in k(y)$ and suppose $\sum_{x \in y}$Tr$_{k(x)/k_{y}}(\mathrm{res}_{x}(r_{x}ds)) = 0 $ for all $(r_{x}) \in \mathbb{A}_{k(y)}$.  Letting $\omega = ds \in \Omega^{1}_{y}$, from the non-degenerate pairing in the lemma above we see $\omega = 0$, i.e. $s \in k(y)^{p}$.\\
Fix $(r_{x})_{x \in y} \in \mathbb{A}_{k(y)}$ and suppose $\sum_{x \in y}$Tr$_{k(x)/k_{y}}(\mathrm{res}_{x}(r_{x}ds)) = 0 $ for all $ s \in k(y)$.  Then:
\[\sum_{x \in y}\text{Tr}_{k(x)/k_{y}}(\mathrm{res}_{x}(r_{x}ds)) = \sum_{x \in y}\text{Tr}_{k(x)/k_{y}}(\mathrm{res}_{x}(d(r_{x}s) - sdr_{x}))\]
\[ \ \ \ \ \ \ \ \ \ \ \ \ \ \ \ \ \ \ \ \ \ \ \ \ \ \ \ \ \ \ \ \ \ \ \ \ \ \ \ \ \ \ \ \ \ \ \ \ \ \ \ \ \ \ \ \ \ \ \ \ \ \  = -\sum_{x \in y}\text{Tr}_{k(x)/k_{y}}(res_{x}(sdr_{x})).\]
As $k(y)^{\perp} = \Omega_{k(y)}^{1}$ with respect to the pairing in the lemma above, we have $(dr_{x}) \in \Omega^{1}_{k(y)}$.  Then the  commutative diagram with exact rows:
\[ \begin{CD}
0 @>>> k(y)/k(y)^{p} @>d>> \Omega^{1}_{k(y)} @>>> \Omega^{1}_{k(y)} @>>> 0 \\
@.  @VVV @VVV @VVV @. \\
0 @>>> \mathbb{A}_{k(y)}/\mathbb{A}_{k(y)}^{p} @>d>> \mathbb{A}_{k(y)}(\Omega^{1}_{y}) @>>> \mathbb{A}_{k(y)}(\Omega^{1}_{y}) @>>> 0 \end{CD}\]
implies $(r_{x}) \in k(y) + \mathbb{A}_{k(y)}^{p}$ as required.\\
Finally the continuous injections of $k_{y}$ vector spaces 
\[k(y)/k(y)^{p} \hookrightarrow \Omega^{1}_{k(y)} \hookrightarrow \mathbb{A}_{k(y)}(\Omega^{1}_{y})\]
induce
\[\mathbb{A}_{k(y)} \to \text{Hom}_{cont}(\mathbb{A}_{k(y)}(\Omega^{1}_{y}), k_{y}) \twoheadrightarrow \text{Hom}(\Omega^{1}_{k(y)}, k_{y}) \twoheadrightarrow \text{Hom}(k(y)/k(y)^{p}, k_{y})\]
where the first map is an isomorphism.
\end{proof}

\bigskip
\noindent We can now use these final two lemmas to prove theorem \ref{mainthmcurves}.
\begin{proof}
\emph{m = 1:}\\
Firstly we prove the pairing is non-degenerate in the second argument.  Let $ f = \sum_{k < 0} f_{k}\tc^{k}$ be a representative of $\fc/((Frob - 1)\fc+ k(y))$ and assume $(\alpha|f]_{y} = 0$ for all $\alpha \in J_{y}$.  Assuming $f \neq 0$, let $l$ be the least index with $f_{l} \neq 0$, and let $\alpha_{-l} \in J_{y, -l}$.  Then
\[ 0 = (\alpha_{-l}|f]_{y} = (\alpha_{-l} | f_{l}\tc^{l}]_{y}\] and lemmas \ref{dualityone} and \ref{dualitytwo} show $f_{l} = 0$, a contradiction.\\
We now prove the map to the homomorphism group is an isomorphism, which will also prove non-degeneracy in the first argument.  Let
\[ \mu : \fc/((\mathrm{Frob} - 1)\fc+ k( y)) \to \mathbb{Z}/p\mathbb{Z}\] be a homomorphism.
By lemma \ref{structuremodfrob}, $\mu$ can be described by a family of continuous maps
\[\mu_{k} : k(y) \to \mathbb{Z}/p\mathbb{Z}\] sending $f_{k}$ to $\mu(f_{k}\tc^{k})$ for each $k < 0$.\\
We will inductively construct an $\alpha \in J_{y}/J_{y}^{p}$ such that 
\[(\alpha | f]_{y} = \mu(f)\]
for all $ f \in \fc/(Frob - 1)\fc .k_{y}$, and such that $\alpha$ is unique up to $\Delta(K_{2}^{top}(\fc))\cap J_{y}$.\\
For any $\alpha \in J_{y}/J_{y}^{p}$, let $\alpha_{\geq 1} \in J_{y, \geq 1}$ be the element defined by $\phionek$ and $\phitwok$ for $k \geq 1$ in the expansion of $\alpha$, and $\alpha_{1}$ the element defined by $\phi^{(1)}_{1,x,y}(u_{x,y})$ and $\phi^{(2)}_{1,x,y}(u_{x,y})$.  Inductively, define
\[\alpha_{\geq j} = \alpha_{j}\alpha_{\geq j + 1}.\]
By corollary \ref{firstcalccor}, 
\[ (\alpha | f_{-k}\tc^{-k}]_{y} = \sum_{1 \leq j \leq k} (\alpha_{j} | f_{-k}\tc^{-k}]_{y}\]
so  $(\alpha | f]_{y} = \mu(f)$ for all $f \in \fc/((Frob - 1)\fc +k_{y})$ if and only if \[(\alpha_{k} | f_{-k}\tc^{-k}]_{y} = \mu_{k}(f_{-k}) - \sum_{1 \leq j \leq k}(\alpha_{j} | f_{-k}\tc^{-k}]_{y}\]
for all $k \geq 1$.\\
By lemma \ref{dualityone} and \ref{dualitytwo}, there exists such an $\alpha_{k}$, uniquely defined up to $(J_{y, \geq k } \cap \Delta(K_{2}^{top}(\fc)))J_{y, \geq k+1}$ for each $k$.  Letting $\alpha = \prod_{k \geq 1}\alpha_{k}$, we obtain the required element.\\

\noindent \emph{Induction}\\

\noindent Suppose we have
\[ \frac{J_{y}}{(\Delta(K_{2}^{top}(\fc))\cap J_{y})J_{y}^{p^{m}}} \cong \text{Hom}\left( \frac{W_{m}(\fc)}{(\mathrm{Frob} - 1)W_{m}(\fc)+W_{m}(k( y))}, \mathbb{Z}/p^{m}\mathbb{Z}\right)\] for some $m \in \mathbb{Z}$.\\
Let $\mu \in\text{Hom}\left( W_{m+1}(\fc)/((\mathrm{Frob} - 1)W_{m+1}(\fc)+W_{m+1}(k( y))), \mathbb{Z}/p^{m+1}\mathbb{Z}\right)$.  Then if
\[ \mu' : W_{m}(\fc)/((\mathrm{Frob} - 1)W_{m}(\fc)+W_{m}(k( y))) \to \mathbb{Z}/p^{m}\mathbb{Z}\]
is the map sending $(f_{0}, \dots, f_{m-1})$ to $V(\mu(f_{0}, \dots, f_{m-1}, 0)$, where $V$ is the usual map in Witt theory $(x_{0}, x_{1} \dots) \mapsto (0, x_{0}, x_{1},\dots)$, then $\mu'$ is a homomorphism.  \\
So we can associate $\alpha \in J_{y}/(\Delta(K_{2}^{top}(\fc))\cap J_{y})J_{y}^{p^{m}}$ to $\mu'$, i.e.:
\[ (\alpha | f_{0}, \dots f_{m-1}]_{y} = V(\mu (f_{0}, \dots, f_{m-1}, 0).\]
Now, for $(0, \dots, 0, f_{m}) \in W_{m+1}(\fc)/((\mathrm{Frob} - 1)W_{m+1}(\fc)+W_{m+1}(k(y )))$, we have 
\[ (\alpha | 0, \dots, 0, f_{m}]_{y} = (0, \dots, 0, (\alpha | f_{m}]_{y}) \in W_{m+1}(\mathbb{F}_{p}).\]
If we consider the Witt vector $(0, \dots, 0, f_{m}) \in \frac{W_{m}(\fc)}{(\mathrm{Frob} - 1)W_{m}(\fc)+W_{m}(k( y))}$, then we see
\[(\alpha | 0, \dots, f_{m}]_{y} = (0, \dots, 0, (\alpha | f_{m}]_{y})\] in $W_{m}(k_{y})$.
But also
\[ (\alpha | 0, \dots, 0, f_{m}]_{y} = V(\mu(0, \dots, 0, f_{m}, 0)) = \mu(0, \dots, 0, f_{m})\]
(in $W_{m+1}(F_{y})$), as $V$ commutes with any homomorphism of Witt vectors.
This gives: 
\[ (\alpha | f_{0}, \dots, f_{m-1}, f_{m}]_{y} = (\alpha | f_{0}, \dots, f_{m-1}, 0]_{y} + (\alpha | 0, \dots, 0, f_{m}]_{y} 
\]
\[= \mu(y_{0}, \dots, f_{m-1}, 0) + \mu(0, \dots, 0, f_{m}) = \mu(f_{0}, \dots, f_{m-1}, f_{m})\]
as required.  To see that $\alpha$ is unique up to $J^{p^{m+1}}$, use proposition \ref{wittproperties} $iv$.\end{proof}

For a singular curve $y \subset X$, we see that the above theorem holds for each irreducible component $y_{i} \subset y$, as the fields $F_{x,y}$ and $F_{y}$ depend only on the normalisation of $y$. So the theorem is also true for the products $J_{y}$,  $F_{y}$ and $k_{y}$ in this case.

\bigskip

We next prove a similar duality theorem for the higher tame symbol, enabling us to define the tamely ramified part of the reciprocity map. Our ultimate aim is the following theorem.

\begin{thm}\label{tamecurvethm}
Fix a non-singular curve $y \subset X$ and define $\mathfrak{J}_{y}$ to be the ring generated by the subgroup of the $K$-groups \[{\prod}'_{x \in y} \{k_{y}(x), u_{x,y}\} \times \{k_{y}(x), t_{y}\} \times \{u_{x,y}, t_{y}\},\] i.e. the elements of $\mathcal{J}_{y}$ with either one entry in the constant field and one entry a local parameter for either  $y$ or some $x\in y$, or both entries the local parameters for $y$ and some $x \in y$. The restricted product is because $\mathfrak{J}_{y}$ is a subgroup of  $\mathcal{J}_{y}$. Then the global higher tame pairing on 
\[\mathfrak{J}_{y}/(\Delta(K_{2}^{top}(F_{y}))\cap \mathfrak{J}_{y})\mathfrak{J}_{y}^{q-1} \times F_{y}^{\times}/(F_{y}^{\times})^{q-1} \to \mathbb{F}_{q}^{\times}\]

\noindent is continuous and non-degenerate. 
\end{thm}

By the reciprocity law \ref{tamereciprocity}, we know that the intersection with $K_{2}^{top}(F_{y})$ is contained in the kernel of the left hand side of the pairing. Proceeding in a similar way to the method used for the Witt pairing, we will look at the structure of the group on the left hand side and prove non-degeneracy by a combinatorial argument. It may be noted that the higher tame pairing requires only linear algebraic methods to understand, and so the argument will be much more simple than in the case of the Witt pairing, as the $p$-part of the reciprocity map is harder to understand.\\

We briefly recall Parshin's argument from \cite[3.1.3]{P4}, that is, the proof of duality for a single higher local field. In our language, we fix a point $x$ and just discuss the case of a two-dimensional local field. Fix a $(q-1)^{th}$ root of unity $\zeta \in F_{x,y}$ so that the left hand side of the pairing is generated by the elements
\[\{\zeta, u_{x,y}\}, \ \{\zeta, t_{y}\} \ \text{and} \ \{u_{x,z}, t_{y}\}.\]
The higher tame pairing takes non-zero values when the above elements are paired with
\[ t_{y}, u_{x,y} \ \text{and} \ \zeta\] respectively. But these three elements generate the group $F_{x,y}^{\times}/(F_{x,y}^{\times})^{q-1}$, which in the local case is the right hand side of the pairing, as required. 

\begin{lemma}
Fix a $(q-1)^{th}$ root of unity $ \zeta \in F_{y}$. Then the group $F_{y}^{\times}/(F_{y}^{\times})^{q-1}$ is generated by the elements
\[ \zeta, \ t_{y}, \ \text{and for each } x \in y, \ u_{x,y}.\]
\end{lemma}
\begin{proof}
It is well known that the first residue field, isomorphic to a one-dimensional global field $\mathbb{F}_{q}(u)$, has multiplicative group generated by $\zeta$ and the primes of the field. These primes are in bijective correspondence with the points $x \in y$ and can be represented by the equations $u_{x,y} \in F_{y}$. 

Then under the isomorphism $F_{y} \cong \mathbb{F}_{q}(u)((t_{y}))$, it is clear from the theory of complete fields that we need only to include $t_{y}$ to generate the multiplicative group $F_{y}^{\times}$. All of these elements have order $q-1$ in the quotient group, so they also generate $F_{y}^{\times}/(F_{y}^{\times})^{q-1}$. 
\end{proof}

Now we examine the structure of the elements of the groups $\mathfrak{J}_{y}$ and $\mathfrak{J}_{y}/$ $(\Delta(K_{2}^{top}(F_{y})) \cap \mathfrak{J}_{y})\mathfrak{J}_{y}^{q-1}$ which will give non-zero values when paired with elements of the form in the above lemma.\\
The non-degeneracy on the right hand side of the pairing with $\mathfrak{J}_{y}$ is easy to see. Pairing:
\begin{enumerate} \item{the root of unity $\zeta$ with an element of $\mathfrak{J}_{y}$ with $\{u_{x,y},t_{y}\}$ in the $x$-position and trivial everywhere else;} 
\item{the local parameter $t_{y}$ with an element with $\{\zeta, u_{x,y}\}$ at the $x$-position and trivial everywhere else;}
 \item{any parameter $u_{x,y}$ with the element with $\{\zeta, t_{y}\}$ at the $x$-position and trivial everywhere else}\end{enumerate}
  all yield non-zero values. We must check that these remain non-zero when we quotient by the diagonal elements, and prove non-degeneracy on the left hand side.
  
In the following lemma, and throughout the rest of the section, we will refer to the ``point at infinity". By this, we mean the following: let $k(y)$ be isomorphic to the field $\mathbb{F}_{q}(u)$. Then the element $1/u$ must be considered as a prime element, and taken into account when we prove reciprocity laws. We will refer to this point of the curve $y$ as the point at infinity. See \cite{We} for more details of this definition from the classical theory.  
  
\begin{lemma}
Let $\alpha \in \Delta (K_{2}^{top}(F_{y}))\cap \mathfrak{J}_{y}$. Then $\alpha$ is a product of elements of the form:
\begin{enumerate}
\item{$\{u_{x,y}, t_{y}\}$ in the $x$-position and the position corresponding to the point at infinity, ones everywhere else;}
\item{$\{\zeta, u_{x,y}\}$ in the $x$-position and the position corresponding to the point at infinity, ones everywhere else;}
\item{$\{\zeta, t_{y}\}$ in every position. }
\end{enumerate}
\end{lemma}
  
\begin{proof}
Lemma \ref{kgpstructurecurves} gives us four types of elements that appear in $K_{2}^{top}(F_{y})$, but the elements of types 3 and 4 contain principal units and hence are not in the group when intersected with $\mathfrak{J}_{y}$. So we are left with the diagonal embeddings of elements of types 1 and 2, i.e. $\{a, t_{y}\}$ and $\{a,b\}$ where $a$ and $b$ are lifts of elements in $k(y)^{\times}$. 

We can decompose the elements of $k(y)^{\times}\cong k_{y}(u)^{\times}$ as products of elements of $k_{y}^{\times}$ and primes which may be represented as parameters $u_{x,y}$. Then we can restrict to elements of type $\{\zeta, t_{y}\}$ and $\{u_{x,y}, t_{y}\}$ from the first type of element, and $\{\zeta, u_{x,y}\}$ from the second type - we know that $K_{2}^{top}(k_{y}) = 0$ and $\{u_{x,y}, u_{x,y}\} = \{-1, u_{x,y}\}$, so these are the only elements of the second type.

So we now investigate the diagonal elements of each of these types of elements. The elements of $\mathfrak{J}_{y}$ with $\{\zeta, t_{y}\}$ at every place cannot be reduced into a more simple form, so this is the third type of element in the list above.

The elements of type $\{\zeta, u_{x,y}\}$ will take this form at the $x$-position and the point at infinity, but at other positions the parameter $u_{x,y}$ can be decomposed as a product of principal units and elements of $k_{y}$, as it is not a prime at these points. But these elements are all either trivial in the topological $K$-group or not in the intersection with $\mathfrak{J}_{y}$, so we are left with an element of type 2.

The elements of type $\{u_{x,y}, t_{y}\}$ will take this form at the $x$-position and the point at infinity. At the other positions, $u_{x,y}$ is not a prime and so can be decomposed as a product of principal units and elements of $k_{y}$.  We can renormalise the other local parameters $u_{x',y}$, where $x \neq x'$, so that $u_{x,y}$ decomposes just as a principal unit in each place. Then when we take the intersection with $\mathfrak{J}_{y}$, these elements are all trivial and only the element of type one remains.
\end{proof}

We use this simple form of elements in the diagonal embedding of $K_{2}^{top}(F_{y})$ to study the elements in the quotient.

\begin{lemma}
Let $\alpha \in \mathfrak{J}_{y}/(\Delta(K_{2}^{top}(F_{y}))\cap \mathfrak{J}_{y})\mathfrak{J}_{y}^{q-1}$. Then $\alpha$ can be written as a finite product of elements of the form:
\begin{enumerate}
\item{$\{u_{x,z},t_{y}\}$ in the position corresponding to the point at infinity, for some $x \in y$, and trivial in every other position;}

\item{$\{a, u_{x,y}\}$ in the position corresponding to the point at infinity, for some $x \in y$ and $a \in k_{y}$, and trivial in every other position;}

\item{An element $\{a_{x},t_{y}\}$ in any $x$-position, where $a_{x} \in k_{y}$.}

\end{enumerate}
\end{lemma}

\begin{proof}
From our definition of $\mathfrak{J}_{y}$ in \ref{tamecurvethm}, we know any element of $\mathfrak{J}_{y}$ is a product of elements $\{u_{x,y}, t_{y}\}$, $\{a, u_{x,y}\}$ and $\{a, t_{y}\}$ where $x$ runs through the closed points of  $y$. We prove that when quotienting by the diagonal elements, these generators take the above form. 

Firstly, for an element with $\{u_{x,y}, t_{y}\}$ in the $x$-position, we multiply by the element of $\Delta(K_{2}^{top}(F_{y}))$ with $\{u_{x,z}^{q-2},t_{y}\}$ in the $x$-position and the point at infinity to obtain an element which is a product of those of type one in the lemma.

For an element with $\{b, u_{x,y}\}$ in the $x$-position, where $b \in k_{y}(x)$, we must do slightly more work. If $b \in k_{y}$, similarly to the above we can just multiply by the element with $\{b, u_{x,y}^{q-2}\}$ in the $x$-position and the point at infinity to obtain an element which is a product of elements of type two. 
But if $b \in k_{y}(x)\setminus k_{y}$ this will not work.

We know that the higher tame symbol will take the value $N_{k_{y}(x)/k_{y}}(b)$ when the element $\{b, u_{x,y}\}$ is paired with $t_{y}$ - and that this is true for all conjugates of $b$ in the extension $k_{y}(x)/k_{y}$. So if we can show that $b$ is equivalent to its norm in the quotient, then we may replace $b$ with $N_{k_{y}(x)/k_{y}}(b) \in k_{y}$ and proceed as above.

Let $k_{y}$ have size $q$ and $k_{y}(x)$ have size $q^{n}$. Then the Galois group of the extension $k_{y}(x)/k_{y}$ is generated by the map $\alpha \mapsto \alpha^{q}$. So our element $b$ differs from each of its conjugates $b^{q}, b^{q^{2}}, \dots$ by a $(q-1)^{th}$-power. Hence it differs from its norm $N_{k_{y}(x)/k_{y}}(b) = bb^{q}b^{q^{2}}\dots$, a product of $n$ terms, by a $(q-1)^{th}$-power also. So these elements are all products of elements of the second type.

For an element made up of symbols $\{b_{x}, t_{y}\}$ in each $x$-position, $b_{x} \in k_{y}(x)$, we may use the above method to show $b_{x}$ is equivalent to $a_{x}= N_{k_{y}(x)/k_{y}}(b_{x}) \in k_{y}$. Then an element containing entries only of this type is an element of the third type, as required.

The product is finite because of the intersection with the adelic group.
\end{proof}

We can now complete the proof of theorem \ref{tamecurvethm}. Following Parshin's local approach detailed above, we provide each of the generators of $F_{y}^{\times}/(F_{y}^{\times})^{q-1}$ with exactly one of the generators of $\mathfrak{J}_{y}/(\Delta(K_{2}^{top}(F_{y}))\cap \mathfrak{J}_{y})\mathfrak{J}_{y}^{q-1}$ with which it has a non-zero value when the higher tame pairing is applied. 

Let $\zeta$ be a primitive $(q-1)^{th}$ root of unity in $F_{y}$, i.e. a lift of a generator of $k_{y}$. Then as above, $F_{y}$ is generated by $\zeta$, $t_{y}$ and a local parameter $u_{x,y}$ for each $x \in y$. 

We pair $\zeta$ with the element of $\mathfrak{J}_{y}/(\Delta(K_{2}^{top}(F_{y}))\cap \mathfrak{J}_{y})\mathfrak{J}_{y}^{q-1}$ with $\{u_{x,y},t_{y}\}$ at the position corresponding to the point at infinity for some $x \in y$, and trivial everywhere else. This is independent of the choice of $x \in y$: the value of the tame pairing depends on the valuation of $u_{x,y}$, which here is related to the degree of the polynomial $u_{x,y}$ when written as a polynomial in a fixed variable $u$. So the non-trivial case is when $u_{x,y}$ has degree greater than one, which coincides with the case $k_{y}(x) \neq k_{y}$.

Let $v_{x,y}$ be the linear factor of $u_{x,y}$ corresponding to the valuation on $F_{x,y}$. Now, as discussed in the above proof, in the $x$-position the element $u_{x,y} = N_{k_{y}(x)/k_{y}}(v_{x,y})$ differs from the factor $v_{x,y}$ by a $(q-1)^{th}$-power. So again, modulo the power of $(q-1)$, the pairings have the same value in this position, and hence also when shifted to the point at infinity.

We pair the parameter $t_{y}$ with the element of $\mathfrak{J}_{y}/(\Delta(K_{2}^{top}(F_{y}))\cap \mathfrak{J}_{y})\mathfrak{J}_{y}^{q-1}$ with $\{\zeta, u_{x,y}\}$ at the point at infinity for some $x \in y$ and trivial everywhere else. As above, this is independent of our choice of $x \in y$.

Finally, for each $x \in y$, we pair the parameter $u_{x,y}$ with the element of $\mathfrak{J}_{y}/(\Delta(K_{2}^{top}(F_{y}))\cap \mathfrak{J}_{y})\mathfrak{J}_{y}^{q-1}$ with $\{\zeta, t_{y}\}$ in the $x$-position and trivial everywhere else.

For a singular curve $y \subset X$, we may use the above construction for each irreducible component $y_{i}$ of $y$ which will induce a non-degenerate pairing on the groups $F_{y}^{\times}$ and $\mathfrak{J}_{y}/(\Delta(K_{2}^{top}(F_{y}))\cap\mathfrak{J}_{y})\mathfrak{J}_{y}^{q-1}$, both direct sums over $y_{i} \subset y$, where the case for $y_{i}$ again follows as the fields involved depend only on the normalisation of the curves.

\bigskip
We will now construct the reciprocity map for a product of complete discrete valuation fields over a global field, associated to a curve on an arithmetic surface.  Our method uses only basic Galois theory and decomposition groups.  Let $L/F_{y}$ be a finite Galois extension with Galois group $G$ - then $L$ will also be a product of complete discrete valuation fields over a global field. The extension of residue fields $\bar{L}/k(y)$ determines a finite morphism of curves $\pi: y' \to y$, where $y' = y \times_{F} L$ and  $k(y') = \bar{L}$. For each point $x \in y$, we have the decomposition
\[L \otimes_{F_{y}} F_{x,y} = \oplus_{x' \in y', \pi(x') = x}L_{x', y'}\]
where the $L_{x',y'}$ are products of two-dimensional local fields. Each term in the product is a finite extensions of $F_{x,z}$, where $z$ is a branch of $y$ passing through $x$.   \\

\noindent For each $x' \in y'$ with $\pi(x')=x$, define the decomposition group
\[G_{x'}=\{g \in G: g(x')=x'\}.\]
For another $x''$ such that $\pi(x'')=x$, the groups $G_{x'}$ and $G_{x''}$ are conjugate in $G$. We have $G_{x'} \cong$ Gal$(L_{x',y'}/F_{x,y})$.\\
Now let $L/F_{y}$ be an abelian extension. Then the homomorphism 
\[\text{Gal}(L_{x',y'}/F_{x,y}) \cong G_{x'} \to G = \text{Gal}(L/F_{y})\]
is independent of the choice of $x'$. \\
So the product of the higher tame and Witt symbols
\[{\prod}'_{x \in y}K_{2}^{top}(F_{x,y}) \to \text{Gal}(F_{y}^{ab}/F_{y})\] is well-defined. 

In addition, we must define the unramified part of the reciprocity map. The unramified closure of the field $F$ is the field generated by $F$ and $\bar{\mathbb{F}}_{q}$, and its Galois group is canonically isomorphic to $\hat{\mathbb{Z}}$, generated by the Frobenius automorphism of $\bar{\mathbb{F}}_{q}$, Frob.

\begin{defn}\label{unrammap}
Let $\delta: K_{2}^{top}(F_{x,y}) \to K_{1}^{top}(\bar{F}_{x,y})$ be the boundary homomorphism of $K$-theory. We define the map
\[Un_{x,y}: K_{2}^{top}(F_{x,y}) \to \hat{\mathbb{Z}}\]
by 
\[\{\alpha,\beta\} \mapsto \mathrm{Frob}^{v_{\bar{F}_{x,y}}(\delta(\{\alpha,\beta\})},\]
where $v_{\bar{F}_{x,y}}$ is the valuation map of the local field $\bar{F}_{x,y}$.
\end{defn}

\noindent We define $Un_{y}$ to be the product of the $Un_{x,y}$. Note that this product is well-defined on the adelic group ${\prod}'_{x \in y}K_{2}^{top}(F_{x,y})$, as for all but finitely many $x \in y$, the component $\{\alpha_{x,y},\beta_{x,y}\}$ is in $K_{2}^{top}(\mathcal{O}_{x,y})$ and the value of $\delta(\{\alpha_{x,y},\beta_{x,y}\})$ is $1$. 

The unramified part of the map also obeys the reciprocity law: it follows straight from the reciprocity law for $k(y)$. So we may define the product of all three maps
\[{\prod}'_{x \in y}K_{2}^{top}(F_{x,y}) \to \text{Gal}(F_{y}^{ab}/F_{y}).\]

\noindent Define
\[\psi_{y}: {\prod}'_{x \in y} K_{2}^{top}(F_{x,y}) \to \text{Gal}(L/F_{y})
\]
as the product of the $\phi_{x,y}(L)$.
\begin{lemma}\label{ctshomocurves}
Let $L/F_{y}$ be a finite abelian extension. Then for almost all $x \in y$, we have $\phi_{x, y}(L)=1$, and hence $\phi_{y}$ is a continuous homomorphism. \end{lemma}
\begin{proof}
By \cite{P4} section four, it is sufficient to prove the lemma in the three cases $L =F_{y}(\gamma)$, $L/F_{y}$ an Artin-Schreier extension with $\gamma^{p} - \gamma = \alpha$ for some $\alpha \in F_{y}$, $L = F_{y}(\beta)$ is a Kummer extension where $\beta^{l} = \delta$ for some $l|q-1$ and $\delta \in F_{y}$, and $L/F_{y}$ is unramified.

This is sufficient as the abelian closure, $F_{y}^{ab}/F_{y}$ is generated by the maximal unramified extension, the maximal ramified and prime to $p$ extension, and the maximal $p$-extension. These three types of extension are disjoint, except for the unramified $p$-extension, where the maps are compatible.

For the first case, the local residue symbol is described by the relation
\[\phi_{x,y}(w_{x, y})(z) = (w_{x,y}|\alpha]_{x,y}(z)
\]
 for $w_{x,y} \in K_{2}^{top}(F_{x,y})$ and we know this is zero for almost all $x \in y$ from lemma \ref{wittwelldefined} above.

For the Kummer extension, the local residue symbol is described by the relation
\[\phi_{x,y}(w_{x,y})(z) = (w_{x,y}, \delta)_{x,y}\]
and similarly we know this is trivial for almost all $x \in y$ by lemma \ref{tamewlldefined}.
The comment below definition \ref{unrammap} proves the lemma in the unramified case.

The continuity of the reciprocity map follows, as the preimage of any open subgroup of Gal$(F_{y}^{ab}/F_{y})$ has only finitely many non-zero elements of $J_{y}$. But from the definition of the topology, this is exactly what is required in the direct sum and product topology.
\end{proof}

From \cite{P4} section four, we see that these maps are compatible for different abelian extensions $L/F_{y}$, so we have a continuous homomorphism 
\[\phi_{y} : {\prod}'_{x \in y}K_{2}^{top}(F_{x,y}) \to \text{Gal}(F_{y}^{ab}/F_{y}).
\]

We now prove the main theorem of this section. Recall that the restricted product of the groups $K_{2}^{top}(\mathcal{O}_{x,y})$ is given by the intersection of the product with the adelic group $\mathbb{A}_{X}$.
\begin{thm}\label{cftcurves}
Let $X/\mathbb{F}_{q}$ be a regular projective surface and $y \subset X$ an irreducible curve. Then the continuous map
\[\phi_{y} : {\prod}'_{x \in y}K_{2}^{top}(\mathcal{O}_{x,y}) \to \text{Gal}(F_{y}^{ab}/F_{y}) 
\]
is injective with dense image and satisfies:
\begin{enumerate}
\item{$\phi_{y}$ depends only on $F_{y}$, not on the choice of model of $X$;}
\item{For any finite abelian extension $L/F_{y}$, the following sequence is exact:
\[\begin{CD}@.
\frac{\prod'_{x' \in y', \pi(x')=x}K_{2}^{top}(L_{x'})}{\Delta(K_{2}^{top}(L))\cap \prod'_{x' \in y', \pi(x')=x}K_{2}^{top}(L_{x'})} @>N>>   @. \\
@. \mathcal{J}_{y}/\Delta(K_{2}^{top}(F_{y}))\cap \mathcal{J}_{y} @>\phi_{y}>> 
@. \text{Gal}(L/F_{y})@>>> 0.
\end{CD}
\]}
\item{For any finite separable extension $L/F_{y}$, the following diagrams commute:
\[\begin{CD}
\mathcal{J}_{L}/\Delta(K_{2}^{top}(L)) @>\phi_{L}>> \text{Gal}(L^{ab}/L)\\
@AAA @A V AA\\
\mathcal{J}_{y}/\Delta(K_{2}^{top}(F_{y})) @>\phi_{y}>> \text{Gal}(F_{y}^{ab}/F_{y})\\
\end{CD}\]
where  $V$ is the group transfer map, and

\[\begin{CD}
\mathcal{J}_{L}/\Delta(K_{2}^{top}(L)) @>\phi_{L}>> \text{Gal}(L^{ab}/L)\\
@V N VV @VVV\\
\mathcal{J}_{y}/\Delta(K_{2}^{top}(F_{y})) @>\phi_{y}>> \text{Gal}(F_{y}^{ab}/F_{y}).\\
\end{CD}\]
}

\end{enumerate}
\end{thm}

\begin{proof}
By propositions \ref{reciprocitylaw} and \ref{tamereciprocity}, we know $\phi_{y}(K_{2}^{top}(F_{y}))$ is trivial in the absolute abelian Galois group. We again separate into the three cases of an Artin-Schreier extension, a Kummer extension, and an unramified extension.

For the unramified extension, the commutative diagram
\[\begin{CD}
\mathcal{J}_{y}/\Delta(K_{2}^{top}(F_{y})) @>\delta>> {\prod}'_{x \in y}k(y)_{x}^{\times}/k(y)^{\times} @>>> 0\\
@V\phi_{y}VV @V\phi_{k(y)}VV @.\\
\text{Gal}(F_{y}^{ab}/F_{y}) @>>> \text{Gal}(k(y)^{ab}/k(y)) @>>> 0 
\end{CD}
\]
and the fact that the right vertical map is injective with dense image show the left map is injective and has dense image.

Artin-Schreier-Witt duality and theorem \ref{mainthmcurves} induce the isomorphism
\[ J_{y}/(\Delta(K_{2}^{top}(F_{y}))\cap J_{y})J_{y}^{p^{m}} \to G^{wr}/(G^{wr})^{p^{m}}
\]
and passing to the projective limit gives the decomposition
\[J_{y}/(\Delta(K_{2}^{top}(F_{y}))\cap J_{y}) \to \varprojlim J_{y}/(\Delta(K_{2}^{top}(F_{y}))\cap J_{y})J_{y}^{p^{m}} \cong G^{wr}
\]
and hence the wildly ramified part of $\phi_{y}$ has dense image.\\
To show $\phi_{y}$ is injective, we must show
\[\cap_{m}(\Delta(K_{2}^{top}(F_{y}))\cap J_{y})J_{y}^{p^{m}} = \Delta(K_{2}^{top}(F_{y}))\cap J_{y}.
\] 
Now, for each $x \in y$ we have $\cap_{m}K_{2}^{top}(\mathcal{O}_{x, y}, \mathfrak{p}_{x, y})^{p^{m}} = \{ 1 \}$ - see \cite[Section 2, Lemma 3]{P4} - and hence this is true for the adelic product also. Hence the above equality holds, and so the wildly ramified part of the map in the projective limit is injective.

We now study the tamely ramified part of the reciprocity map. Kummer duality and theorem \ref{tamecurvethm} induce the isomorphism
\[\mathfrak{J}_{y}/(\Delta(K_{2}^{top}(F_{y}))\cap\mathfrak{J}_{y})\mathfrak{J}_{y}^{q-1} \to G^{tr}\]
showing that this part of the map is injective with dense image also.

Finally, noting that $({\prod}'_{x \in y}K_{2}^{top}(\mathcal{O}_{x,y}))/K_{2}^{top}(F_{y})\cap{\prod}'_{x \in y}K_{2}^{top}(\mathcal{O}_{x,y}) = J_{y} \times \mathfrak{J}_{y}$ and that the Galois group the Witt and Kummer dualities generate is isomorphic to Gal$(F_{y}^{ab}/F_{y})/\text{Gal}(F_{y}^{unram}/F_{y})$, we see that the whole reciprocity map is injective and has dense image.

For the remaining properties, 1 follows from theorem \ref{independenceofmodel}. For 2, consider the commutative diagram with exact lower row
\[
\begin{CD}
\frac{{\prod}'K_{2}^{top}(L_{x'})}{\Delta(K_{2}^{top}(L))\cap {\prod}'K_{2}^{top}(L_{x'})} @>N>> \frac{{\prod}'_{x \in y}K_{2}^{top}(F_{x,y})} {\Delta(K_{2}^{top}(F_{y}))\cap {\prod}'_{x \in y}K_{2}^{top}(F_{x,y})} @>\phi_{y}>> \text{Gal}(L/F_{y}) @>>> 0 \\
@V\phi_{L}VV  @V\phi_{y}VV  @||| @. \\
\text{Gal}(L^{ab}/L) @>>> \text{Gal}(F_{y}^{ab}/F_{y}) @>>> \text{Gal}(L/F_{y}) @>>> 0 . \end{CD}
\]
The exactness of the upper row follows from the fact that the image of the norm map $N$ is closed \cite[section 6]{IHLF} and that the images of the first two vertical maps are dense. 

The commutative diagrams follow straight from the corresponding local properties - see \cite{P4} - without the factorisation by the diagonal elements, and then the reciprocity laws from \ref{reciprocitylaw} prove them with the factorisation.
\end{proof}

\section{Arithmetic Two-Dimensional Local Rings}
\setcounter{subsection}{1}\subsection*{}\setcounter{thm}{0}
We will now fix a point $x \in X$ and study a ring of the type $F_{x}$ described in definition \ref{fieldsdefs} part 3. As in the preceding section, we will first study the Witt pairing for the wildly ramified part of the class field theory, then the higher tame pairing for the tamely ramified part.
\begin{defn}
Define $J_{x}:= \prod_{y \ni x} K_{2}^{top}(\mathcal{O}_{x,y}, \mathfrak{m}_{x,y})$.
\end{defn}

\noindent This group will be related to the Witt symbol, and is the analogue of $J_{y}$ in the preceding section. Note that the product is finite, so need not be a restricted product.

\medskip
 \noindent Firstly we will consider the case where our surface $X$ satisfies condition $\dagger$:
 
 \medskip
 $X$ has only strictly normal crossings, i.e. all intersections are transversal and $k_{y}(x)=k(x)$ for all $y \ni x$. 
 
 \medskip
\noindent We let the point $x$ lie on two curves, defined by parameters $u$ and $t$, so that the two dimensional local fields associated to $x$ are:
\[F_{u,t} : = k(x)((u))((t)) \text{  and  } F_{t,u}:=k(x)((t))((u)).\]
We aim to prove the following theorem:
\begin{thm}\label{mainthmpoints}
Fix a point $x \in X$. Then the pairing
\[\frac{J_{x}}{(\Delta(K_{2}^{top}(F_{x}))\cap J_{x})J_{x}^{p^{m}}} \times \frac{W_{m}(F_{x})}{(\mathrm{Frob} - 1)W_{m}(F_{x})} \to \mathbb{Z}/p^{m}\mathbb{Z}\]
is continuous and non-degenerate for each $m \in \mathbb{N}$. The induced homomorphism from $J_{x}/(\Delta(K_{2}^{top}(F_{x}))\cap J_{x})J_{x}^{p^{m}}$ to 
\[\mathrm{Hom}\left(\frac{W_{m}(F_{x})}{(\mathrm{Frob} - 1)W_{m}(F_{x})}, \mathbb{Z}/p^{m}\mathbb{Z}\right)\]
is a topological isomorphism. \end{thm}
We will prove this theorem in case $\dagger$ and then prove we can always reduce to this case.\\
We begin with some lemmas on the structure of the $K$-groups similar to the lemmas in the preceding section.

\begin{lemma}\label{kgpstructurepoints}
Fix $x \in X$, and let $u,t$ generate the maximal ideal of $\hat{\mathcal{O}}_{X,x}$. Let $y$, $y'$ run through the local irreducible curves in a neighbourhood of $x$ such that Spec$(\mathcal{O}_{X,x}) \cap y$ (resp. $y'$) determines a curve in a neighbourhood of $x$ with equation $t_{y}$ (resp. $t_{y'}$).  Then $K_{2}^{top}(F_{x})$ is generated by symbols of the form:
\begin{enumerate}
\item{ $\{t_{y}, t_{y'}\}$}
\item{ $\{a, t_{y}\}$ with $ a \in k(x)^{\times}$}
\item{$ \{1+ a u^{i}t^{j}, t_{y}\}$ with $ a \in k(x)^{\times}$, $ i, j \geq 1$.}
\end{enumerate}\end{lemma}
\begin{proof}
Let $\mathcal{E}_{x}$ be the group generated by the elements 
\[\{ a \in \hat{\mathcal{O}}_{X,x}^{\times} : a \equiv 1 \text{ mod } \mathfrak{p}_{x}\} = \{1 + b u^{i}t^{j}: b \in k(x)^{\times}, i, j \geq 1\}.\]
Then we can decompose the multiplicative group as
\[F_{x}^{\times} = \mathcal{E}_{x} \times k(x)^{\times} \times \oplus_{y} < t_{y} >\]
where the direct sum is taken over the curves as in the statement of the lemma.
This is because the irreducible curves in a neighbourhood of $x$ define a prime ideal of height one in $\mathcal{O}_{X,x}$ generated by the equation $t_{y}$, and it is easy to see that these are exactly the part of $F_{x}^{\times}$ not contained in the group generated by by the constants and the principal units of $\hat{\mathcal{O}}_{X,x}$.  Once we have this decomposition, the lemma follows exactly as in the above case.
\end{proof}

\begin{lemma}\label{relativekgppoints}
Let $\alpha \in J_{x}$ with $x \in X$ satisfying $\dagger$. Then $\alpha$ is a product of symbols:
\[\text{in } F_{u,t}: \begin{cases} \{ 1+a_{i,j}u^{i}t^{j}, t\} & p \nmid i;\\
\{1+b_{i,j}u^{i}t{^j}, u\} & p \mid i, \text{ and if } p \mid j, \ b_{i,j} = 0; \end{cases}\]
\[\text{in } F_{t,u}: \begin{cases} \{ 1+ a_{i,j}t^{j}u{^i}, t\} & p \mid j \text{ and if } p \mid i, \ a_{i,j} = 0;\\
\{1+ b_{i,j}t^{j}u^{i}, u\} & p \nmid j;\end{cases}\]
for $i,j \in \mathbb{N}$. \end{lemma}
\begin{proof} This follows immediately from theorem \ref{topkgpstructure} and the fact that for $x \in X$ satisfying $\dagger$, the product $\prod_{y \ni x}K_{2}^{top}(\mathcal{O}_{x,y}, \mathfrak{m}_{x,y})$ becomes the product of elements of types $4$ and $5$ in theorem \ref{topkgpstructure} for the fields $F_{u,t}$ and $F_{t,u}$. \end{proof}

Let $z = c_{i,j}u^{-i}t^{-j} \in W_{m}(F_{x})/($Frob$ - 1)W_{m}(F_{x})$. Notice that exactly as in lemma \ref{structuremodfrob}, at least one of the pair $(i,j)$ must be greater than zero. Similarly to the case of a fixed curve, we will look at the values of the pairing for a pair $(i,j)$ and distinguishing the cases depending on whether $i$ or $j$ is divisible by $p$.\\
\begin{lemma}\label{firstcalcpoints} Suppose $p \nmid i$, $p \mid j$ and let $\alpha = (\{1+a_{i,j}u^{i}t^{j}, t\}, \{1+b_{i,j}t^{j}u^{i}, t\}) \in J_{x}$ and $z = c u^{k}t^{l} \in F_{x}/(\mathrm{Frob}-1)F_{x}$. Then
\[ (\alpha | z ]_{x} = \begin{cases} ic(a_{i,j} - b_{i,j}) & i + k = 0, \ j+k=0 \\
0 & i+k>0 \text{ and } j+l > 0. \end{cases}\]
Symmetrically, if $p \mid i$ and $p \nmid j$, let $\beta = (\{1+ a'_{i,j}u^{i}t^{j}, u\}, \{1+b'_{i,j}t^{j}u^{i}, u\})$, then 
\[ (\beta | z]_{x} = \begin{cases} j c (b'_{i,j} - a'_{i,j}) & i+k = 0, \ j+l = 0\\
0 & i+k>0 \text{ and } j+l>0. \end{cases}\]\end{lemma}
\begin{proof} This is a simple calculation of residues, following as in lemma \ref{firstcalculation}.\end{proof}

\begin{lemma} \label{secondcalcpoints}Suppose $ p \nmid i,j$. Let $\alpha = (\{1+a_{i,j}u^{i}t^{j},t\}, \{1+b_{i,j}t^{j}u^{i},u\}) \in J_{x}$ and $z = c u^{k}t^{l} \in F_{x}/(\mathrm{Frob} - 1)F_{x}$. Then
\[ (\alpha|z]_{x} = \begin{cases} c(ia_{i,j} + j b_{i,j}) & i+k = 0, \ j+l=0\\
0 & i+k>0 \text{ and } j+l>0.\end{cases}\]\end{lemma}
\begin{proof} As for lemma \ref{firstcalcpoints}.\end{proof}

For $(i,j) \in \mathbb{N}^{2}$, define $J_{x, \geq i,j}$ to be the set of elements with both $a_{k,l}$ and $b_{k,l}$ equal to zero for all $k<i$, $l<j$, when expressed as a product of elements of the form given in lemma \ref{relativekgppoints}.

\begin{lemma}\label{dualityonepoints} Let $x \in X$ satisfy $\dagger$ and fix $(i,j) \in \mathbb{N}^{2}$ with $p \nmid i$, $p \mid j$. Then the map
\[ ( \ | \ ]_{x}: \frac{J_{x,\geq i,j}}{(\Delta(K_{2}^{top}(F_{x}))\cap J_{x, \geq i,j}).J_{x, \geq i+1, j}.J_{x,\geq i,j+1}} \times u^{-i}t^{-j}k(x) \to k(x)\]
is a non-degenerate pairing of $k(x)$-vector spaces. The induced homomorphism
\[ \frac{J_{x,\geq i,j}}{(\Delta(K_{2}^{top}(F_{x}))\cap J_{x, \geq i,j}).J_{x, \geq i+1, j}.J_{x,\geq i, j+1}} \to \mathrm{Hom}(k(x), k(x)) \cong k(x)\]
is an isomorphism. \end{lemma}

\begin{proof}
Let $\alpha_{\geq i,j} \in J_{x, \geq i,j}$. By lemma \ref{relativekgppoints}, $\alpha_{\geq i,j}$ is uniquely determined modulo $J_{x, \geq i+1. j}.J_{x, \geq i, j+1}$ by the pair $(a_{i,j}, b_{i,j})$, where $\alpha_{\geq i,j}$ is represented by $(\{1+a_{i,j}u^{i}t^{j},t\}, \{1+b_{i,j}t^{j}u^{i}, t\})$.\\
Let $z = c_{i,j}u^{-i}t^{-j}$, and suppose $(\alpha_{\geq i,j}| z]_{x} = 0$. Then by lemma \ref{firstcalcpoints}, $i c_{i,j}(a_{i,j} - b_{i,j}) = 0$. Since we are assuming $z \notin ($Frob$ - 1)F_{x}$ and $\alpha_{\geq i,j} \neq 0$, we must have $a_{i,j}=b_{i,j}$. Hence $\alpha_{\geq i, j} \in \Delta(K_{2}^{top}(F_{x}))\cap J_{x, \geq i,j}$, and the pairing is non-degenerate.\\
The isomorphism to the homomorphism group follows easily, as the left hand side is obviously isomorphic to $k(x)$. \end{proof}

\noindent The case $(i,j) \in \mathbb{N}^{2}$, $p\nmid j$, $p \mid i$ is identical to the above lemma.

\begin{lemma}\label{dualitytwopoints} Let $x \in X$ satisfy $\dagger$ and fix $(i,j) \in \mathbb{N}^{2}$ with $p \nmid i,j$. Then the map
\[ ( \ | \ ]_{x}: \frac{J_{x,\geq i,j}}{(\Delta(K_{2}^{top}(F_{x}))\cap J_{x,\geq i,j}).J_{x,\geq i+1, j}.J_{x, \geq i, j+1}} \times u^{-i}t^{-j}k(x) \to k(x)\]
is a non-degenerate pairing of $k(x)$-vector spaces. The induced homomorphism
\[\frac{J_{x,\geq i,j}}{(\Delta(K_{2}^{top}(F_{x}))\cap J_{x,\geq i,j}).J_{x,\geq i+1, j}.J_{x, \geq i, j+1}} \to \mathrm{Hom}(k(x), k(x)) \cong k(x)\]
is an isomorphism. \end{lemma}

\begin{proof}
As in the proof of lemma \ref{dualityonepoints}, $\alpha_{\geq i,j} \in J_{x, \geq i,j}$ is uniquely determined modulo $J_{x, \geq i+1, j}.J_{x,\geq i, j+1}$ by $(a_{i,j}, b_{i,j})$ and represented by $(\{1+a_{i,j}u^{i}t^{j}, t\}, \{1+ b_{i,j}t^{j}u^{i}, u\})$.  \\
Let $z = c_{i,j}u^{-i}t^{-j}$, so that by lemma \ref{secondcalcpoints} we have \[(\alpha_{\geq i,j}| z]_{x} = c_{i,j}(ia_{i,j} + jb_{i,j}).\]
Suppose $(\alpha_{\geq i,j}|z]_{x}=0$. As in the lemma above, we must then have $jb_{i,j}=-ia_{i,j}$, i.e.
\[\alpha_{\geq i,j} = (\{1+a_{i,j}u^{i}t^{j},t\}, \{1-(i^{-1}j)a_{i,j}t^{j}u^{i}, u\}).\]
We use the $K$-group identity from \ref{Kgpcalc2}
\[ \{1 - ivt^{j}u^{i}, u\} \equiv \{1 + jvu^{i}t^{j}, t\} \text{ mod } K_{2}^{top}(\mathcal{O}_{t,u}, \mathfrak{p}_{t,u}).\]
So $\alpha_{\geq i,j}$ can be represented by:
\[(\{1+ a_{i,j}u^{i}t^{j}, t\}, \{1+ a_{i,j}u^{i}t^{j}, t\}) \in \Delta(K_{2}^{top}(F_{x}))\cap J_{x, \geq i,j}.\]
Hence the pairing is non-degenerate and the proof ends exactly as in lemma \ref{dualityonepoints}. \end{proof}

\noindent We now proceed to the proof of theorem \ref{mainthmpoints}, in the case $x$ satisfies hypothesis $\dagger$.
\begin{proof} \emph{m=1}:
We first prove non-degeneracy in the second argument. Let $z = \sum_{i,j}c_{i,j}u^{i}t^{j}$ be a representative of $F_{x}/(\text{Frob} - 1)F_{x}$, and suppose $(\alpha|z]_{x}=0$ for all $\alpha \in J_{x}$.  Assume $z \neq 0$ and let $(i,j)$ be the least index with $c_{i,j} \neq 0$ ordered lexicographically.\\
Let $\alpha_{-i,-j} \in J_{x,-i,-j}$. Then
\[ 0 = (\alpha_{-i,-j}|z]_{x} = (\alpha_{-i,-j}|c_{i,j}u^{i}t^{j}]_{x}\]
by lemmas \ref{firstcalcpoints} and \ref{secondcalcpoints}, but also by these lemmas, this is a contradiction. Hence the pairing is non-degenerate in the right argument.\\
Now let \[\mu: F_{x}/(\text{Frob} -1)F_{x} \to \mathbb{Z}/p\mathbb{Z}\] be a homomorphism. We describe $\mu$ via a family of homomorphisms
\[\mu_{i,j}: k(x) \to \mathbb{Z}/p\mathbb{Z}\]
mapping $c_{i,j}$ to $\mu(c_{i,j}u^{i}t^{j})$ for each $(i,j)$ not both greater than zero.\\
We will construct an $\alpha \in J_{x}/J_{x}^{p}$ such that 
\[(\alpha|z]_{x}= \mu(z)\]
for all $z \in F_{x}/(\text{Frob}-1)F_{x}$ and such that $\alpha$ is unique up to $\Delta(K_{2}^{top}(F_{x}))\cap J_{x}$.\\
Similarly to the proof of theorem \ref{mainthmcurves}, for $\alpha \in J_{x}/J_{x}^{p}$ we inductively define
\[\alpha_{> i,j} = \alpha_{i,j}\alpha_{\geq i+1, j}\alpha_{\geq i, j+1}\] and $\alpha_{1,1} $ the element defined by $a_{1,1}$ and $b_{1,1}$ in our expansion of $\alpha$ as a product of elements of the form given in lemma \ref{relativekgppoints}.\\
By lemmas \ref{firstcalcpoints} and \ref{secondcalcpoints}, we have
\[(\alpha |c_{-i,-j}u^{-i}t^{-j}]_{x} = \sum_{1 \leq k \leq i , \ 1 \leq l \leq j}(\alpha_{k,l}|c_{-i,-j}u^{-i}t^{-j}]_{x}\]
and so $(\alpha | z ]_{x} = \mu(z)$ for all $z \in F_{x}/(\text{Frob} - 1)F_{x}$ if and only if
\[(\alpha_{i,j}| c_{-i,-j}u^{-i}t^{-j}]_{x} = \mu_{i,j}(c_{-i,-j}) - \sum_{1 \leq k \leq i , \ 1 \leq l \leq j}(\alpha_{k,l}| c_{-k,-l}u^{-k}t^{-l}]_{x}\]
for all pairs $(i,j)$ not both less than zero.\\
Now by lemmas \ref{dualityonepoints} and \ref{dualitytwopoints}, there does exist such an $\alpha_{i,j}$ for each pair $(i,j)$, uniquely defined up to $(\Delta(K_{2}^{top}(F_{x}))\cap J_{x,\geq i,j}).J_{x,\geq i+1, j}.J_{x, \geq i, j+1}$. So let $\alpha = \prod\alpha_{i,j}$, and we have the required element. Hence the proof is complete for $m = 1$.\\
The induction follows exactly in the proof of theorem \ref{mainthmcurves}. \end{proof}
\bigskip

We now study the global higher tame pairing for a fixed point on the surface $X$. We first define $\mathfrak{J}_{x}$ to be the ring generated by
\[{\prod_{y \ni x}}' \{k(x), t_{y}\} \times \prod_{y, y' \ni x, y' \neq y}\{t_{y},t_{y'}\}.\]
We will proceed very similarly to the case of a fixed curve, aiming to prove the theorem below.

\begin{thm}\label{tamepointthm}
Fix a point $x \in X$, and let $k(x)$, the residue field at $x$, be a finite field of size $q$ Then the global higher tame pairing on
\[ \mathfrak{J}_{x}/(\Delta(K_{2}^{top}(F_{x}))\cap \mathfrak{J}_{x})\mathfrak{J}_{x}^{q-1} \times F_{x}^{\times}/(F_{x}^{\times})^{q-1} \to \mathbb{F}_{q}^{\times}\]
is continuous and non-degenerate.
\end{thm}

As above for the Witt pairing, we will apply condition $\dagger$ at first. We will again use a combinatorial argument to give each generator of $F_{x}^{\times}/(F_{x}^{\times})^{q-1}$ exactly one generator of the quotient group $\mathfrak{J}_{x}/(\Delta(K_{2}^{top}(F_{x}))\cap \mathfrak{J}_{x})\mathfrak{J}_{x}^{q-1}$ with which it has a non-zero value when the higher tame pairing is applied.

From the description of $F_{x}^{\times}$ in lemma \ref{kgpstructurepoints}, we see that $F_{x}^{\times}/(F_{x}^{\times})^{q-1}$ is generated by a $(q-1)^{th}$ root of unity $\zeta \in k(x)$ and a local parameter $t_{y}$ for each $y$ passing through $x$ - with condition $\dagger$, we will have just two such local parameters $u$ and $t$. We study the elements in the diagonal embedding of $K_{2}^{top}(F_{x})$, and use this to study the generators of the quotient group.

\begin{lemma}
Let $\alpha \in \Delta(K_{2}^{top}(F_{x}))\cap\mathfrak{J}_{x}$, where $x \in X$ satisfies condition $\dagger$. Then $\alpha$ is a product of elements of the form: 
\begin{enumerate}
\item{$(\{\zeta,u\}, \{\zeta,u\})$;}
\item{$(\{\zeta,t\}, \{\zeta, t\})$;}
\item{$(\{u,t\}, \{u,t\})$.}
\end{enumerate} 
\end{lemma}
\begin{proof}
As $x$ satisfies $\dagger$, we know by lemma \ref{kgpstructurepoints} that $K_{2}^{top}(F_{x})$ is generated by the elements $\{\zeta,u\}$, $\{\zeta,t\}$ and $\{u,t\}$, and the principal units which we do not need to consider here. Embedding each of these elements diagonally into $\mathfrak{J}_{x}$, we get elements which are non-trivial in both local fields $F_{u,t}$ and $F_{t,u}$ and can still be written in this form.
\end{proof}

\begin{lemma}
Let $\alpha \in \mathfrak{J}_{x}/(\Delta(K_{2}^{top}(F_{x}))\cap\mathfrak{J}_{x})\mathfrak{J}_{x}^{q-1}$, where $x \in X$ satisfies condition $\dagger$. Then $\alpha$ can be written as a product of elements on the form:
\begin{enumerate}
\item{$(\{\zeta,u\},1)$;}
\item{$(\{\zeta,t\},1)$;}
\item{$(\{u,t\},1)$.}
\end{enumerate}
\end{lemma}

\begin{proof}
Any element with non-trivial entries only in the first column clearly satisfies the lemma because of the structure of the topological $K$-groups of a higher local field. So suppose $\alpha = (\beta, \gamma)$ for some elements $\beta \in K_{2}^{top}(F_{u,t})$ and $\gamma \in K_{2}^{top}(F_{t,u})$. Then multiplying by the element $(\gamma, \gamma)^{q-2} \in \Delta(K_{2}^{top}(F_{x}))\cap\mathfrak{J}_{x}$ gives us the element $(\beta\gamma, 1) \in \mathfrak{J}_{x}/(\Delta(K_{2}^{top}(F_{x}))\cap\mathfrak{J}_{x})\mathfrak{J}_{x}^{q-1}$, which is a product of elements of types 1, 2 and 3 as before. 
\end{proof}

We may now very simply prove \ref{tamepointthm} in the case $x$ satisfies $\dagger$. Recall we wish to pair each generator of $F_{x}^{\times}/(F_{x}^{\times})^{q-1}$ with a generator of $\mathfrak{J}_{x}/(\Delta(K_{2}^{top}(F_{x}))\cap\mathfrak{J}_{x})\mathfrak{J}_{x}^{q-1}$, so that the two elements have non-zero pairing, thus showing non-degeneracy. We pair the elements $\zeta$, $u$ and $t$ with the elements of types 3, 2 and 1 respectively as defined in the lemma above. Each pair yields one of the elements $\pm \zeta \in \mathbb{F}_{q}$. This completes the proof when $x$ satisfies $\dagger$.

\bigskip

We now prove the theorems without the condition $\dagger$. We first look at the case of the Witt pairing.

\medskip
\noindent\emph{Proof}

\medskip
\noindent We will proceed by considering a general element in $J_{x}$, and asking when it can be a ``degenerate element" - i.e. an element $\alpha$ which has $(\alpha|h]_{x}=0$ for every $h \in F_{x}$. We will use the local case, case $\dagger$ and consider the forms the elements can take, and see that the only degenerate elements which can occur must be diagonal elements.

 We briefly introduce our argument. First, we take a typical element of $J_{x}$:
\[\left( \{1+\beta_{1}\alpha_{1}^{i_{1}}\gamma_{1}^{j_{1}}, \alpha_{1}\},   \{1+\beta_{2}\alpha_{2}^{i_{2}}\gamma_{2}^{j_{2}},\alpha_{2}\}, \dots \right).\]

\noindent Then we have two options:
\begin{enumerate}
\item{There exists a pair of local parameters $\alpha_{k}$, $\gamma_{k}$ such that the entry in one of the two local fields defined by the parameters is non-trivial, and the entry in the other is trivial.}
\item{The entries in all pairs of local fields as described above are either both non-trivial or both trivial.}
\end{enumerate}

In the first case, the local case and first part of the argument below will show this type of element can never be degenerate. In the second case, condition $\dagger$ shows that each pair of local parameters must have the same entry for both the local fields defined, and an element with all entries of this form is itself diagonal.

\medskip
We now begin the rigorous argument by considering the general element
\[ \left( \{1+\beta_{1}\alpha_{1}^{i_{1}}\gamma_{1}^{j_{1}}, \alpha_{1}\},   \{1+\beta_{2}\alpha_{2}^{i_{2}}\gamma_{2}^{j_{2}},\alpha_{2}\}, \dots \right) \in J_{x},\]

\noindent where the $\beta_{k} \in k_{y_{k}}(x)$, and the $\alpha_{k}$, $\gamma_{k}$ are local parameters for the localisation of $F_{x}$ given by the prime $y_{k} \in \mathcal{O}_{X,x}$.

We examine when such an element can be degenerate in the left hand side of the Witt pairing on
\[J_{x}/\Delta(K_{2}^{top}(F_{x})) \times F_{x}/(\text{Frob}-1)F_{x}.\]

If just one of the entries $\{1+ \beta_{k}\alpha_{k}^{i_{k}}\gamma_{k}^{j_{k}}, \alpha_{k}\}$ is non-trivial, then we are in the local situation and can always find an element of $F_{x}$ with which our element has non-zero Witt pairing - here take 
\[h_{0}=\alpha_{k}^{-i_{k}}\gamma_{k}^{-j_{k}}.\]

 See lemma \ref{firstcalcpoints} for the calculation. There is a further difficulty here if the element is in $K_{2}^{top}(F_{x,y})^{p^{m}}$. Then we must pair it with an element in $W_{m}(F_{x})$, a part of the induction we will discuss more later. 
 
So to be degenerate we must have more than one non-trivial entry.

We now look at the case where exactly two of the entries are non-trivial. Suppose these entries are in localisations of $F_{x}$ with \emph{different} local parameters from each other, say corresponding to primes $y_{k}$ and $y_{l}$. 
So the two non-trivial entries are $\{1+ \beta_{k}\alpha_{k}^{i_{k}}\gamma_{k}^{j_{k}}, \alpha_{k}\}$ and $\{1+ \beta_{l}\alpha_{l}^{i_{l}}\gamma_{l}^{j_{l}}, \alpha_{l}\}$.

Then letting 
\[h_{0} = \alpha_{k}^{-i_{k}}\gamma_{k}^{-j_{k}}\]

\noindent we have $(\{1+ \beta_{k}\alpha_{k}^{i_{k}}\gamma_{k}^{j_{k}}, \alpha_{k}\}|h_{0}]_{\alpha_{k},\gamma_{k}}= j_{k}\beta_{k}$ by lemma \ref{firstcalcpoints}.

So if $j_{k}$ is not divisible by $p$, we have an element of $F_{x}$  which has a non-zero pairing with the first entry and zero with the second, and hence non-zero when summed. If $p^{m}$ is the maximal power of $p$ dividing $j_{k}$, replace $h_{0}$ by the Witt vector with $h_{0}$ in the $m^{th}$ position to get the same result, by property 7 of lemma \ref{wittproperties}. \
Hence if our element is a product of local elements with all the non-trivial entries in fields defined by \emph{different} curves $y_{i} \ni x$, then the element cannot be degenerate. 

\medskip
We now consider the case where there are two entries from fields defined by the \emph{same} pair of local parameters, starting with this pair being the only nontrivial entries.

 Then we can apply the calculations from lemmas \ref{dualityonepoints} and \ref{dualitytwopoints}, where condition $\dagger$ is satisfied, to see that it must have identical entries in the two non-trivial places. But this is exactly the image of the element of $K_{2}^{top}(F_{x})$ with these entries diagonally embedded in $J_{x}$, as they are trivial in all other localisations. So we have proven non-degeneracy in this case.

\medskip
Finally, we discuss the case where there are more than two non-trivial entries in the element. Suppose first that there is an entry with local parameters $\alpha_{k}$, $\gamma_{k}$ which aren't local parameters for any of the other non-trivial entries, i.e $\alpha_{k} \neq \alpha_{l}$ and $\gamma_{k} \neq \gamma_{l}$ for all $l$ with $\beta_{l}\neq 0$.

 Then arguing as in the case of two non-trivial entries above, we have an element  $\alpha_{k}^{-i_{k}}\gamma_{k}^{-j_{k}} \in F_{x}$ which has a non-zero pairing with $\{1+ \beta_{k}\alpha_{k}^{i_{k}}\gamma_{k}^{j_{k}}, \alpha_{k}\}$ and zero with each other element. So to be degenerate, an element must have at least two entries for each pair of local parameters.

But the only two fields which can be defined by these parameters $\alpha_{k}, \ \gamma_{k}$ are the localisations with respect to the prime ideal generated by one of them, then the maximal ideal generated by both - i.e. $F_{\alpha_{k},\gamma_{k}}$ and $F_{\gamma_{k},\alpha_{k}}$. So in fact any degenerate element must be a sum of the type discussed above. But such an element is the image, under the diagonalisation map, of the product of all its entries - each parameter can be regarded as trivial in the topological $K$-groups of the local fields where it is not a local parameter.

Hence the pairing is non-degenerate on the left hand side $J_{x}/\Delta(K_{2}^{top}(F_{x}))$.

\medskip
So we now must prove that the pairing is non-degenerate on the right-hand side. Following from the calculations \ref{dualityonepoints} and \ref{dualitytwopoints}, we see this is equivalent to proving that the elements of $W_{m}(F_{x})/(\text{Frob}-1)W_{m}(F_{x})$ are all of the required form for each integer $m$, i.e. every entry $f$ in the Witt vector is 
a sum of elements of the form
of the form $f=\beta\prod_{k}\alpha_{k}^{i_{k}}$ where $\beta \in k(x)$, the $\alpha_{k}$ are primes of $F_{x}$, and at least one of the $i_{k}$ is negative. This argument follows as before, in lemma \ref{structuremodfrob}: suppose all coefficients are greater than zero, then look at the convergent (with respect to the topology of $\mathcal{O}_{X,x}$)
sum $f'=(-f)+(-f)^{p}+(-f)^{p^{2}}+\dots $, hence $f=f'^{p}-f$ is trivial modulo (Frob$-1)$. So we can now complete the proof, as this shows the non-degeneracy on the right-hand side of the pairing.

\medskip
\noindent This completes the proof of theorem \ref{mainthmpoints}.

\medskip
\noindent \emph{Remark}

\medskip
One can use the work of Matsumi to understand the structure of $F_{x}/(\text{Frob}-1)F_{x}$, then induct as in the case of a fixed curve. For a complete two-dimensional local ring $R$ of positive characteristic, Matsumi's paper \cite{KM} finds a simple form for the rings $R/R^{p}$ and $F/F^{p}$, where $F$ is the fraction field of $R$. 

\medskip
\noindent We now complete the proof of theorem \ref{mainthmpoints}.
\begin{proof}
The above discussion uses the structure of the group $J_{x}$, the argument for the normal crossings case, and the structure of $F_{x}/(\mathrm{Frob}-1)F_{x}$ to show that
\[\frac{J_{x}}{(\Delta(K_{2}^{top}(F_{x}))J_{x}^{p}} \cong \text{Hom}\left(\frac{F_{x}}{(\text{Frob}-1)F_{x}}, \mathbb{Z}/p\mathbb{Z}\right).\]

We can then induct on the length of the Witt vectors as in the proof of \ref{mainthmcurves}. Suppose
\[\frac{J_{x}}{(\Delta(K_{2}^{top}(F_{x}))J_{x}^{p^{m}}} \cong
\text{Hom}\left(\frac{W_{m}(F_{x})}{(\text{Frob}-1)W_{m}(F_{x})}, \mathbb{Z}/p^{m}\mathbb{Z}\right)\]
for some integer $m$.
Let $\mu \in $ Hom$(W_{m+1}(F_{x})/(\text{Frob}-1)W_{m+1}(F_{x}), \mathbb{Z}/p^{m+1}\mathbb{Z})$. Then as before we define the restriction $\mu': W_{m}(F_{x})/(\text{Frob}-1)W_{m}(F_{x}) \to \mathbb{Z}/p^{m}\mathbb{Z}$ by
\[\mu'(f_{0},\dots, f_{m-1}) = V(\mu(f_{0},\dots, f_{m-1}, 0)).\]

Then $\mu'$ can be associated to $\alpha \in J_{x}/(\Delta(K_{2}^{top}(F_{x}))J_{x}^{p^{m}}$. Then as in the proof of \ref{mainthmcurves}, $\alpha$ will also associate to $\mu$ in the same way, uniquely up to $J_{x}^{p^{m}}$ as required.
\end{proof}

\medskip
We now remove the necessity for condition $\dagger$ for the case of the higher tame symbol and $\mathfrak{J}_{x}$. We proceed in a broadly similar manner to the argument for the Witt symbol, reducing back down to the case of exactly two local parameters and looking at the quotient by the diagonal elements.

Firstly, we note that $F_{x}/(F_{x})^{q-1}$ is generated by $k(x)$, and a local parameter $t_{y}$ for each curve $y$ passing through $x$. Let $k(x)$ itself be generated by the $(q-1)^{th}$ root of unity $\zeta$.

We wish to pair $\zeta$ and each $t_{y}$ with an element of $\mathfrak{J}_{x}$, unique up to the diagonal elements and a power of $(q-1)$, so that the elements we choose generate $\mathfrak{J}_{x}/(\Delta(K_{2}^{top}(F_{x}))\mathfrak{J}_{x}^{q-1}$. As before, this is enough to prove Kummer duality and the tame part of the reciprocity map.

We pair an element $t_{y}$ with an element with $\{\zeta, t_{y'}\}$ in the position corresponding to a two-dimensional local field with $t_{y}$ and $t_{y'}$ as local parameters - there are two such fields in the adeles at $x$, but case $\dagger$ above shows that this choice does not matter. We must show also that the choice of prime $t_{y'}$ does not matter.

The higher tame pairing will take the value
\[(\{\zeta,t_{y'}\},t_{y})_{x}= \zeta^{(t_{y},t_{y'})_x}\]
where $(t_{y},t_{y'})_{x}$ is the intersection multiplicity at $x$. Since the intersection multiplicity satisfies
\[(t_{y},t_{y'})_{x} = \text{dim}_{k(x)}\left(\mathcal{O}_{X,x}/(t_{y},t_{y'})\right),\]
we see that the value of the higher tame pairing again differs by norms as the choice of $y'$ varies, so as in the argument for a fixed curve, the choice of $y'$ in the quotient does not matter, as it will change only up to a power of $(q-1)$.

We pair the element $\zeta$ with an element with $\{t_{y},t_{y'}\}$ in the position corresponding to a two-dimensional local field with $t_{y}$ and $t_{y'}$ as local parameters and trivial everywhere else, where $t_{y}$ and $t_{y'}$ are distinct height one primes of $\mathcal{O}_{X,x}$.

To show our choice does not matter in the quotient, we argue as above. Firstly the choice of fields $F_{t_{y},t_{y'}}$ or $F_{t_{y'},t_{y}}$ does not matter by condition $\dagger$ in the preceding section. 

Secondly, we consider our choice of $y$ and $y'$. Up to a sign, the pairing will take the value $\zeta^{(t_{y},t_{y'})_{x}}$. If we change the primes to $t_{y_{1}}$ and $t_{y'_{1}}$, this is equivalent in the adelic quotient to multiplying by the element with $\{t_{y_{1}},t_{y'_{1}}\}$ in the place corresponding to the new primes, and $\{t_{y},t_{y'}\}^{q-2}$ in the place corresponding to the old primes. 

But by the same argument as before, we may replace all these elements by their norms, which are $(q-1)^{th}$-powers in the adeles. Hence the above value of the higher tame pairing is unchanged and the elements all become equivalent in the quotient.

 So we have paired each generator of $F_{x}/F_{x}^{q-1}$ with a generator of the tame part of the adeles at $x$ modulo $(q-1)^{th}$ powers, and hence can apply Kummer theory as usual.

\bigbreak

We now construct the reciprocity map for the ring $F_{x}$. As in the previous section, we begin with some basic Galois theory to show the map is well-defined.

Let $L/F_{x}$ be a finite extension with Galois group $G$, $\mathcal{O}_{L}$ the integral closure of $\hat{\mathcal{O}}_{X,x}$ in $L$ and $\mathfrak{p}_{L}$ its maximal ideal. As mentioned at the start of section two, every height one prime ideal $\mathfrak{q} \subset \mathfrak{p}_{L}$ determines a two-dimensional local field $L_{\mathfrak{p}_{L}, \mathfrak{q}}$. Spec$(\mathcal{O}_{L})$ is a normal two-dimensional scheme over the residue field $l$ (see \cite[8.2.39]{L}), a finite extension of $k(x)$, and we have a finite morphism
\[\phi: \text{Spec}(\mathcal{O}_{L}) \to \text{Spec}(\hat{\mathcal{O}}_{X,x}).\] 
For a height one prime ideal $\mathfrak{q}$ of $\mathcal{O}_{L}$, define the stabiliser
\[G_{\mathfrak{q}} = \{g \in G : g(\mathfrak{q}) = \mathfrak{q}\}.\]
If $\mathfrak{q}$, $\mathfrak{q}'$ are two such primes, and $\phi(\mathfrak{q}) = \phi(\mathfrak{q}')$ then $G_{\mathfrak{q}}$ is conjugate to $G_{\mathfrak{q}'}$ in $G$. \\

Now let $L/F_{x}$ be an abelian extension - then the homomorphism
\[\text{Gal}(L_{\mathfrak{p}_{L}, \mathfrak{q}}/F_{x,y}) \cong G_{\mathfrak{q}} \to G = \text{Gal}(L/F_{x})\]
is independent of the choice of $\mathfrak{q}$, where $\mathfrak{q}$ is any prime ideal such that $\phi(\mathfrak{q})$ is the prime ideal of $\hat{\mathcal{O}}_{X,x}$ associated to the curve $y$. Of course, this is just basic valuation theory - see \cite{Bour}, chapter VI.

We now define the unramified part of the reciprocity map.  The unramified closure of the ring $F$ is the ring generated by $F$ and $\bar{\mathbb{F}}_{q}$, and its Galois group is canonically isomorphic to $\hat{\mathbb{Z}}$, generated by the Frobenius automorphism of $\bar{\mathbb{F}}_{q}$, Frob.

\begin{defn}\label{unrammap}
Let $\delta: K_{2}^{top}(F_{x,y}) \to K_{1}^{top}(\bar{F}_{x,y})$ be the boundary homomorphism of $K$-theory. We define the map
\[Un_{x,y}: K_{2}^{top}(F_{x,y}) \to \hat{\mathbb{Z}}\]
by 
\[\{\alpha,\beta\} \mapsto \mathrm{Frob}^{v_{\bar{F}_{x,y}}(\delta(\{\alpha,\beta\})},\]
where $v_{\bar{F}_{x,y}}$ is the valuation map of the local field $\bar{F}_{x,y}$.
\end{defn}

We define $Un_{x}$ to be the product of the $Un_{x,y}$ over the local irreducible curves $y \ni x$. Note that this product is well-defined on the adelic group ${\prod}'_{y \ni x}K_{2}^{top}(F_{x,y})$, as for all but finitely many $y \ni x$, the component $\{\alpha_{x,y},\beta_{x,y}\}$ is in $K_{2}^{top}(\mathcal{O}_{x,y})$ and hence the value of $\delta(\{\alpha_{x,y},\beta_{x,y}\})$ is $1$. 

\begin{lemma}\label{unramreci}
The map $Un_{x}$ obeys the reciprocity law, i.e. for an element $\{\alpha,\beta\} \in K_{2}^{top}(F_{x})$, we have $Un_{x}(\{\alpha, \beta\}) = 1$.
\end{lemma}
 
\begin{proof}
Using lemma \ref{kgpstructurepoints}, we may calculate the image of $K_{2}^{top}(F_{x})$ under the map $Un_{x}$. We have:
\[\delta_{x,y}(\{t_{y},t_{y'}\})=t_{y};\]
\[\delta_{x,y'}(\{t_{y},t_{y'}\})=-t_{y'};\]
\[\delta_{x,y}(\{a,t_{y}\})=a;\]
\[\delta_{x,y}(1+au^{i}t^{j},t_{y}\} =1;\]
and all other values are trivial. Hence the image of $K_{2}^{top}(F_{x})$ is generated by the images of the elements
\[t_{y}-t_{y'}, \ \ \ \ \ \ a\]
 in $k(y)_{x}^{\times}$, where $y$ and $y'$ range through the curves passing through $x$ and $a \in k(x)^{\times}$.
 Now to find the image of $Un_{x}$, we sum the valuations of the image of $\delta$ over the curves $y \ni x$. It is easy to see that in both cases the sum of the generators over the two non-zero values is zero.
\end{proof}

So the product of all the symbols
\[{\prod}'_{y \ni x}K_{2}^{top}(F_{x,y}) \to \text{Gal}(F_{x}^{ab}/F_{x})\]
is well-defined. By lemmas \ref{wittwelldefined} and \ref{tamewlldefined}, we know the product converges.

Now we define 
\[\phi_{x} : {\prod}'_{y \ni x} K_{2}^{top}(F_{x,y}) \to \text{Gal}(L/F_{x})\] to be the sum of the $\phi_{x,y}(L)$.
\medskip
\begin{lemma}
Let $L/F_{x}$ be a finite abelian extension. Then for almost all $y \ni x$, we have $\phi_{x,y}(L) = 1$, and hence $\psi_{x}$ is a continuous homomorphism.
\end{lemma}
\begin{proof}
By \cite{P4} section four, it is sufficient to prove the lemma in the three cases $L =F_{x}(\gamma)$, $L/F_{x}$ an Artin-Schreier extension with $\gamma^{p} - \gamma = \alpha$ for some $\alpha \in F_{x}$,  $L = F_{x}(\beta)$ is a Kummer extension where $\beta^{l} = \delta$ for some $l|q-1$ and $\delta \in F_{x}$, and an extension of only the base field $k(x)$.

This is sufficient as the abelian closure, $F_{x}^{ab}/F_{x}$ is generated by the maximal unramified extension, the maximal ramified and prime to $p$ extension, and the maximal $p$-extension. These three types of extension are disjoint, except for the unramified $p$-extension, where the maps are compatible.

For the first case, the local residue symbol is described by the relation
\[\phi_{x,y}(w_{x, y})(z) = (w_{x,y}|\alpha]_{x,y}(z)
\]
 for $w_{x,y} \in K_{2}^{top}(F_{x,y})$ and we know this is zero for almost all $y \in x$ from lemma \ref{wittwelldefined}.

For the Kummer extension, the local residue symbol is described by the relation
\[\phi_{x,y}(w_{x,y})(z) = (w_{x,y}, \delta)_{x,y}\]
and similarly we know this is trivial for almost all $y \in x$ by lemma \ref{tamewlldefined}.

For the extension of $k(x)$, using the calculations in lemma \ref{unramreci} we see that $\phi_{x,y}$ is non-trivial only in the case where the component is of the form $\{t_{y},t_{y'}\}$, which by our adelic restrictions can happen in only finitely many places.

The continuity of the reciprocity map follows, as the preimage of any open subgroup of Gal$(F_{x}^{ab}/F_{x})$ has only finitely many non-zero elements of $J_{x}$. But from the definition of the topology, this is exactly what is required in the direct sum and product topology.

\end{proof}

\noindent We now prove the main theorem of the section.

\begin{thm}\label{cftpoints}
Let $X$ be a regular projective surface over the finite field $\mathbb{F}_{q}$, and $x \in X$ a closed point. Then the continuous map

\[\psi_{x}: \mathcal{J}_{x}/\Delta(K_{2}^{top}(F_{x})) \to \text{Gal}(F_{x}^{ab}/F_{x})\]
is injective with dense image. It also satisfies:
\begin{enumerate}
\item{For any finite abelian extension, the following sequence is exact
\[\begin{CD}
\frac{\prod_{y'\ni x'}K_{2}^{top}(L_{y'})}{\Delta(K_{2}^{top}(L))\cap\prod_{y'\ni x'}K_{2}^{top}(L_{y'})}  @>N>> \mathcal{J}_{x}/\Delta(K_{2}^{top}(F_{x})) @>\psi_{x}>> \text{Gal}(L/F_{x})@>>> 0.
\end{CD}\]}

\item{For any finite separable extension $L/F_{x}$, the following diagrams commute:
\[\begin{CD}
\mathcal{J}_{L}/\Delta(K_{2}^{top}(L)) @>\phi_{L}>> \text{Gal}(L^{ab}/L)\\
@AAA @A V AA\\
\mathcal{J}_{x}/\Delta(K_{2}^{top}(F_{x})) @>\phi_{x}>> \text{Gal}(F_{x}^{ab}/F_{x})\\
\end{CD}\]
where  $V$ is the group transfer map, and

\[\begin{CD}
\mathcal{J}_{L}/\Delta(K_{2}^{top}(L)) @>\phi_{L}>> \text{Gal}(L^{ab}/L)\\
@V N VV @VVV\\
\mathcal{J}_{x}/\Delta(K_{2}^{top}(F_{x})) @>\phi_{x}>> \text{Gal}(F_{x}^{ab}/F_{x}).\\
\end{CD}\]
}
\end{enumerate}
\end{thm}

\begin{proof}
As in the case for a fixed curve, the commutative diagrams follow from the local case proved in \cite{P4} and the reciprocity laws in \ref{reciprocitylaw}.

We now show $\psi_{x}$ is injective with dense image, using the basic facts of Artin-Schreier-Witt and Kummer duality in a similar manner to the proof of \ref{cftcurves}.

Artin-Schreier-Witt duality and theorem \ref{mainthmpoints} induce the isomorphism
\[ J_{x}/(\Delta(K_{2}^{top}(F_{x}))\cap J_{x})J_{x}^{p^{m}} \to \text{Gal}(F_{x}^{ab,p}/F_{x})/(\text{Gal}(F_{x}^{ab,p}/F_{x}))^{p^{m}}
\]
and passing to the projective limit gives the decomposition
\[J_{y}/(\Delta(K_{2}^{top}(F_{x}))\cap J_{x}) \to \varprojlim J_{x}/(\Delta(K_{2}^{top}(F_{x}))\cap J_{x})J_{x}^{p^{m}} \cong \text{Gal}(F_{x}^{ab,p}/F_{x})
\]
and hence the wildly ramified part of $\phi_{x}$ has dense image.\\
To show $\phi_{x}$ is injective, we must show
\[\cap_{m}(\Delta(K_{2}^{top}(F_{x}))\cap J_{x})J_{x}^{p^{m}} = \Delta(K_{2}^{top}(F_{x}))\cap J_{x}.
\] 
Now, for each $y \ni x$ we have $\cap_{m}K_{2}^{top}(\mathcal{O}_{x, y}, \mathfrak{p}_{x, y})^{p^{m}} = \{ 1 \}$, and hence this is true in the adelic product also. So the wildly ramified part of the map is injective.

We now study the tamely ramified part of the reciprocity map. Kummer duality and theorem \ref{tamepointthm} induce the isomorphism
\[\mathfrak{J}_{x}/(\Delta(K_{2}^{top}(F_{x}))\cap\mathfrak{J}_{x})\mathfrak{J}_{x}^{q-1} \to \text{Gal}(F_{x}^{ab}/F_{x})/(\text{Gal}(F_{x}^{ab,p}/F_{x})\text{Gal}(F_{x}^{unram}/F_{x}))\]
showing that this part of the map is injective with dense image also.

We complete this part of the proof by checking that the part of the reciprocity map related to the algebraic closure of $k(x)$ is injective with dense image. Since the image is $\mathbb{Z} \subset \hat{\mathbb{Z}}$, the density of the image is clear.

To show this part of the map is injective, we use the commutative diagram:
\[\begin{CD}
\frac{\mathcal{J}_{x}}{\Delta(K_{2}^{top}(F_{x}))} @>\delta>> \frac{\oplus_{y \ni x}k(y)_{x}^{\times}}{\delta(\Delta(K_{2}^{top}(F_{x})))} @>>> 0\\
@V\phi_{x}VV @V\phi_{k(x)}VV @.\\
\text{Gal}(F_{x}^{ab}/F_{x}) @>>> \text{Gal}(k(x)^{ab}/k(x)) @>>> 0.\\
\end{CD}
\]
Both the rows are exact, and the kernel of the first map on the top row is $\prod_{y \ni x}K_{2}^{top}(\mathcal{O}_{x,y})$ which is the part of the group related via the reciprocity map to the kernel of the first map on the bottom row, by definition of the Galois groups. Hence this part of the reciprocity map is injective also, and this part of the proof is complete.

Finally, to prove exact sequence $1$, we consider the commutative diagram with exact lower row:
\[
\begin{CD}
\frac{{\prod}'_{y'\ni x}K_{2}^{top}(L_{y'})}{\Delta(K_{2}^{top}(L))\cap {\prod}'_{y'\ni x'}K_{2}^{top}(L_{y'})} @>N>> \frac{{\prod}'_{y \ni x}K_{2}^{top}(F_{x,y})} {\Delta(K_{2}^{top}(F_{x}))\cap {\prod}'_{y \ni x}K_{2}^{top}(F_{x,y})} @>\psi_{x}>> \text{Gal}(L/F_{x}) @>>> 0 \\
@V\phi_{L}VV  @V\psi_{y}VV  @||| @. \\
\text{Gal}(L^{ab}/L) @>>> \text{Gal}(F_{x}^{ab}/F_{x}) @>>> \text{Gal}(L/F_{x}) @>>> 0  \end{CD}
\]

where $N$ is the product of the local norm maps. The commutivity follows from property two of this theorem and Galois theory. Now as in the corresponding theorem in the previous section, the density of the images of the first two vertical maps and the fact that the image of $N$ is closed complete the proof that the  top sequence is exact.
\end{proof}

\appendix
\section{Calculations in Milnor $K$-groups}\subsection{}
This appendix will give details of various calculations in Milnor $K$-groups necessary throughout the text.
\begin{lemma}\label{Kgpcalc1}
With $\delta_{x}$ as in lemma \ref{mainthmcurves}, we have the identity:
\[\{1 + \delta_{x}^{p}\tc^{k}, \tc\} \equiv   \{1 + \delta_{x}{^p}\tc^{k}, \delta_{x}\}^{p} \text{ mod } J_{\geq k + 1}.\]
\end{lemma}
\begin{proof}
Using the basic identity given in lemma 3.3.1, we have:
\[\{1+ \delta_{x}^{p}\tc^{k}, \tc\} = \{1+ \delta_{x}^{p}\tc^{k}, \tc\} + \{1 + \delta_{x}^{p}\tc^{k}, -\delta_{x}^{p}\tc^{k}\}\]
which in turn is equal to
\[\{1+ \delta_{x}^{p}\tc^{k}, \tc\} + \{1 + \delta_{x}^{p} \tc^{k}, \delta_{x}\}^{p} + \{1 + \delta_{x}^{p}\tc^{k}, \tc\}^{k}\]
\[ = \{(1+ \delta_{x}^{p}\tc^{k})(1 + \delta_{x}^{p}\tc^{k})^{k}, \tc\} + \{1 + \delta_{x}^{p}\tc^{k}, \delta_{x}\}^{p} \equiv \{1 + \delta_{x}^{p}\tc^{k}, \delta_{x}\}^{p} \text{ mod } J_{\geq k + 1}\]
as required. \end{proof}

\begin{lemma}\label{Kgpcalc2}
For $u$ and $t$ primes of a complete two-dimensional local ring $\hat{\mathcal{O}}_{X,x}$, $v \in k(x)$, and integers $i$ and $j$, we have the following identity:

\[
\{1 - ivt^{j}u^{i}, u\} \equiv \{1 + jvu^{i}t^{j}, t\} \text{ mod } K_{2}^{top}(\mathcal{O}_{t,u}, \mathfrak{p}_{t,u}^{j+1}).\]

\end{lemma}

\begin{proof}

We have:
\[ 1 = \{1+ivu^{i}t^{j}, -ivu^{i}t^{j}\} = \{1+vu^{i}t^{j},u\}^{i}\{1+ vu^{i}t^{j},t\}^{j}.\]
Now expanding these powers we get

\[(1+vu^{i}t^{j})^{i} = 1 + ivu^{i}t^{j} + \binom{i}{2}v^{2}u^{2i}t^{2j} + \dots = (1+ivu^{i}t^{j})(1 + \binom{i}{2}v^{2}u^{2i}t^{2j} + \dots )\]

\noindent so that
\[\{(1+vu^{i}t^{j})^{i}, u\} \equiv \{1 + ivu^{i}t^{j}, u\} \text{ mod } K_{2}^{top}(\mathcal{O}_{t,u}, \mathfrak{p}_{t,u}^{j+1}).\]

\noindent Symmetrically we get 

\[\{(1+vu^{i}t^{j})^{j}, t\} \equiv \{1 + jvu^{i}t^{j}, t\} \text{ mod } K_{2}^{top}(\mathcal{O}_{t,u}, \mathfrak{p}_{t,u}^{j+1}).\]

\noindent So from our first equation, we have

\[ 1 \equiv \{1+ivu^{i}t^{j}, u\} \{1+jvu^{i}t^{j}, t\}.\]

\noindent Now multiplying both sides by $\{1+ivu^{i}t^{j},u\}^{-1}$ and performing a similar calculation to those above gives the result.

\end{proof}

\begin{lemma}\label{Kgpcalc3}
Let $f, g \in k(x)$.  We have:
\[\{1 + fu^{i}t^{l}, 1+gu^{j}\} \equiv \left\{1 + fu^{i}\frac{jgu^{j}}{1 + gu^{j}}t^{l}, u\right\} \text{ mod } K_{2}^{top}(\mathcal{O}_{x,y}, \mathfrak{m}_{x,y}^{l+1}).\]
\end{lemma}

\begin{proof}
We have:
\[\{1+fu^{i}t^{l},1+gu^{j}\} = \{(1+fu^{i}t^{l}(1+gu^{j}))(1+fu^{i}t^{l})^{-1},1+gu^{j}\}^{-1}\]
\[\ \ \ \ \ \ \ \ \ \ \ \ \ \ \ \ \ \ \ \ \ \ \ \ \ \ \times\{1+fu^{i}t^{l}(1+gu^{j}),1+gu^{j}\}\]
which is equal to
\[\{(1+fu^{i}t^{l}(1+gu^{j}))(1+fu^{i}t^{l})^{-1},1+gu^{j}\}^{-1}\{1+fu^{i}t^{l}(1+gu^{j}), -fu^{i}t^{l}\}^{-1}\]
by the definition of $K$-groups.
This last expression is equal to
\[\{(1+fu^{i}t^{l}(1+gu^{j}))(1+fu^{i}t^{l})^{-1},1+gu^{j}\}^{-1}\]
\[\ \ \ \ \ \ \ \ \ \ \ \ \ \ \ \ \ \ \ \ \ \ \ \ \ \ \ \ \times\{(1+fu^{i}t^{l}(1+gu^{j}))^{-i},u\}\{(1+fu^{i}t^{l}(1+gu^{j}))^{-l},t\},\]
where the first term is trivial modulo $K_{2}^{top}(\mathcal{O}_{x,y}, \mathfrak{m}_{x,y}^{l+1})$.

So we are left with 
\[\{(1+fu^{i}t^{l}(1+gu^{j}))^{-i},u\}\{(1+fu^{i}t^{l}(1+gu^{j}))^{-l},t\}.\]
Calculating as in lemma \ref{Kgpcalc2} above, we have
\[\{1+fu^{i}t^{l}(1+gu^{j}),t\}^{-l} \equiv \{1-fu^{i}t^{l}(1+gu^{j}),u\}^{i}\{1-fu^{i}t^{l}(1+gu^{j}),1+gu^{j}\}\] so the above expression becomes

\[\{1-fu^{i}t^{l}(1+gu^{j}),1+gu^{j}\}.
\]
So reversing our argument, we see the original expression is also equivalent to
\[\{1-\frac{fu^{i}t^{l}}{1+gu^{j}},1+gu^{j}\}.\]

Repeating the argument for this new element and applying the lemma above gives us the identity.
\end{proof}

\bibliographystyle{plain}
\bibliography{References}

\end{document}